\documentclass[11pt]{amsart}
\usepackage{geometry}
\usepackage{float}
\usepackage{caption}
\usepackage{subcaption}
\usepackage[round]{natbib}

\usepackage{amsthm, amsmath, amssymb, mathtools, mathrsfs, enumerate, caption}
\usepackage{tikz, tikz-cd}
\usepackage{hyperref}
\usepackage{bm}
\usepackage{mathrsfs}
\numberwithin{equation}{section}
\newtheorem{thm}{Theorem}[section]
\newtheorem{defn}[thm]{Definition}
\newtheorem{lem}[thm]{Lemma}
\newtheorem{prop}[thm]{Proposition}
\newtheorem{cor}[thm]{Corollary}
\theoremstyle{remark}
\newtheorem{rmk}{Remark}
\newcommand{\Lm}{\Lambda}
\newcommand{\Fl}{\mathcal{F}\ell}
\newcommand{\D}{\mathbb{D}}
\DeclareMathOperator{\Gl}{GL}
\DeclareMathOperator{\Sl}{SL}
\DeclareMathOperator{\Sp}{Sp}
\DeclareMathOperator{\GSp}{GSp}
\DeclareMathOperator{\PSL}{PSL}
\DeclareMathOperator{\PSp}{PSp}

\newcommand{\lm}{\lambda}
\DeclareMathOperator{\Bd}{Bound}
\DeclareMathOperator{\Gr}{Gr}
\newcommand{\G}{\mathbb{G}}

\newcommand{\Prj}{\text{Project}}
\newcommand{\mf}{\mathbf}
\newcommand{\reals}{\mathbb{R}}
\newcommand{\naturals}{\mathbb{N}}
\newcommand{\complexes}{\mathbb{C}}
\newcommand{\integers}{\mathbb{Z}}

\newcommand{\Q}{\mathcal{Q}}

\usepackage{graphicx}

\newcommand{\Le}{\raisebox{\depth}{\rotatebox{180}{$\Gamma$}}}

\newcommand{\bdot}[2]{
\draw [black, fill = black] (#1, #2) circle [radius = 0.1];
}

\newcommand{\wdot}[2]{
\draw [black, fill = white] (#1, #2) circle [radius = 0.1];
}

\newcommand{\edge}[3]{
\draw (#1,#2) -- (#1, #2 + #3);
\bdot{#1}{#2}
\wdot{#1}{#2 + #3}
}

\author{Rachel Karpman}
\address{
Department of Mathematics,
University of Michigan, Ann Arbor,
530 Church St.,
Ann Arbor, MI 48109-1043
USA}
\email{rkarpman@umich.edu}
\title{Total Positivity for the Lagrangian Grassmannian}

\begin{document}

\begin{abstract}
The stratification of the Grassmannian by positroid varieties has been the subject of extensive research.  Positroid varieties are in bijection with a number of combinatorial objects, including $k$-Bruhat intervals and bounded affine permutations.  In addition, Postnikov's \emph{boundary measurement map} gives a family of parametrizations of each positroid variety; the domain of each parametrization is the space of edge weights of a weighted planar network.
In this paper, we generalize the combinatorics of positroid varieties to the Lagrangian Grassmannian $\Lm(2n)$, which is the type $C$ analog of the ordinary, or type $A$, Grassmannian.  The Lagrangian Grassmannian has a stratification by projected Richardson varieties, which are the type $C$ analogs of positroid varieties.  We define type $C$ generalizations of bounded affine permutations and $k$-Bruhat intervals, as well as several other combinatorial posets which index positroid varieties.  In addition, we generalize Postnikov's network parametrizations to projected Richardson varieties in $\Lm(2n)$.  In particular, we show that restricting the edge weights of our networks to $\mathbb{R}^+$ yields a family of parametrizations for totally nonnegative cells in $\Lm(2n)$.  In the process, we obtain a set of linear relations among the Pl\"{u}cker coordinates on $\Gr(n,2n)$ which cut out the Lagrangian Grassmannian set-theoretically.
\end{abstract}

\thanks{This research was supported in part by NSF grant DGE-1256260 and NSF grant DMS-0943832.}

\maketitle{}
\tableofcontents{}

\section{Introduction}

Lusztig defined the \emph{totally nonnegative part} of an abstract flag manifold $G/P$ and conjectured that it was made up of topological cells, a conjecture proved by Rietsch in the late 1990's \citep{Lus94, Lus98, Rie99}.  Postnikov later introduced the \emph{positroid stratification} of the totally nonnegative Grassmannian $\Gr_{\geq 0}(k,n)$, and showed that his stratification was a special case of Lusztig's \citep{Pos06}.  While Lusztig's approach relied on the machinery of canonical bases, Postnikov's was more elementary.  Each \emph{positroid cell} in Postnikov's stratification was defined as the locus in $\Gr_{\geq 0}(k,n)$ where certain Pl\"{u}cker coordinates vanish.  

The positroid stratification of $\Gr_{\geq 0}(k,n)$ extends to a stratification of the complex Grassmannian $\Gr(k,n)$ of $k$-planes in $n$-space.   That is, we can decompose $\Gr(k,n)$ into \emph{positroid varieties} $\Pi$ which are the Zariski closures of Postnikov's totally nonnegative cells.  Remarkably, these positroid varieties are the images of \emph{Richardson varieties} in $\mathcal{F}\ell(n)$ under the natural projection 
\begin{displaymath} \pi_k:\mathcal{F}\ell(n) \rightarrow \Gr(k,n).\end{displaymath}
For each positroid variety $\Pi$, there is a family of Richardson varieties which project birationally to $\Pi$, and this family has an elementary combinatorial description \citep{KLS13}.

The stratification of $\Gr(k,n)$ by projected Richardson varieties was first studied by Lusztig \citep{Lus98}.  Brown, Goodearl and Yakimov investigated the same stratification from the viewpoint of Poisson geometry \citep{BGY06}.  Finally, Knutson, Lam and Speyer showed that Lusztig's strata were in fact the Zariski closures of Postnikov's totally nonnegative cells \citep{KLS13}.  

There are a number of combinatorial objects which index positroid varieties, including \emph{bounded affine permutations} and \emph{$k$-Bruhat intervals}.  In addition, there is a remarkable combinatorial construction, due to Postnikov, which yields a family of parametrizations of each positroid variety.

Postnikov's original construction gives a family of maps onto each positroid cell in $\Gr_{\geq 0}(k,n)$.  The domain of each map is the space of positive real edge weights of a weighted planar network, called a \emph{plabic graph}, and there is a class of such networks for each positroid cell \citep{Pos06}.  Let $G$ be a planar network corresponding to a positroid cell 
\begin{math}\mathring{\Pi}_{\geq 0}\end{math} 
of dimension $d$ in $\Gr_{\geq 0}(k,n)$.  Specializing all but an appropriately chosen set of $d$ edge weights to $1$ yields a homeomorphism 
\begin{math}(\reals^+)^d \rightarrow \mathring{\Pi}_{\geq 0}\end{math} 
which we call a \emph{parametrization} of 
\begin{math}\mathring{\Pi}_{\geq 0}.\end{math} If we let the edge weights range over $\complexes^{\times}$ instead of $\reals^+,$ we obtain a well-defined homeomorphism onto a dense subset of the positroid variety $\Pi$ in $\Gr(k,n)$ corresponding to the totally nonnegative cell 
\begin{math}\mathring{\Pi}_{\geq 0}\end{math}
\citep{MS14}.  We call these maps \emph{parametrizations} as well.

While the combinatorial theory of the positroid stratification of $\Gr(k,n)$  is particularly nice, the projected Richardson stratification for general $G/P$ is also of great interest.  From a combinatorial standpoint, the poset of projected Richardson varieties is a shellable ball \citep{Wil07, KLS13}.  In addition, projected Richardson varieties have nice geometric properties: they are normal, Cohen-Macaulay, and have rational singularities \citep{KLS13}.  

In this paper, we extend the combinatorial theory of positroid varieties to the \emph{Lagrangian Grassmannian} $\Lm(2n)$. While the ordinary Grassmannian is the moduli space of $k$-dimensional subspaces of an $n$-dimensional complex vector space, points in the Lagrangian Grassmannian correspond to maximal isotropic subspaces of $\mathbb{C}^{2n}$ with respect to a symplectic form.  Hence we may realize $\Lm(2n)$ as a subvariety of $\Gr(n,2n)$.  Alternatively, $\Lm(2n)$ is the quotient of the symplectic group $\Sp(2n)$ by a parabolic subgroup, and is thus a partial flag variety.  By Lusztig's general theory, $\Lm(2n)$ has a stratification by projected Richardson varieties, which will be our principal objects of study.

The poset $\Q(k,n)$ of $k$-Bruhat intervals which index positroid varieties has a natural analog for any partial flag variety.  We explicitly describe the corresponding poset $\Q^C(2n)$ which indexes projected Richardson varieties in $\Lm(2n)$, and relate $\Q^C(2n)$ to $\mathcal{Q}(n,2n)$.  
We then define the natural analogs of bounded affine permutations for the Lagrangian Grassmannian.  These turn out to be bounded affine permutations which satisfy a symmetry condition.  

Next, we construct network parametrizations of projected Richardson varieties in the Lagrangian Grassmannian.  The underlying graphs are plabic graphs which satisfy a symmetry condition, and we impose a corresponding symmetry condition on the edge weights: if $e$ and $e'$ are a symmetric pair of edges, then $e$ and $e'$ must have the same weight.  See Figure \ref{introfig} for an example.  We investigate the combinatorics of these networks, which neatly parallels the combinatorics of ordinary plabic graphs.  In particular, we show that symmetric plabic graphs with positive real edge weights parametrize totally nonnegative cells in $\Lm(2n)$.  Finally, we use the combinatorics of symmetric plabic graphs to define analogs for the Lagrangian Grassmannian of various posets which index positroid cells. 

\begin{figure}[H]
\centering
\begin{tikzpicture}[scale = 0.8]
\draw (0,0) circle (3);
\draw ({2*cos(60)},{2*sin(60)}) -- ({-2*cos(60)},{2*sin(60)})  -- ({-2*cos(60)},{-2*sin(60)})  -- ({2*cos(60)},{-2*sin(60)}) --  ({2*cos(60)},{2*sin(60)});
\draw [gray, thin] (0,3.5) -- (0,-3.5);
\draw ({2*cos(60)},{2*sin(60)}) -- ({3*cos(60)},{3*sin(60)});
\draw ({-2*cos(60)},{2*sin(60)}) -- ({-3*cos(60)},{3*sin(60)});
\draw ({-2*cos(60)},{-2*sin(60)}) -- ({-3*cos(60)},{-3*sin(60)});
\draw ({2*cos(60)},{-2*sin(60)}) -- ({3*cos(60)},{-3*sin(60)});
\draw ({3*cos(30)},{3*sin(30)}) -- (2,0) -- ({3*cos(30)},{-3*sin(30)});
\draw ({-3*cos(30)},{3*sin(30)}) -- (-2,0) -- ({-3*cos(30)},{-3*sin(30)});
\bdot{{2*cos(60)}}{{2*sin(60)}};\wdot{{-2*cos(60)}}{{2*sin(60)}};\bdot{{-2*cos(60)}}{{-2*sin(60)}};\wdot{{2*cos(60)}}{{-2*sin(60)}};
\wdot{{3*cos(60)}}{{3*sin(60)}};\bdot{{3*cos(30)}}{{3*sin(30)}};\wdot{2}{0};\bdot{{3*cos(30)}}{{-3*sin(30)}};\bdot{{3*cos(60)}}{{-3*sin(60)}};
\bdot{{-3*cos(60)}}{{3*sin(60)}};\wdot{{-3*cos(30)}}{{3*sin(30)}};\bdot{-2}{0};\wdot{{-3*cos(30)}}{{-3*sin(30)}};\wdot{{-3*cos(60)}}{{-3*sin(60)}};
\node [above right] at ({3*cos(60)},{3*sin(60)}) {1};
\node [above right] at ({3*cos(30)},{3*sin(30)}) {2};
\node [below right] at ({3*cos(-30)},{3*sin(-30)}) {3};
\node [below right] at ({3*cos(-60)},{3*sin(-60)}) {4};
\node [below left] at ({3*cos(240)},{3*sin(240)}) {5};
\node [below left] at ({3*cos(210)},{3*sin(210)}) {6};
\node [above left] at ({3*cos(150)},{3*sin(150)}) {7};
\node [above left] at ({3*cos(120)},{3*sin(120)}) {8};
\node [right] at (0,2) {4};
\node [right] at (0,-2) {3};
\node [right] at (1,0) {7};
\node [left] at (-1,0) {7};
\end{tikzpicture}
\caption{A symmetric weighting of a symmetric plabic graph.  All unlabeled edges have weight $1$.}
\label{introfig}
\end{figure}
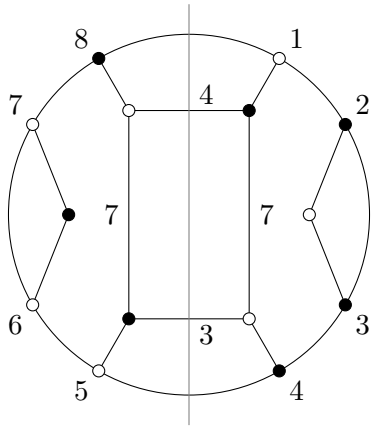

Much of the inspiration for this project came from \citep{LW07}, where the authors generalize Postnikov's \Le-diagrams to all \emph{cominiscule Grassmannians}.  \Le-diagrams are tableaux filled with $0$'s and $+$'s which satisfy a pattern-avoidance condition.  Each positroid variety in $\Gr(k,n)$ corresponds to a unique \Le-diagram, which in turn encodes a particular choice of planar network \citep{Pos06}.  Hence, generalizing \Le-diagrams is a step toward generalizing Postnikov's theory of planar networks.  Moreover, the \emph{type B decorated permutations} of Lam and Williams are in bijection with our bounded affine permutations for the Lagrangian Grassmannian.  Note that Lam and Williams' type B objects index projected Richardson varieties in both the odd orthogonal Grassmannian, a flag variety of type B, and the Lagrangian Grassmannian \citep{LW07}.  This follows from the fact that the Weyl groups of types B and C are isomorphic.  

The organization of the paper is as follows.  In Section \ref{back}, we provide background on flag varieties, Richardson varieties and projected Richardson varieties, including \emph{Deodhar decompositions} of flag varieties.  We also review the combinatorial theory of positroid varieties, most notably plabic graphs and the boundary measurement map.  In Section \ref{bound}, we realize $\mathcal{Q}^C(2n)$ as a subposet of $\mathcal{Q}(n,2n)$, and define bounded affine permutations for the Lagrangian Grassmannian.  We then use these results to relate the geometry of the projected Richardson stratification of $\Lm(2n)$ to the geometry of the positroid stratification.  

In Section \ref{Cbridge} we define the Lagrangian Grassmannian analogs of a special class of plabic graphs called \emph{bridge graphs}, and in Section \ref{Cbdry} we generalize this construction to give a theory of plabic graphs for the Lagrangian Grassmannian.   In the process, we obtain a set of linear relations among the Pl\"{u}cker coordinates on $\Gr(n,2n)$ which cut out $\Lm(2n)$ set-theoretically.  In Section \ref{positivity}, we relate our network parametrizations to total positivity in $\Lm(2n)$. Finally, in Section \ref{index}, we describe several more combinatorial indexing sets for projected Richardson varieties in $\Lm(2n)$.  Namely, we give type C analogs of \emph{Grassmann necklaces}, \emph{dual Grassmann necklaces}, and a class of matroids called \emph{positroids}.

\subsection*{Acknowledgements} I would like to thank Thomas Lam for his invaluable guidance throughout the course of this project. I am grateful to Greg Muller and David E Speyer, for sharing a then-unpublished manuscript of \citep{MS14}.  Thanks also to Yi Su for insightful discussions of the combinatorics of symmetric plabic graphs.  Finally, thanks to Jake Levinson and Gabriel Frieden for careful reading of this preprint, and many helpful suggestions.

\section{Background}
\label{back}

\subsection{Notation for partitions, root systems and Weyl groups}

For $a \in \mathbb{N}$, write $[a]$ for the set 
\begin{math}\{1,2,\ldots,a\} \subseteq \naturals,\end{math}
and let $[a,b]$ denote the set 
\begin{math}\{a,a+1,\ldots,b\}\end{math}
for $a \leq b$.  For $a > b$ we set 
\begin{math}[a,b] = \emptyset\end{math}.
 Let 
 \begin{math} I,J \subseteq \naturals\end{math} with 
 \begin{align}
 I &= \{i_1 < i_2 < \ldots < i_m\}\\
J &= \{j_1 < j_2 < \ldots < j_m\}.
 \end{align}
 We say $I \leq J$ if $i_r \leq j_r$ for all $1 \leq r \leq m$.  We denote the set of all $k$-element subsets of $[n]$ by ${{[n]}\choose k}$.

Let $\Phi$ denote a finite root system with simple roots $\{\alpha_i \mid 1 \leq i \leq m\}$ for some $m \in \mathbb{N}$.  Let $(W,S)$ denote the Weyl group of $\Phi$, with simple reflections $S = \{s_i \mid 1 \leq i \leq m\}$ corresponding to the $\alpha_i$.
Let $\ell$ denote the standard length function on $W$.  

For \begin{math} u,w \in W\end{math}, we write $u \leq w$ to denote a relation in the (strong) Bruhat order. A factorization 
\begin{math} u=vw \in W\end{math}
 is \emph{length additive} if 
\begin{equation} \ell(u) = \ell(v)+\ell(w). \end{equation}
We use 
\begin{math} u \leq_{(r)} w\end{math}
 to denote a relation in the \emph{right weak order}, so 
\begin{math} u \leq_{(r)} w\end{math}
if there exists $v \in W$ such that $uv= w$ and the factorization is length additive.  Similarly, we write 
\begin{math} u \leq_{(l)} w\end{math} 
to denote \emph{left weak order}, and say 
\begin{math} u \leq_{(l)} w\end{math} if there is a length-additive factorization $vu = w$.  All functions and permutations act on the left, so 
\begin{math} \sigma \rho \end{math}
means ``first apply $\rho$, then apply $\sigma$ to the result.''

We fix notation for the root systems of types $A_{n-1}$ and $C_n$ respectively.  Let $\{\epsilon_1,\ldots,\epsilon_n\}$ denote the standard basis of $\mathbb{R}^n$.  We realize the root system $A_{n-1}$ as a subset of the vector space
\[V \coloneqq \left\{ (\lm_1,\ldots,\lm_n) \in \mathbb{R}^n \left| \;\sum_{i=1}^n \lm_i = 0 \right.\right\}\]
The roots are given by 
\[\Phi = \{\epsilon_i-\epsilon_j \mid 1\leq i,j \neq n\}\]
and the simple roots are given by 
\[\{\alpha_i = \epsilon_i - \epsilon_{i+1} \mid 1 \leq i \leq n-1\}.\]
The Weyl group of type $A_{n-1}$ is the symmetric group $S_n$ on $n$ letters, acting by permutations on the standard basis.  
The simple reflection $s_i^A$ corresponding to $\alpha_i$ is the transposition $(i,i+1)$.  Let $(S_n)_k$ denote the parabolic subgroup of $S_n$ corresponding to $\{\alpha_i \mid i \neq k\}$.  Then $(S_n)_k$ is the Young subgroup $S_k \times S_{n-k}$
consisting of permutations which fix the sets $[k]$ and $[k+1,n]$.  
 
Similarly, we realize the root system of type $C_n$ as a subset of $\mathbb{R}^{n}$.  The roots are given by 
\[\{\pm \epsilon_i \pm \epsilon_j \mid 1 \leq i,j \neq n\} \cup \{2\epsilon_i \mid 1 \leq i \leq n\}\]
and the simple roots by 
\[\alpha_i =
\begin{cases} \epsilon_i - \epsilon_{i+1} & 1 \leq i \leq n-1\\
2\epsilon_n & i=n
\end{cases}
\]
The Weyl group  of type $C_n$ consists of permutations $\sigma$ of $\{\pm i \mid 1 \leq i \leq n\}$ which satisfy $\sigma(-i) = -\sigma(i)$.  
Permutations act on the standard basis, where we set $\epsilon_{-i} \coloneqq -\epsilon_i.$

Alternatively, we may realize the Weyl group of type $C_n$ as the subgroup $S_n^C$ of $S_{2n}$ consisting of permutations $\tau \in S_{2n}$ which satisfy
\begin{equation}\label{signed} \tau(2n+1-a) = 2n+1-\tau(a). \end{equation}
The simple generators $s^C_1,\ldots,s^C_n$ are given by
\[s^C_i = 
\begin{cases}s_i^As^A_{2n-i} & 1 \leq i \leq n-1\\
s_i^A & i = n
\end{cases}
\]
and the map $S_n^C \hookrightarrow S_{2n}$ is a Bruhat embedding \citep{BB05}.  With these conventions, the parabolic subgroup $(S_n^C)_n$ corresponding to $\{\alpha_i \mid i \neq n\}$ is the subgroup of permutations in $S_n \times S_n$ which satisfy \eqref{signed}.  
We denote the length functions on $S_n$ and $S_n^C$ by $\ell^A$ and $\ell^C$ respectively.

We number the rows of all matrices from top to bottom, and the columns from left to right.  When specifying matrix entries, position $(i,j)$ denotes row $i$ and column $j$.  For $w \in S_n$ or $S_n^C$, we let $w([a])$ denote the unordered set $\{w(1),w(2),\ldots,w(a)\}.$  

Fix $n \in \mathbb{N}$.  For $a \in [2n]$, we write $a'$ to denote $2n+1-a$.  Let $I \in {{[2n]}\choose{n}}$.  Then we define $R(I) \in {{[2n]}\choose{n}}$ by setting $R(I) = [2n] \backslash \{a' \mid a \in I\}.$

\subsection{Flag varieties, Schubert varieties and Richardson varietes}

We recall some basic facts about flag varieties, as stated for example in \citep{KLS14}.  Let $G$ be a semisimple Lie group over $\mathbb{C}$, and let $B_+$ and $B_-$ be a pair of opposite Borel subgroups of $G$. Then $G/B_+$ is a flag variety, and there is a set-wise inclusion of the Weyl group $W$ of $G$ into $G/B_+$.  The flag variety $G/B_+$ has a stratification by Schubert cells, indexed by elements of $W$.  For $w \in W$, the \emph{Schubert cell} $\mathring{X}_w$ is $B_-wB_+/B_+$ while the \emph{Schubert variety} $X_w$ is the closure of $\mathring{X}_w$. The \emph{opposite Schubert cell} $\mathring{Y}_w$ is $B_+wB_+/B_+$ while the \emph{opposite Schubert variety} $Y_w$ is the closure of $\mathring{Y}_w$.   The cells $\mathring{X}_w$ and $\mathring{Y}_w$ are both isomorphic to affine spaces.  They have codimension and dimension $\ell(w)$, respectively. Both Schubert and opposite Schubert cells give stratifications of $G/B_+$. 

The \emph{Richardson variety} $\mathring{R}_{u,w}$ is the transverse intersection $\mathring{R}_{u,w} = \mathring{X}_u \cap \mathring{Y}_w.$
This is empty unless $u \leq w$, in which case it has dimension $\ell(w) - \ell(u)$.  
The closure of $\mathring{R}_{u,w}$ is the \emph{Richardson variety} $R_{u,w}$.  Open Richardson varieties form a stratification of $G/B_+$ which refines the Schubert and opposite Schubert stratifications.

Let $P$ be a parabolic subgroup of $G$ containing $B_+$ and let $W_P$ be the Weyl group of $P$.  Then $G/P$ is a partial flag variety, and there is a natural projection $\pi_P:G/B_+ \rightarrow G/P$.  We are interested in the images of Richardson varieties under this projection map.

In this paper, we focus on flag varieties of types $A$ and $C$.  We use the superscripts $A$ and $C$ to indicate subvarieties of flag varieties of types $A$ and $C$, respectively. For example, the Schubert cell in the type $A$ flag variety corresponding to a Weyl group element $w$ is denoted $\mathring{X}_w^A$.  

\subsubsection{Type $A$}

Let $G = \Sl(n)$, a semisimple Lie group of type $A$.  Let $B_+$ be the subgroup of upper-triangular matrices, and let $B_-$ be the subgroup of lower-triangular matrices.  Let $\mathcal{F}\ell(n)$ be the quotient of $\Sl(n)$ by the right action of $B_+$.
Then $\mathcal{F}\ell(n)$ is an algebraic variety whose points correspond to flags
\begin{equation}V_{\bullet} = \{0 \subset V_1\subset V_2 \subset \cdots \subset V_n = \complexes^n\}\end{equation}
where $V_i$ is a subspace of $\complexes^n$ of dimension $i$.  
Hence an $n \times n$ matrix $M \in \Sl(n)$ represents the flag whose $i^{th}$ subspace is the span of the first $i$ columns of $M$. 

The Weyl group of type $A$ is the symmetric group $S_n$ on $n$ letters.  We may represent a permutation $w$ in $\mathcal{F}\ell(n)$ by any matrix $\overline{w}$ with $\pm 1$'s in positions $(w(i),i)$, and $0$'s everywhere else, where the number of $-1$'s is odd or even, depending on the parity of $w$.  These signs are necessarily to ensure that the matrix representative is contained in $\Sl(n)$.
Then the Schubert cell $\mathring{X}_{w}^A$ corresponding to $w$ is given by $B_{-}\overline{w}B_+,$ while the opposite Schubert cells $\mathring{Y}_w^A$ is $B_{+}\bar{w}B_+$. 

We have concrete descriptions of the Schubert and opposite Schubert varieties in $\mathcal{F}\ell(n)$, as stated in \citep[Section 4]{KLS13}.  For a subset $J$ of $[n]$, let 
\begin{math} \Prj_J:\complexes^n \rightarrow \complexes^J \end{math} be projection onto the coordinates indexed by $J$.  For a permutation $w \in S_n$, the Schubert cell corresponding to $w$ is given by
\begin{equation}\mathring{X}^A_{w}=\{ V_{\bullet} \in \mathcal{F}\ell(n) \mid \dim(\Prj_{[j]}(V_i)) = |w([i]) \cap [j]| \text{ for all } i\}\end{equation}
The closure of $\mathring{X}^A_w$ is the Schubert variety
\begin{equation}X_{w}^A=\{ V_{\bullet} \in \mathcal{F}\ell(n) \mid \dim(\Prj_{[j]}(V_i)) \leq |w([i]) \cap [j]| \text{ for all } i\}\end{equation}

Similarly, the opposite Schubert cell and opposite Schubert variety respectively are given by
\begin{equation}\mathring{Y}^A_w=\{ V_{\bullet} \in \mathcal{F}\ell(n) \mid \dim(\Prj_{[n-j+1,n]}(V_i)) = |w([i]) \cap [n-j+1,n]| \text{ for all } i\}\end{equation}
\begin{equation}Y^A_w=\{ V_{\bullet} \in \mathcal{F}\ell(n) \mid \dim(\Prj_{[n-j+1,n]}(V_i)) \leq |w([i]) \cap [n-j+1,n]| \text{ for all } i\}\end{equation}

Let $P$ denote the parabolic subgroup of $\Sl(n)$ consisting of block-diagonal matrices of the form
\[\begin{bmatrix}
C & D \\
0 & E
\end{bmatrix}\]
where the block $C$ is a $k \times k$ square, and $E$ is an $(n-k) \times (n-k)$ square.  Then $\Sl(n)/P$ is the Grassmannian $\Gr(k,n)$ of $k$-dimensional linear subspaces of the vector space $\complexes^n$.  

We may realize $\Gr(k,n)$ as the space of full-rank $k \times n$ matrices modulo the left action of $\text{GL}(k)$, the group of invertible $k\times k$ matrices; a matrix $M$ represents the space spanned by its rows.  
The natural projection
\begin{math} \pi_k:\mathcal{F}\ell(n) \rightarrow \Gr(k,n),\end{math}
carries a flag $V_{\bullet}$ to the $k$-plane $V_k$.  
If $V$ is a representative matrix for $V_{\bullet}$ then transposing the first $k$ columns of $V$ gives a representative matrix $M$ for $V_k$.  

The \emph{Pl{\"u}cker embedding}, which we denote $p$, maps $\Gr(k,n)$ into the projective space 
\begin{math}\mathbb{P}^{ {n\choose k}-1}\end{math}
with homogeneous coordinates $x_J$ indexed by the elements of 
\begin{math}{{[n]}\choose k}.\end{math}
For 
\begin{math} J \in {[n] \choose k}\end{math}
let $\Delta_J$ denote the minor with columns indexed by $J$.  Let $V$  be an $k$-dimensional subspace of $\complexes^n$ with representative matrix $M$.  Then $p(V)$ is the point defined by 
\begin{math} x_J = \Delta_J(M)\end{math}.
This map embeds $\Gr(k,n)$ as a smooth projective variety in 
\begin{math} \mathbb{P}^{ {n\choose k}-1}.\end{math}
The homogeneous coordinates $\Delta_J$ are known as \emph{Pl{\"u}cker  coordinates} on $\Gr(k,n)$.  The \emph{totally nonnegative Grassmannian}, denoted $\Gr_{\geq 0}(k,n),$ is the subset of $\Gr(k,n)$ whose Pl{\"u}cker coordinates are all nonnegative real numbers, up to multiplication by a common scalar. 

\subsubsection{Type $C$}

We now outline the same story in type $C$.  Our discussion follows \citep[Chapter 3]{BL00}.  However, we use the bilinear form given in \citep{Ma11}.  This choice yields a \emph{pinning} of $\Sp(2n)$, defined below, which is compatible with the \emph{standard pinning} for $\Sl(n)$.  We will use this fact frequently in what follows.

Let $V$ be the complex vector space $\mathbb{C}^{2n}$ with standard basis $e_1,\ldots,e_{2n}$.  Let $\langle \cdot, \cdot \rangle$ denote the non-degenerate, skew-symmetric form defined by
\[\langle e_i,e_j \rangle = \begin{cases}
(-1)^{j} &\text{ if } j = 2n+1-i\\
0 & \text{ otherwise}
\end{cases}\]
Let $E$ be the matrix of this bilinear form.  For example, for $n=2$, we have
\[E = \begin{bmatrix}
0 & 0 & 0 & 1 \\
0 & 0 & -1 & 0 \\
0 & 1 & 0 & 0\\
-1 & 0 & 0 & 0
\end{bmatrix}.\]
A subspace $U \subseteq V$ is \emph{isotropic} if $\langle u,v \rangle = 0$ for all $u,v \in V$.

The \emph{symplectic group} $\Sp(2n)$ is the group of matrices $A \in \Sl(2n)$ which leave the form $\langle \cdot, \cdot \rangle$ invariant.  Alternatively, define a map $\sigma:\Sl(2n) \rightarrow \Sl(2n)$ by 
\[\sigma(A) = E(A^t)^{-1}E^{-1}.\]
Then $\Sp(2n)$ is the group of all fixed points of $\sigma$.  It is a semi-simple Lie group of type $C_n$.

Let $B_+$, $B_-$ and $P$ be the subgroups of $\Sl(2n)$ given above, where $k = n$.
The Borel subgroups $B_+$ and $B_{-}$ are both stable under $\sigma$, and so is the parabolic $P$.  Let $B_+^{\sigma}$, $B_{-}^{\sigma},$ and $P^{\sigma}$ denote the intersection of $B_+$, $B_{-}$ and $P$ respectively with $\Sp(2n)$.  Then $B_+^{\sigma}$ and $B_{-}^{\sigma}$ are a pair of opposite Borel subgroups of $\Sp(2n)$, while $P^{\sigma}$ is a parabolic subgroup of $\Sp(n)$.  

The generalized flag variety $\Sp(2n)/B_{+}^{\sigma}$ embeds in $\Fl(2n)$ in the obvious way, and $\Sp(2n)/P^{\sigma}$ embeds in $\Gr(n,2n)$.  The image of $\Sp(2n)/P^{\sigma}$ is precisely the subset of $\Gr(n,2n)$ corresponding to maximal isotropic subspaces; that is, the Lagrangian Grassmannian $\Lm(2n)$.  We have a commutative diagram of inclusions and projections, shown in Figure \ref{diag}.  

\begin{figure}[h]
\centering
\begin{tikzcd}
\Sl(2n) \arrow[two heads]{r} &  \Fl(2n) \arrow[two heads]{r} & \Gr(n,2n)\\
\Sp(2n) \arrow[hook]{u} \arrow[two heads]{r} &  \Sp(2n)/B^{\sigma}_+ \arrow[hook]{u} \arrow[two heads]{r}& \Lm(2n) \arrow[hook]{u}\\
\end{tikzcd}
\caption{Realizing $\Lm(2n)$ as a submanifold of $\Gr(n,2n)$.}
\label{diag}
\end{figure}

For each permutation $w \in S^C_n$, 
let $\overline{w}$ denote the matrix with $(-1)^{\rho_i}$'s at positions $(w(i),i)$ and $0$'s everywhere else, where
\[\rho_i = 
\begin{cases}
1 & \text{if }1 \leq i \leq n \text{ and $w(i)$ and $i$ have opposite parity}\\
0 & \text{otherwise}
\end{cases}
.\]

Then the Schubert cell $\mathring{X}_{w}^C$ corresponding to $w$ is given by $B_{-}^{\sigma}\overline{w}B_+^{\sigma},$ while the opposite Schubert cell $\mathring{Y}_w^C$ is $B_{+}^{\sigma}\overline{w}B_+^{\sigma}$. It is straightforward to check that the following set-theoretic identities hold under the embedding $\Sp(2n)/B_+^{\sigma} \hookrightarrow \mathcal{F}\ell(2n)$ given above:
\begin{align*} \mathring{X}_w^C &= \mathring{X}^A_w \cap (\Sp(2n)/B_{+}^{\sigma}) \\
\mathring{Y}_w^C&= \mathring{Y}^A_w \cap (\Sp(2n)/B_{+}^{\sigma})\\
\mathring{R}_{u,w}^C &= \mathring{R}_{u,w}^A \cap (\Sp(2n)/B_{+}^{\sigma}).
\end{align*}

\subsection{Projected Richardson varieties and $P$-order}

We now return to the case of general $G/P$ with all notation defined as above.  Let $u,w \in W$.  The variety $\Pi_{u,w} \coloneqq \pi_P(R_{u,w})$ is a \emph{projected Richardson variety}.  Note that we do not have a one-to-one correspondence between Richardson varieties and projected Richardson varieties.  Rather, for each $u \leq w$, there is a family of Richardson varieties $R_{u',w'}$ such that $\Pi_{u',w'} = \Pi_{u,w}$.

Projected Richardson varieties have a number of nice geometric properties.  In particular, they are normal, Cohen-Macaulay, and have rational singularities \citep{KLS13}.  When $G/B = \mathcal{F}\ell(n)$ and $G/P = \Gr(k,n)$, projected Richardson varieties coincide with \emph{positroid varieties,} which have been studied extensively. 

We review the notion of $P$-Bruhat order, as defined in \citep{KLS14}.  
We say $u \lessdot_P w$ if $u \lessdot w$ in the usual Bruhat order and $uW_P \neq wW_P$.  The \emph{$P$-Bruhat} order on $W$ is the transitive closure of these relations.  Note that, as the notation suggests, the relations $\lessdot_P$ are the covering relations in the resulting partial order.  When $W = S_n$ and $W_P = S_k \times S_{n-k}$, we obtain the $k$-order introduced in \citep{BS98}.
We use the symbol $\leq_P$ to denote $P$-Bruhat order.  If $u \leq_P w,$
we write $[u,w]_P$ for the $P$-Bruhat interval 
\begin{math}\{v \mid u \leq_P v \leq_P w\}.\end{math}  

Every (closed) projected Richardson variety may be realized as $\Pi_{u,w}$ with $u \leq_P w$. The projection map $\pi_P$ is an isomorphism when restricted to $\mathring{R}_{u,w}$ if and only if $u \leq_P w$. In this case, we define the \emph{open projected Richardson variety} $\mathring{\Pi}_{u,w}$ to be $\pi_P(\mathring R_{u,w})$.  The variety $\mathring{\Pi}_{u,w}$ is irreducible, of dimension $\ell(w)-\ell(u)$ \citep{KLS14}.  The projected Richardson variety $\Pi_{u,w}$ is the closure of $\mathring{\Pi}_{u,w}$, and $\pi_P$ maps $R_{u,w}$ birationally to $\Pi_{u,w}$ \citep{KLS13}.

We now review some combinatorial results about $P$-Bruhat order.

\begin{lem}{\citep{KLS14}} Suppose we have $u \leq v$ and $x \leq y$ in $W$, and suppose there is some $z \in W_P$ such that $x = uz$, $y = vz$ with both factorizations length-additive.  Then $u \leq_P v$ if and only if $x \leq_P y$. 
\end{lem}

Hence there is an equivalence relation on $P$-Bruhat intervals, which is generated by setting 
\begin{math}[u,w]_P \sim [x,y]_P\end{math}
 if there is some $z \in W_P$
 such that $x = uz$ and $y=wz$, with both factorizations length-additive.  We write 
\begin{math} \langle u, w \rangle_P\end{math}
for the equivalence class of $[u,w]_P$ and denote the set of all such classes by $\mathcal{Q}(W,W_P)$.  

Define a partial order on $\mathcal{Q}(W,W_P)$ by setting 
\begin{math}\langle u, w \rangle_P \leq \langle x, y \rangle_P\end{math} 
if there exist representatives $[u',w']_P$ of 
\begin{math} \langle u,w \rangle_P\end{math}
 and $[x',y']_P$ of 
 \begin{math} \langle x, y \rangle_P\end{math}
 with \begin{math} [x',y']_P\subseteq [u',w']_P.\end{math}
 The poset $\mathcal{Q}(W,W_P)$ was first studied by Rietsch, in the context of closure relations for totally nonnegative cells in general flag manifolds \citep{Rie06}.  Williams proved a number of combinatorial results about this poset directly \citep{Wil07}.  We quote a few more results from \citep{KLS13}.

\begin{lem}\label{stillless}  Suppose $u \leq_P w$, $u = u^Pu_P$, and $w = w^Pw_P$, where $u^P$ and $w^P$ are minimal-length coset representatives for $W_P$, and $u_P,w_P \in W_P$.  Then $w_P \leq_{(l)} u_P$.
\end{lem}
\begin{lem}\label{minrep} Suppose $u \leq w$, and $w$ is a minimal-length coset representative for $wW_P$.  Then $u \leq_P w$.
\end{lem} 

\begin{lem}
 Each $\langle x,y \rangle_P$ has a unique representative $[x',y']_P$ where $y'$ is of minimal length in its left coset $y'W_P$.
\end{lem}

Hence elements of $\mathcal{Q}(W,W_P)$ are in bijection with pairs $(u,w)$ where $u \leq w$, and $w$ is of minimal length  in its coset $wW_P$.  In the case where $W = S_n$ and $W_P = S_k \times S_{n-k}$, these minimal-length coset representatives are called \emph{Grassmannian permutations} of type $(k,n)$.  Concretely, a permutation is Grassmannian of type $(k,n)$ if it is increasing on $[k]$ and on $[k+1,n]$.  The following lemma is due to Postnikov \citep[Section 20]{Pos06}.  

\begin{lem}\label{grass}Let $u,w \in S_n$, and suppose $w$ is Grassmannian of type $(k,n)$.  Then $u \leq w$ if and only if $u(i) \leq w(i)$ for all $i \in [k]$, and $u(i) \geq w(i)$ for all $i \in [k+1,n]$.
\end{lem}

If $u \leq_P w$ and $u' \leq_P w'$, then $\Pi_{u,w} = \Pi_{u',w'}$ if and only if $\langle u,w \rangle_P = \langle u',w' \rangle_P.$  Moreover, in this case $\mathring{\Pi}_{u,w}$ and $\mathring{\Pi}_{u',w'}$ are equal.  Hence there is a unique open projected Richardson variety $\mathring{\Pi}_{\langle u,w \rangle_P}$ corresponding to $\langle u,w \rangle_P$.  The poset $\mathcal{Q}(W,W_P)$ is isomorphic to the poset of projected Richardson varieties, ordered by \emph{reverse} inclusion \citep{Rie06}.  

When $W = S_n$ and $W_P = S_k \times S_{n-k}$, we denote the poset $\mathcal{Q}(W,W_P)$ by $\mathcal{Q}(k,n)$.  When $W=S_n^C$ and $W_P = (S_n^C)_n$, we denote the $P$-Bruhat order by $\leq_n^C$ and write $\mathcal{Q}^C(2n)$ for $\mathcal{Q}(W,W_P)$.  We denote equivalence classes in $\mathcal{Q}(k,n)$ and $\mathcal{Q}^C(2n)$ by $\langle u, w \rangle_n$ and $\langle u, w \rangle_n^C$ respectively.

\subsection{Deodhar components and MR parametrizations}

\subsubsection{Distinguished subexpressions}

Deodhar defined a class of cell decompositions for general flag varieties $G/B$ with components indexed by \emph{distinguished subexpressions} of the Weyl group $W$ of $G$ \citep{Deo85}.  Let $\mf{w}=s_{i_1}\cdots s_{i_m}$ be a reduced word for some $w \in W$.  We gather some facts about distinguished subexpressions, borrowing most of our conventions from \citep{MR04}.  A \emph{subexpression} $\mf{u}$ of $\mf{w}$ is an expression obtained by replacing some of the factors $s_{i_j}$ of $\mf{w}$ with the identity permutation, which we denote by $1$.  

We write \begin{math}\mf{u} \preceq \mf{w}\end{math} to indicate that $\mf{u}$ is a subexpression of $\mf{w}$.  We denote the $t^{th}$ factor of $\mf{u}$,  which may be either $1$ or a simple transposition, by $u_t$, and write $u_{(t)}$ for the product \begin{math}u_1u_2\ldots u_t\end{math}.  For notational convenience, we set \begin{math}u_0 = u_{(0)} = 1\end{math}.  We denote the $t^{th}$ simple transposition in $\mf{u}$ by $u^t$.  

\begin{defn}A subexpression $\mf{u}$ of $\mf{w}$ is called \textbf{distinguished} if we have
\begin{equation}u_{(j)} \leq u_{(j-1)}s_{i_j}\end{equation}
for all \begin{math}1 \leq j \leq m\end{math}.
\end{defn}

\begin{defn}A distinguished subexpression $\mf{u}$ of $\mf{w}$ is called \textbf{positive distinguished} if 
\begin{equation}u_{(j-1)} < u_{(j-1)} s_{i_j}\end{equation}
for all \begin{math}1 \leq j \leq m\end{math}
We will sometimes abbreviate the phrase ``positive distinguished subexpression'' to \emph{PDS}.
\end{defn}

Given a subexpression \begin{math}\mf{u} \preceq \mf{w}\end{math}, we say $\mf{u}$ is \emph{a subexpression for} \begin{math}u= u^1u^2\cdots u^r \end{math}.  By abuse of notation, we identify the subexpression $\mf{u}$ with the word \begin{math} u^1u^2\cdots u^{r}\end{math}, also denoted $\mf{u}$.  If the subexpression $\mf{u}$ is positive distinguished, then the corresponding word is reduced.  The following lemma is an easy consequence of the above definitions.

\begin{lem}{\citep{MR04}}
Let \begin{math} u \leq w \in W,\end{math} and let $\mf{w}$ be a reduced word for $w$.  Then the following are equivalent
\begin{enumerate}
\item $\mf{u}$ is a positive distinguished subexpression of $\mf{w}$.
\item $\mf{u}$ is the lexicographically first subexpression for $u$ in $\mf{w}$, working from the right.
\end{enumerate}
In particular, there is a unique PDS for $u$ in $\mf{w}$.
\end{lem}

Condition 2 means precisely that the factors $u^t$ are chosen greedily, as follows.  Suppose $\ell(u) = r$.  Working from the right, we set \begin{math}u^r = s_{i_j},\end{math} where $j$ is the largest index such that \begin{math}s_{i_j} \leq_{(l)} u.\end{math}  We then take $u^{r-1}$ to be the next rightmost factor of $\mf{w}$ such that \begin{math}u^{r-1}u^r \leq_{(l)} u\end{math} and so on, until we have \begin{math} u^1u^2\cdots u^r = u\end{math}.  

\subsubsection{Deodhar's decomposition}

To construct a Deodhar decomposition of $G/B$, we first fix a reduced word $\mf{w}$ for each $w \in W$.  There is a Deodhar component \begin{math}\mathcal{R}_{\mf{u},\mf{w}}\end{math} for each distinguished subexpression \begin{math}\mf{u} \preceq \mf{w}.\end{math}  For a reduced word $\mf{w}$ of $w$, we have
\begin{equation}\mathring{R}_{u,w} = \bigsqcup_{\mf{u} \preceq \mf{w} \text{ distinguished}}\mathcal{R}_{\mf{u},\mf{w}}.\end{equation}

Marsh and Rietsch gave explicit parametrizations for the components $\mathcal{R}_{\mf{u},\mf{w}}$, which we call \emph{MR parametrizations}.  Rather than using Deodhar's original construction, we define the \emph{Deodhar component} $\mathcal{R}_{\mf{u},\mf{w}}$ to be the image of the parametrization corresponding to $\mf{u} \preceq \mf{w}$, as defined in \citep{MR04}.  We give the parametrizations explicitly in Section \ref{params}.

\begin{rmk}\label{changeterms} In \cite{Kar14}, we called MR parametrizations ``Deodhar parametrizations,'' since they parametrize Deodhar components.  In this paper, we use ``MR parametrizations'' for clarity of attribution, and to keep our terminology consistent with \cite{TW13}.
\end{rmk}

Fix a parabolic subgroup $P$ of $G$ which contains $B$.  If $u,w \in W$ with $u \leq_P w,$ then $\mathring{R}_{u,w} \subseteq G/B$ projects isomorphically onto its image in $G/P$, which is one of Lusztig's projected Richardson strata.  To define a Deodhar decomposition of $G/P$, choose a representative $[u,w]_P$ for each \begin{math} \langle u',w' \rangle_P \in \mathcal{Q}(W,W_P),\end{math} and a reduced word $\mf{w}$ for each chosen $w$.  
For each of the selected $u \leq_P w$ and each distinguished subexpression $\mf{u} \preceq \mf{w}$ where $\mf{w}$ is the selected word for $w$, the \emph{Deodhar component} of $G/P$ corresponding to \begin{math}\mf{u} \preceq \mf{w}\end{math} is given by 
\begin{equation} \mathcal{D}_{\mf{u},\mf{w}} := \pi_P(\mathcal{R}_{\mf{u},\mf{w}}).\end{equation}
Since $\pi_P$ is an isomorphism on $\mathring{R}_{u,w}$, composing Marsh and Rietsch's parametrization of $\mathcal{R}_{\mf{u},\mf{w}}$ with $\pi_P$ gives a parametrization of $\mathcal{D}_{\mf{u},\mf{w}}$.  

\subsubsection{Total nonnegativity.}  Let $\mf{w}$ be a reduced word of $w$, let $u \leq_P w$, and let $\mf{u} \preceq \mf{w}$ be the unique PDS for $u$ in $\mf{w}$.  Then $\mathcal{R}_{\mf{u},\mf{w}}$ is the unique top-dimensional Deodhar component of $\mathring{R}_{u,w}$, and $\mathcal{R}_{\mf{u},\mf{w}}$ is dense $\mathring{R}_{u,w}$ \citep{MR04}.  Hence the corresponding parametrization of $\mathcal{D}_{\mf{u},\mf{w}}$ gives a birational map to $\mathring{\Pi}_{u,w}$ which is an isomorphism onto its image.  We call this an \emph{MR parametrization} of $\mathring{\Pi}_{u,w}$.

The \emph{totally nonnegative part} $(\mathring{R}_{u,w})_{\geq 0}$ of the projected Richardson variety $\mathring{R}_{u,w}$ is the subset of $\mathcal{R}_{\mf{u},\mf{w}}$ where all parameters in the MR parametrization are positive real.  Projecting to $G/P$, we obtain a parametrization of the totally nonnegative part of $\mathring{\Pi}_{u,w}$.  Hence $(\mathring{\Pi}_{u,w})_{\geq 0}$ is the part of $\mathcal{D}_{\mf{u},\mf{w}}$ where all parameters are positive real.  Note that the image is independent of our choice of $\mf{w}$ \citep{MR04}.  Remarkably, when $G/P = \Gr(k,n)$, this is precisely the locus where all Pl\"{u}cker coordinates are nonnegative real, up to multiplication by a common scalar.  We show in Section \ref{positivity} that a similar statement holds for $\Lm(2n)$.

\subsubsection{Parametrizing Deodhar components}

\label{params}

We now review Marsh and Rietsch's parametrizations for Deodhar components.  Let $G/B$ be a flag variety, and let 
\[\Pi=\{\alpha_i \mid i \in I\}\]
be a choice of simple roots for the root system of $G$.  For each $\alpha_i \in \Pi,$ there is a group homomorphism $\varphi_i:\Sl_2 \rightarrow G,$
which yields $1$-parameter subgroups 

\[x_i(m) = \varphi_i
\begin{pmatrix}
1 & m \\
0 & 1 
\end{pmatrix}
\qquad
y_i(m) = \varphi_i
\begin{pmatrix}
1 & 0 \\
m & 1 
\end{pmatrix}
\qquad
\alpha_i^{\vee}(t) = \varphi_i
\begin{pmatrix}
t & 0 \\
0 & t^{-1}
\end{pmatrix}
\]

Such a choice of simple roots and of homomorphisms $\varphi_i$ is called a \emph{pinning}.
For each simple reflection $s_i$ of $W$, we define a corresponding element of $G/B$ by 
\[\dot{s_i} = \varphi_i
\begin{pmatrix}
0 & -1 \\
1 & 0 
\end{pmatrix}\]

If $w = s_{i_1}\cdots s_{i_m}$ is a reduced word, define $\dot{w} = \dot{s}_{i_1}\cdots \dot{s}_{i_m}$.  This is independent of the choice of reduced expression \citep{MR04}.  For each root $\alpha$, there is some simple root $\alpha_i$ and $w \in W$ such that $w\alpha_i = \alpha$.  The \emph{root subgroup} corresponding to $\alpha$ is given by $\{\dot{w}x_i(t)\dot{w}^{-1}  \mid t \in \mathbb{C}\}.$

Marsh and Rietsch give explicit parametrizations for Deodhar components, using products of the elements $x_i$, $y_i$, and $\dot{s_i}$ \citep{MR04}.
For $\mf{u} \preceq \mf{w}$ a subexpression, with $\mf{u} = s_{i_1}\cdots s_{i_m}$, they define
\[J_{\mf{u}}^+ = \{k \in \{1,\ldots,m\} \mid u_{(k-1)}<u_{(k)}\}\]
\[J_{\mf{u}}^{\circ} = \{k \in \{1,\ldots,m\} \mid u_{(k-1)} = u_{(k)}\}\]
\[J_{\mf{u}}^{-} = \{k \in \{1,\ldots,m\} \mid u_{(k-1)} > u_{(k)}\}\]

We assign an element of $G$ to each simple reflection $s_{i_k}$ of $\mathbf{w}$ according to the rule below, where $t_k$ and $m_k$ are parameters, the parameter $t_k$ takes values in $\mathbb{C}^{\times}$, and $m_k$ takes values on $\mathbb{C}$, by:
\[g_k =
\begin{cases}
x_{i_k}(m_k)\dot{s}_{i_k}^{-1} & k \in J_{\mf{u}}^-\\
y_{i_k}(t_k) & k \in J_{\mf{u}}^{\circ}\\
\dot{s}_{i_k}& k \in J_{\mf{u}}^+
\end{cases}
\]
Note that if $\mf{u}$ is a PDS of $\mf{w}$, the third case never occurs.
We define a subset of $G$ corresponding to $\mf{u} \preceq \mf{w}$ by setting
\begin{equation}G_{\mf{u},\mf{w}} = \{ g_1\ldots g_n \mid t_k \in \mathbb{C}^{\times},\,m_k \in \mathbb{C}\} \label{pars}\end{equation}

\begin{prop}
\emph{\citep[Proposition 5.2]{MR04}}
The map \[(\mathbb{C}^*)^{J_{\mf{u}}^{\circ}} \times \mathbb{C}^{J_{\mf{u}}^-} \rightarrow G_{\mf{u},\mf{w}}\]
from \eqref{pars} is an isomorphism. 
Projecting to $G/B$ gives an isomorphism
\[G_{\mf{u},\mf{w}} \rightarrow \mathcal{R}_{\mf{u},\mf{w}}.\]
\end{prop}

\subsubsection{MR parametrizations in Type $A$.}
We now give explicit parametrizations for Deodhar components in $\mathcal{F}\ell(n)$.  Our discussion follows \citep[Section 4]{TW13}, which in turn draws on \citep[Sections 1 and 3]{MR04}.  However, we use a slightly different set of conventions for our matrix representatives.

Let $E_{i,j}(t)$ denote the elementary matrix with a nonzero entry $t$ in position $(i,j)$, $1$'s along the main diagonal, and $0$'s everywhere else.  Then we have
$x_i^A(t) = E_{i,i+1}(t)$, while $y_i^A(t) = E_{i+1,i}(t)$.
The matrix $\dot{s}_i^A$ is obtained from the $n \times n$ identity by replacing the $2 \times 2$ block whose upper left corner is at position $(i,i)$ with the block matrix
\begin{equation}\begin{bmatrix}
0 & -1\\
1 & 0
\end{bmatrix}\end{equation}

For $\alpha = \epsilon_i - \epsilon_j$ with $i < j$, the corresponding root subgroup consists of matrices $x_{\alpha}^A(t)=E_{i,j}(t)$, while the root subgroup corresponding to $-\alpha$ is given by the matrices $y_{\alpha}^A(t)=E_{j,i}(t)$.  We have
\[\dot{s}_i^Ax_{\alpha}^A(t)(\dot{s}_i^A)^{-1} = 
\begin{cases}
x_{s_i\alpha}(\pm t)& \text{ if $s_i\alpha$ is a positive root}\\
y_{(-s_i\alpha)}(\pm t) & \text{ if $s_i\alpha$ is a negative root}
\end{cases}
\]

\[\dot{s}_i^Ay_{\alpha}^A(t)(\dot{s}_i^A)^{-1} = 
\begin{cases}
x_{s_i\alpha}(\pm t)& \text{ if $s_i\alpha$ is a positive root}\\
y_{(-s_i\alpha)}(\pm t) & \text{ if $s_i\alpha$ is a negative root}
\end{cases}
\]
where the sign in each case depends on $i$ and $\alpha$.

For convenience, we let $x_{(i,j)}^A$ denote $x_{\alpha}^A$ with $\alpha = \epsilon_i - \epsilon_j$, and define $y_{(i,j)}^A$ similarly.
Suppose now that $\mf{u} \preceq \mf{w}$ is the unique PDS for $u$ in $w$, where $\ell(w) = m.$  For $1 \leq j \leq m,$ let 
\begin{equation}\dot{u_j}=\begin{cases}
\dot{s}^A_{i_j} & s_{i_j}^A \in \mf{u}\\
1 & \text{otherwise}
\end{cases}\end{equation}
and let $j_1,\ldots,j_d$ be the indices of the simple transpositions of $\mf{w}$ which do not appear in $\mf{u}$.  
 
For $1 \leq r \leq d$, let
\begin{equation}\beta_{r} = (\dot{u_1}\cdots \dot{u}_{j_r-1}) (y^A_{i_{j_r}}(t_{j_r}))(\dot{u}_{j_r-1}^{-1} \cdots \dot{u_1}^{-1}).\end{equation}
Then we can rewrite each \begin{math} g \in G_{\mf{u},\mf{w}}\end{math} in the form
\begin{equation}\label{betas}g = (\beta_1\beta_2 \cdots \beta_d)(\dot{u}_1\dot{u}_2\cdots \dot{u}_m) \end{equation}

Again by direct calculation, we have the following lemma.

\begin{lem}
Suppose for $a < b \in [n]$, we have \begin{math} s_i(a) < s_i(b).\end{math}  Then
\begin{equation} \label{signs} \dot{s_i}y^A_{(a,b)}(t)\dot{s_i}^{-1}=
\begin{cases}
y^A_{(s_i(a),s_i(b))}(-t)& i \in \{a-1,b-1\}\\
y^A_{(s_i(a),s_i(b))}(t)& \text{otherwise}
\end{cases}
\end{equation}
\end{lem}

This lemma implies the following.  Details may be found in \citep{Kar14}. 
 \begin{lem}
 For $\beta_r$ as above, define:
 \begin{align} a_r &= u_{(j_r-1)}(i_{j_r})\\
 b_r &= u_{(j_r-1)}(i_{j_r}+1)\\
 \theta_r &= |u([k]) \cap [a+1,b-1])|
 \end{align}
 Then we have
 \begin{equation} \beta_{r} = y^A_{(a,b)}((-1)^{\theta_r} t_{j_r}).\end{equation}
 \end{lem}

For notational convenience, we renumber our parameters $t_{i_j}$ and define $a_r,b_r$ such that 
\begin{equation} \beta_r = y^A_{(a_r,b_r)}(\pm t_r)\end{equation}
where the sign is determined as above.

\subsubsection{MR parametrizations in type $C$}

We now give a pinning for $\Sp(n)$.  Again, we follow \citep[Chapter 3]{BL00}, but with some sign changes to account for our choice of symplectic form.  Then we have 
\[x_i^C(t) = \begin{cases}
x^A_{i}(t)x_{2n-i}^A(t) & 1 \leq i \leq n-1\\
x_{n}^A(t) & i = n
\end{cases}\]
\[y_i^C(t) = \begin{cases}
y^A_{i}(t)y_{2n-i}^A(t) & 1 \leq i \leq n-1\\
y_{n}^A(t)  & i = n
\end{cases}\]
Similarly, the Weyl group elements are $\dot{s}^C_i$ given by
\[\dot{s}_i^C = \begin{cases}
\dot{s}_i^A\dot{s}_{2n-i}^A & 1 \leq i \leq n-1\\
\dot{s}_n^A  & i = n
\end{cases}\]

We now relate MR parametrizations of $\Lm(2n)$ to MR parametrizations of $\Gr(n,2n)$.  It is not hard to check that the embedding $S_n^C \hookrightarrow S_{2n}$ carries distinguished subexpressions $S_n^C$ to distinguished subexpressions in $S_{2n}$.

Let $\widetilde{\mf{v}} \preceq \widetilde{\mf{w}}$ be a distinguished subexpression in $S_n^C$ and let $\mf{v} \preceq \mf{w}$ be the corresponding distinguished subexpression in $S_{2n}$.  Let $\mathcal{R}^C_{\widetilde{\mf{u}},\widetilde{\mf{w}}}$ be the Deodhar component of $\Sp(2n)/B_+^{\sigma}$ corresponding to $\widetilde{\mf{u}} \preceq \widetilde{\mf{w}}$, and let $\mathcal{R}^A_{\mf{u},\mf{w}}$ be the component of $\mathcal{F}\ell(2n)$ corresponding to $\mf{u} \preceq \mf{w}$. 

\begin{lem}\label{deoAC} We have $\mathcal{R}^C_{\widetilde{\mf{u}},\widetilde{\mf{w}}} \hookrightarrow \mathcal{R}_{\mf{w},\mf{u}}$ under the map $\Sp(2n)/B_+^{\sigma} \hookrightarrow \mathcal{F}\ell(n)$.
\end{lem}

\begin{proof}
Consider the parametrization of $\mathcal{R}^C_{\widetilde{\mf{u}},\widetilde{\mf{w}}}$ described above.  We expand the formula for each element of $G^C_{\widetilde{\mf{u}},\widetilde{\mf{w}}}$ as a product of matrices $x_{i_r}^A(t_r)$, $y_{i_r}^A(t_r)$, and $\dot{s}_{i_r}^A$. 
The resulting sequence of matrices in $\Sl(n)$ is identical to the sequence of matrices corresponding to $\mf{u} \preceq \mf{w}$, except that the $t_i$ satisfy some relations of the form $t_i = t_{i+1}$.  Hence, $G^C_{\widetilde{\mf{u}},\widetilde{\mf{w}}}$ corresponds to a locally closed subset of $\mathcal{R}^A_{\mf{u},\mf{w}}$ and the claim follows.
\end{proof}

\subsection{Positroid varieties}

Let 
\begin{math}V \in \Gr_{\geq 0}(k,n)\end{math}.
The indices of the non-vanishing Pl{\"u}cker  coordinates of $V$ give a set
\begin{math} \mathcal{J} \subseteq {{[n]}\choose k}\end{math}
called the \emph{matroid} of $V$.  We define the \emph{matroid cell} 
\begin{math} \mathcal{M}_{\mathcal{J}}\end{math}
as the locus of points $V \in \Gr_{\geq 0}(k,n)$ with matroid $\mathcal{J}$.  The nonempty matroid cells in $\Gr_{\geq 0}(k,n)$ are the \emph{positroid cells} defined by Postnikov, and the corresponding matroids are called \emph{positroids}.  Positroid cells form a stratification of $\Gr_{\geq 0}(k,n)$, and each cell is homeomorphic to 
\begin{math}(\reals^+)^d\end{math} for some $d$ \citep[Theorem 3.5]{Pos06}.

The positroid stratification of $\Gr_{\geq 0}(k,n)$ extends to the complex Grassmannian $\Gr(k,n)$.  Taking the Zariski closure of a positroid cell of $\Gr_{\geq 0}(k,n)$ in $\Gr(k,n)$ gives a \emph{positroid variety}.  For a positroid variety
\begin{math} \Pi^A \subseteq \Gr(k,n),\end{math} we define the \emph{open positroid variety} 
\begin{math} \mathring{\Pi}^A \subset \Pi^A \end{math}
by taking the complement in $\Pi^A$ of all lower-dimensional positroid varieties.  The open positroid varieties give a stratification of $\Gr(k,n)$\citep{KLS13}.

Positroid varieties in $\Gr(k,n)$ may be defined in numerous other ways.  There is a beautiful description of positroid varieties as intersections of cyclically permuted Schubert varieties.  In particular, the positroid stratification refines the well-known Schubert stratification of $\Gr(k,n)$ \citep{KLS13}.  

Remarkably, positroid varieties in $\Gr(k,n)$ coincide with projected Richardson varieties \citep[Section 5.4]{KLS13}. Indeed, let \begin{math} u \leq_k w.\end{math}
The projection $\pi_k$ maps 
\begin{math} \mathring{R}^A_{u,w}\end{math}
homeomorphically onto its image, which is an open positroid variety 
\begin{math} \mathring{\Pi}^A_{u,w}.\end{math}  
The closure of $\mathring{\Pi}^A_{u,w}$ is a (closed) positroid variety $\Pi^A_{u,w}$ and we have $\pi_k(R^A_{u,w})=\Pi^A_{u,w}$.  Every positroid variety arises in this way.

Since positroid varieties are projected Richardson varieties, we have an isomorphism between $\mathcal{Q}(k,n)$ and the poset of positroid varieties, ordered by \emph{reverse} inclusion \citep[Section 5.4]{KLS13}.

\subsection{Grassmann necklaces}

Grassmann necklaces, first introduced in \citep{Pos06}, give another family of combinatorial objects which index positroid varieties.  

\begin{defn}A \emph{Grassmann necklace} $\mathcal{I} = (I_1,I_2,\ldots,I_n)$ of type $(k,n)$ is a sequence of $k$-element subsets of $[n]$ such that the following hold for all $i \in [n], $ with indices taken modulo $n$: 
\begin{enumerate}
\item If $i \in I_i$ then $I_{i+1} = (I_i \backslash \{i\}) \cup \{j\}$ for some $j \in [n].$
\item If $i \not\in I_i$, then $I_{i+1} = I_i$. 
\end{enumerate}
\end{defn}

For $a \in [n]$, let $\leq_a$ denote the cyclic shift of the usual linear order on $n$ given by 
\[a < a + 1 < \cdots < n < 1 < \cdots < a-1.\]
Note that $\leq_1$ is the usual order $\leq$.  We extend this to a partial order on ${{[n]}\choose{k}}$ by setting $I \leq_a J$ if we have $i_{\ell} \leq_a j_{\ell}$ for all $\ell \in [k]$, where 
\[I = \{i_1 <_a i_2 <_a \cdots <_a i_k\} \text{ and }J = \{j_1 <_a j_2 <_a \cdots <_a j_k\}.\]

Let $\mathcal{J}$ be a positroid of type $(k,n)$.  That is, $\mathcal{J}$ is the matroid of some nonempty positroid cell in $\Gr_{\geq 0}(k,n)$.  Then $\mathcal{J}$ is a collection of $k$-element subsets of $[n]$.  For each $1 \leq i \leq n$, let $I_i$ be the minimal element of $\mathcal{J}$ with respect to the shifted linear order $\leq_{i}$.  Then $\mathcal{I} = (I_1,\ldots,I_n)$ is a Grassmann necklace of type $(k,n)$. This procedure gives a bijection between Grassmann necklaces and positroids of type $(k,n)$ \citep{Pos06}. For the inverse bijection, let $\mathcal{I}=(I_1,\cdots,I_n)$ be a Grassmann necklace of type $(k,n)$, and let $\mathcal{J}$ be set of all $k$-element subsets $J \in {{[n]}\choose{k}}$ such that $I_a \leq_a J$ for all $a \in [n]$.  Then $\mathcal{J}$ is the positroid with Grassmann necklace $\mathcal{I}$ \citep{Pos06, Oh11}.  

Positroid varieties are not matroid varieties.  Suppose $X \in \Gr(k,n)$ is contained in the open positroid variety $\mathring{\Pi}$ with corresponding positroid $\mathcal{J}$.  The matroid of $X$ is contained in $\mathcal{J},$ but may not be equal to $\mathcal{J}$ \citep{KLS13}.  However, the two matroids are related.  Let $\mathcal{J}_X$ be the matroid of $X$, and let $(\mathcal{I}_1,\ldots,\mathcal{I}_n)$ be the Grassmann necklace of $\mathcal{J}$.  Then $I_i$ is the minimal element of $\mathcal{J}_X$ with respect to the shifted linear order $\leq_i$ for all $i \in [n]$ \citep{KLS14}.  

We define a partial order on Grassmann necklaces by setting $\mathcal{I} \leq \mathcal{I}'$ if $I_i \leq_i I_i'$ for all $i$. Hence the poset of Grassmann necklaces is isomorphic to the poset of positroids, ordered by \emph{reverse} containment.

\subsection{Bounded affine permutations and Bruhat intervals in type $A$}

\begin{defn} An \textbf{affine permutation} of order $n$ is a bijection 
\begin{math}f:\integers \rightarrow \integers\end{math}
which satisfies the condition 
\begin{equation} \label{affine} f(i+n) = f(i)+n\end{equation}
 for all 
 \begin{math} i \in \integers.\end{math}
 The affine permutations of order $n$ form a group, which we denote 
 \begin{math} \widetilde{S}_n.\end{math}
 \end{defn}

For a reductive group $G$, the \emph{extended affine Weyl group} $\widehat{W}$ of $G$ is defined by 
\[\widehat{W} = X_* \times W\]
where $W$ is the Weyl group of $G$ and $X_*$ is the cocharacter lattice of $G$ \citep{KR99}.  The group $\widetilde{S}_n$ is the extended affine Weyl group of $\Gl(n)$.  In particular, we have 
\[\widetilde{S}_n \cong \mathbb{Z}^n \rtimes S_n\]
where $\mathbb{Z}^n$ is the cocharacter lattice of $\Gl(n)$.  Each permutation $w$ acts periodically on $\mathbb{Z}$ with period $n$, while a \emph{translation element} $(a_1,\ldots,a_n) \in \mathbb{Z}^n$ acts by $i \mapsto i+a_in$, again extended periodically with period $n$.  As in \citep{HL11}, a cocharacter $\lm = (\lm_1,\ldots,\lm_n)$ of $\Gl(n)$ corresponds to the translation element of $\widetilde{S}_n$ is given by $(-\lm_1,\ldots,-\lm_n)$.
Note that we may realize both $\mathcal{F}\ell(n)$ and $\Gr(k,n)$ as quotients of $\Gl(n)$ rather than $\Sl(n)$.

\begin{defn}
For $k \in \mathbb{Z}$, an affine permutation of order $n$ has type $(k,n)$ if
\begin{equation}\displaystyle \frac{1}{n} \sum_{i=1}^n (f(i)-i)=k.\end{equation}
We denote the set of affine permutations of type $(k,n)$ by 
\begin{math}\widetilde{S}^k_n.\end{math}
\end{defn}

Affine permutations of type $(0,n)$ form an infinite Coxeter group, the \emph{affine symmetric group}. This is the \emph{affine Weyl group} of type $A_{n-1}$.  It has simple generators 
\begin{math}\widetilde{s}_1,\ldots,\widetilde{s}_{n},\end{math}
where $\widetilde{s}_i$ is the affine permutation which interchanges $i+rn$ and $i+1+rn$ for each 
\begin{math} r \in \integers.\end{math}

The Bruhat order on $\widetilde{S}_n^0$ extends to a Bruhat order on all of $\widetilde{S}_n$.  Each element of $\widetilde{S}_n$ may be written as a product $w\tau$, where $\tau$ preserves all simple roots of the affine Weyl group $\widetilde{S}_n^0$, and $w \in \widetilde{S}_n^0$. We say $w'\tau' \leq w\tau$ if $w' \leq w$ in Bruhat order, and $\tau' = \tau$.  
In our case, the condition $\tau = \tau'$ is equivalent to saying that $w'\tau'$ and $w\tau$ are both of type $(k,n)$ for some $k \in \mathbb{Z}$.  For $f = w \tau$, $f \in \widetilde{S}_n^k$ means that $\tau$ is the function $x \mapsto x + k$.

\begin{defn}
An affine permutation in 
\begin{math} \widetilde{S}_n\end{math}
 is \textbf{bounded} if it satisfies the condition
\begin{equation}i \leq f(i) \leq i+n \text{ for all }i \in \integers.\end{equation}
For $1 \leq k \leq n$, we write $\Bd(k,n)$ for the set of all bounded affine permutations of type $(k,n)$.  
\end{defn}

The set $\Bd(k,n)$ inherits the Bruhat order from 
\begin{math} \widetilde{S}^k_n.\end{math}  In fact, $\Bd(k,n)$ is a lower order ideal in 
\begin{math} \widetilde{S}^k_n\end{math}  
\cite[Lemma 3.6]{KLS13}.
Similarly, the length function $\ell$ on 
\begin{math} \widetilde{S}_n^0\end{math}
 induces a grading $\ell$ on $\Bd(k,n)$, defined by $\ell(w\tau) = \ell(w)$.  An \emph{inversion} of a bounded affine permutation $f$ is a pair $i < j$ such that 
\begin{math} f(i) > f(j).\end{math}
Two inversions $(i,j)$ and $(i',j')$ are equivalent if 
\begin{math} i'=i+rn\end{math}
and
\begin{math}j'=j+rn\end{math}
for some
\begin{math} r \in \integers.\end{math} 
We call the resulting equivalence classes \emph{type $\widetilde{A}$ inversions}. 
The length of a bounded affine permutation
\begin{math} f \in \Bd(k,n)\end{math}
is the number of type $\widetilde{A}$ inversions of $f$ \citep[Theorem 5.9]{KLS13}.

For 
\begin{math} J \in {{[n]}\choose{k}},\end{math}
 we define the \emph{translation element} 
 \begin{math} t_J \in \widetilde{S}_n\end{math} by setting 
\begin{equation}t_J(i)=\begin{cases}
i+n & i \in J\\
i & i \not\in J
\end{cases}\end{equation}
for 
\begin{math}1 \leq i \leq n,\end{math}
 and extending periodically.
Every 
\begin{math} f \in \Bd(k,n)\end{math} may be written in the form
\begin{equation}f=\sigma t_{\mu}=t_{\nu}\sigma\end{equation}
for some
\begin{math} \sigma \in S_n,\end{math}, elements
$\mu$ and $\nu$ of ${{[n]}\choose{k}},$
and $t_{\mu}$, $t_{\nu}$ translation elements.  Both factorizations are unique.

We now give an isomorphism between $\mathcal{Q}(k,n)$ and $\Bd(k,n)$.  This isomorphism may be viewed as a special case of Theorem 2.2 from \citep{HL11}.  Let $\langle u, w \rangle_k \in \Q(k,n)$.  
The function 
\begin{equation}\label{Q2Bd} f_{u,w} = ut_{[k]}w^{-1} \end{equation}
is a bounded affine permutation of type $(k,n)$.  If $[u',w']_k$ is any other representative of 
\begin{math} \langle u,w \rangle_k,\end{math}
then
\begin{math} f_{u',w'}=f_{u,w}.\end{math}
Hence we have a well-defined map 
\begin{math} \mathcal{Q}(k,n) \rightarrow \Bd(k,n),\end{math} which is in fact an isomorphism of posets \citep[Section 3.4]{KLS13}.
Note that for the translation elements $t_{w([k])}$ and $t_{u([k])},$ we have
\begin{equation} \label{permfactors} f_{u,w} = uw^{-1}t_{w([k])}=t_{u([k])}uw^{-1} \end{equation}
This follows easily from \eqref{Q2Bd}.

The poset of bounded affine permutations is anti-isomorphic to the poset of \emph{decorated permutations}, which Postnikov introduced to index positroid varieties \citep{Pos06}.  A decorated permutation of order $n$ is a permutation in $S_n$ with fixed points colored black or white.  If $f$ is a bounded affine permutation, and $\sigma$ is the corresponding decorated permutation, then black fixed points of $\sigma$ correspond to values $i \in [n]$ such that $f(i) =i$.  White fixed points correspond to values $i$ such that $f(i) = i + n$.  A decorated permutation $\sigma$ of order $n$ has type $(k,n)$ if we have
\[k = |\{i \in [n] \mid \sigma^{-1}(i) > i \text{ or $i$ is a white fixed point}\}|.\]

Postnikov represented decorated permutations visually using \emph{chord diagrams} \citep[Section 16]{Pos06}.  Let $\sigma$ be a decorated permutation.  A chord diagram for $\sigma$ is a circle with vertices labeled $1,2,\ldots,n$ in clockwise order.  We then draw arrows from vertex $i$ to vertex $\sigma(i)$ for all $i \in [n]$.  By convention, if $i$ is a white fixed point of $\sigma$, we draw a clockwise loop at $i$; if $i$ is a black fixed point, we draw a counter-clockwise loop.  

We say a pair of arrows in a chord diagram represents a \emph{crossing} if they are arranged in one of the configurations shown in Figure \ref{crossings}, for $i,j \in [n]$.  We say a pair of chords represents an \emph{alignment} if they are arranged in one of the configurations shown in Figure \ref{alignments}.  Note that rotating any of the diagrams shown in Figure \ref{alignments} gives a valid example of an alignment.  The analogous statement holds for crossings.

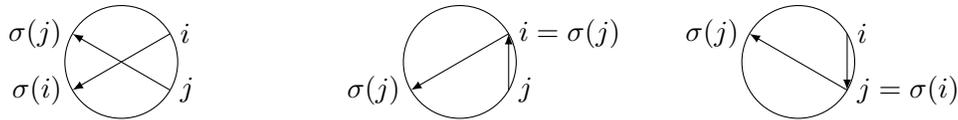
\begin{figure}[ht]
\centering
\begin{tikzpicture}[scale = 0.75]
\draw (0,0) circle (1);
\draw[->, >=latex] ({cos(30)},{sin(30)}) to (-{cos(30)},-{sin(30)});
\draw[->, >=latex] ({cos(-30)},{sin(-30)}) to (-{cos(-30)},-{sin(-30)});
\node[right] at ({cos(30)},{sin(30)}){$i$};
\node[left] at (-{cos(30)},{sin(30)}) {$\sigma(j)$};
\node[right] at ({cos(-30)},{sin(-30)}) {$j$};
\node[left] at (-{cos(-30)},{sin(-30)}) {$\sigma(i)$};
\begin{scope}[xshift = 6 cm]
\draw (0,0) circle (1);
\draw[->, >=latex] ({cos(30)},{sin(30)}) to (-{cos(30)},-{sin(30)});
\draw[->, >=latex] ({cos(-30)},{sin(-30)}) to ({cos(30)},{sin(30)});
\node[right] at ({cos(30)},{sin(30)}){$i = \sigma(j)$};
\node[right] at ({cos(-30)},{sin(-30)}) {$j$};
\node[left] at (-{cos(-30)},{sin(-30)}) {$\sigma(j)$};
\end{scope}
\begin{scope}[xshift = 12 cm]
\draw (0,0) circle (1);
\draw[->, >=latex] ({cos(30)},{sin(30)}) to ({cos(-30)},{sin(-30)});
\draw[->, >=latex] ({cos(-30)},{sin(-30)}) to (-{cos(-30)},-{sin(-30)});
\node[right] at ({cos(30)},{sin(30)}){$i$};
\node[left] at (-{cos(30)},{sin(30)}) {$\sigma(j)$};
\node[right] at ({cos(-30)},{sin(-30)}) {$j=\sigma(i)$};
\end{scope}
\end{tikzpicture}
\caption{Crossings in a chord diagram.}
\label{crossings}
\end{figure}

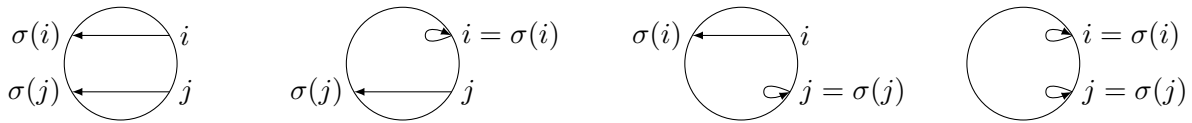
\begin{figure}[ht]
\centering
\begin{tikzpicture}[scale = 0.75]
\draw (0,0) circle (1);
\draw[->, >=latex] ({cos(30)},{sin(30)}) to (-{cos(30)},{sin(30)});
\draw[->, >=latex] ({cos(-30)},{sin(-30)}) to (-{cos(-30)},{sin(-30)});
\node[right] at ({cos(30)},{sin(30)}){$i$};
\node[left] at (-{cos(30)},{sin(30)}) {$\sigma(i)$};
\node[right] at ({cos(-30)},{sin(-30)}) {$j$};
\node[left] at (-{cos(-30)},{sin(-30)}) {$\sigma(j)$};
\begin{scope}[xshift = 5 cm]
\draw (0,0) circle (1);
\draw[->, >=latex] ({cos(30)},{sin(30)}) to [out = 205, in = 270] (0.4,{sin(30)}) to [out = 90, in = 155] (0.866,{sin(30)});
\draw[->, >=latex] ({cos(-30)},{sin(-30)}) to (-{cos(-30)},{sin(-30)});
\node[right] at ({cos(30)},{sin(30)}){$i = \sigma(i)$};
\node[right] at ({cos(-30)},{sin(-30)}) {$j$};
\node[left] at (-{cos(-30)},{sin(-30)}) {$\sigma(j)$};
\end{scope}
\begin{scope}[xshift = 11 cm]
\draw (0,0) circle (1);
\draw[->, >=latex] ({cos(30)},{sin(30)}) to (-{cos(30)},{sin(30)});
\draw[->, >=latex] ({cos(30)},-{sin(30)}) to [out = 155, in = 90] (0.4,-{sin(30)}) to [out = 270, in = 205] (0.866,-{sin(30)});
\node[right] at ({cos(30)},{sin(30)}){$i$};
\node[left] at (-{cos(30)},{sin(30)}) {$\sigma(i)$};
\node[right] at ({cos(-30)},{sin(-30)}) {$j=\sigma(j)$};
\end{scope}
\begin{scope}[xshift = 16 cm]
\draw (0,0) circle (1);
\draw[->, >=latex] ({cos(30)},{sin(30)}) to [out = 205, in = 270] (0.4,{sin(30)}) to [out = 90, in = 155] (0.866,{sin(30)});
\draw[->, >=latex] ({cos(30)},-{sin(30)}) to [out = 155, in = 90] (0.4,-{sin(30)}) to [out = 270, in = 205] (0.866,-{sin(30)});
\node[right] at ({cos(30)},{sin(30)}){$i = \sigma(i)$};
\node[right] at ({cos(-30)},{sin(-30)}) {$j=\sigma(j)$};
\end{scope}
\end{tikzpicture}
\caption{Alignments in a chord diagram.  Note that loops must be oriented as shown to give a valid alignment.}
\label{alignments}
\end{figure}

\subsection{Plabic graphs}

\label{plabic}

A \emph{plabic graph} is a planar graph embedded in a disk, with each vertex colored black or white.  (Plabic is short for planar bicolored.)  The boundary vertices are numbered $1,2,\ldots,n$ in clockwise order, and all boundary vertices have degree one. We call the edges adjacent to boundary vertices \emph{legs} of the graph.  A leaf adjacent to a boundary vertex is called a \emph{lollipop}.  A black leaf adjacent to a white boundary vertex is a \emph{black lollipop}, while a \emph{white lollipop} is the opposite.  We further assume that every vertex in a plabic graph is connected by some path to a boundary vertex.

Postnikov introduced plabic graphs in \citep[Section 11.5]{Pos06}, where he used them to construct parametrizations of positroid cells in the totally nonnegative Grassmannian.  In this paper, we follow the conventions of \citep{Lam13}, which are more restrictive than Postnikov's. In particular, we require our plabic graphs to be bipartite, with the black and white vertices forming the partite sets.  An \emph{almost perfect matching} on a plabic graph is a subset of its edges which uses each interior vertex exactly once; boundary vertices may or may not be used.  We consider only plabic graphs which admit an almost perfect matching.

We can write any plabic graph as a union of paths and cycles, as follows.  Start with any edge $e=\{u,v\}$.  Begin traversing this edge in either direction, say 
\begin{math} u \rightarrow v.\end{math}
Turn (maximally) left at every internal white vertex, and (maximally) right at every internal black vertex.  The path ends when we either reach a boundary vertex, or find ourselves about to retrace the edge $u \rightarrow v$.  Repeating this process gives a description of $G$ as a union of directed paths and cycles, called \emph{trips}.  Each edge is used twice in this decomposition, once in each direction.  Given a plabic graph $G$ with $n$ boundary vertices, we define the \emph{trip permutation} 
\begin{math} \sigma_G \in S_n\end{math}
of $G$ by setting 
\begin{math} \sigma_G(a) = b\end{math}
if the trip that starts at boundary vertex $a$ ends at boundary vertex $b$.  

There are a number of \emph{local moves} and \emph{reductions} defined for plabic graphs; again, we use the conventions of \citep{Lam13}, which are adapted slightly from those of \citep{Pos06}.  
The moves are defined as follows.  Both moves are reversible, and preserve the trip permutation.

\begin{enumerate}[(M1)]
\item The ``spider," ``square," or ``urban renewal" move.  We may transform the portion of a plabic graph shown at left in Figure \ref{squaremove} into the portion shown at right, and vice versa.
\item Degree-two vertex removal.  If a vertex $v$ has degree 2, we may contract the incident edges $(u,v)$ and $(v,u')$ to a single vertex.  Note that if $v$ is adjacent to a boundary vertex $b$, we cannot contract all the incident edges, since boundary vertices must have degree $1$.  Hence, we simply remove the vertex $v$, and reverse the color of $b$.
\end{enumerate}

\begin{figure}[h]
\begin{tikzpicture}[scale = 0.75]
\draw (-2,0) -- (0,1) -- (2,0) -- (0,-1) -- (-2,0);
\draw (0,1) -- (0,2);
\draw (0,-1) -- (0,-2);
\wdot{-2}{0};\wdot{2}{0};\bdot{0}{1};\bdot{0}{-1};\wdot{0}{2};\wdot{0}{-2};
\draw [<->] (3,0) [out = 30, in = 150] to (5,0);
\begin{scope}[xshift = 8 cm]
\draw (-1,0) -- (0,2) -- (1,0) -- (0,-2) -- (-1,0);
\draw (-2,0) -- (-1,0);
\draw (1,0) -- (2,0);
\wdot{-2}{0};\wdot{2}{0};\bdot{-1}{0};\bdot{1}{0};\wdot{0}{2};\wdot{0}{-2};
\end{scope}
\end{tikzpicture}
\caption{A square move}
\label{squaremove}
\end{figure}
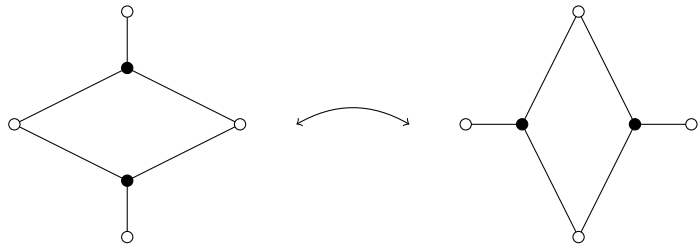

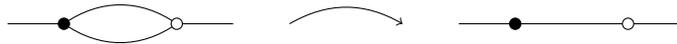
\begin{figure}[ht]
\begin{tikzpicture}[scale = 0.75]
\draw (0,0) -- (1,0) to [out = 35, in = 145] (3,0) -- (4,0);
\draw (1,0) to [out = 325, in = 215] (3,0);
\bdot{1}{0};\wdot{3}{0};
\draw [->] (5,0) [out = 30, in = 150] to (7,0);
\begin{scope}[xshift = 8 cm]
\draw (0,0) -- (1,0) -- (3,0) -- (4,0);
\bdot{1}{0};\wdot{3}{0};
\end{scope}
\end{tikzpicture}
\caption{A reduction.}
\label{reduction}
\end{figure}

Reductions, in contrast, are not reversible, and may change the trip permutation.  We have two reductions.
\begin{enumerate}[(R1)]
\item Multiple edges with the same endpoints may be replaced by a single edge.  See Figure \ref{reduction}.
\item  Leaf removal.  If $v$ is a leaf, and $(u,v)$ the unique edge adjacent to $v$, we may remove both $(u,v)$ and all edges adjacent to $u$.  However, if $u$ is adjacent to a boundary vertex $b$, the edge $(b,u)$ is replaced by a boundary edge $(b,w)$, where $w$ has the same color as $v$, and the color of $b$ flips.
\end{enumerate}

A plabic graph $G$ is \emph{reduced} if it cannot be transformed using the local moves M1-M2 into a plabic graph $G'$ on which we can perform a reduction.
If $G$ is a reduced graph, each fixed point of $\sigma_G$ corresponds to a boundary leaf \citep[Section 13]{Pos06}.
Suppose $G$ has $n$ boundary vertices, and suppose we have
\begin{equation}k=|\{a \in [n] \mid \sigma_G(a) < a \text{ or $\sigma_G(a)=a$ and $G$ has a white boundary leaf at $a$}\}|\end{equation}
Then we can construct a bounded affine permutation $f_G \in \Bd(k,n)$ corresponding to $G$ by setting
\begin{equation}f_G(a) = 
\begin{cases}
\sigma_G(a) & \sigma_G(a) > a \text{ or $G$ has a black boundary leaf at $a$}\\
\sigma_G(a)+n & \sigma_G(a) < a \text{ or $G$ has a white boundary leaf at $a$}
\end{cases}
\end{equation}
Thus we have have a correspondence between plabic graphs and positroid varieties: to a reduced plabic graph $G$, we associate the positroid $\Pi_G^A$ corresponding to $f_G$.  This correspondence is not a bijection.  Rather, we have a family of reduced plabic graphs for each positroid variety.  Two reduced plabic graphs $G$ and $G'$ have the same bounded affine permutation (and hence, the same associated positroid variety) if and only if we can transform $G$ into $G'$ using a sequence of local moves M1 and M2 \citep[Theorem 13.4]{Pos06}.

We now describe a way to build plabic graphs inductively by adding new edges, called \emph{bridges}, to existing graphs.  The resulting graphs are called \emph{bridge graphs}.  This construction appears in \citep{ABCGPT12} and also, in slightly less general form, in \citep{Lam13}.  

We begin with a plabic graph $G$.  To add a bridge, we choose a pair of boundary vertices $a < b$, such that every 
\begin{math} c \in [a+1,b-1]\end{math}
is a lollipop.  Our new edge will have one vertex on the leg at $a$, and one on the leg at $b$.  If $a$ (respectively $b$) is a lollipop, then the leaf at $a$ must be white (respectively black), and we use that boundary leaf as one endpoint of the bridge.  If $a$ (respectively $b$) is not a lollipop, we instead insert a white (black) vertex in the middle of the leg at $a$ (respectively, $b$).  We call the new edge an $(a,b)$-\emph{bridge}. After adding the new edge, our graph may no longer be bipartite.  In this case, we insert additional vertices of degree two or change the color of boundary vertices as needed to obtain a bipartite graph $G'$.  (See Figure \ref{addbridge}.)  

\begin{prop}\citep{Lam13}
\label{canadd}
Suppose $G$ is reduced.  Choose 
\begin{math} 1 \leq a < b \leq n\end{math}
such that
\begin{math} f_G(a) > f_G(b),\end{math}
and each 
\begin{math}c \in [a+1,b-1]\end{math}
is a lollipop.  Let $G'$ be the graph obtained by adding an $(a,b)$-bridge to $G$.  Then $G'$ is reduced and 
\begin{displaymath} f_{G'}= f \circ (a,b) \in \Bd(k,n).\end{displaymath}  
Moreover, 
\begin{math}f_{G'} \lessdot f_G\end{math}
in the Bruhat order on $\Bd(k,n)$, and so $\Pi_G$ is a codimension-one subvariety of $\Pi_{G'}$.  
\end{prop}

\begin{figure}[ht]
\centering
\begin{subfigure}[b]{\textwidth}
\centering
\begin{tikzpicture}[scale = 0.7]
\draw (2,4) -- (2,0);
\draw (0,3) -- (2,3); \wdot{0}{3}; \bdot{2}{3};
\draw (0,1) -- (2,1); \bdot{0}{1}; \wdot{2}{1};
\draw [->] (3,2) [out = 30, in = 150] to (5,2);
\begin{scope}[xshift = 6 cm]
\draw (2,4) -- (2,0);
\draw (0,3) -- (2,3); \wdot{0}{3}; \bdot{2}{3};
\draw (0,1) -- (2,1); \bdot{0}{1}; \wdot{2}{1};
\edge{0}{1}{2};
\end{scope}
\end{tikzpicture}
\caption{Adding a bridge between lollipops.}
\end{subfigure}
\\
\vspace{0.2 in}

\begin{subfigure}[b]{\textwidth}
\centering
\begin{tikzpicture}[scale = 0.7]
\draw (2,4) -- (2,0);
\draw (-1,3.5) -- (0,3);
\draw (-1,2.5) -- (0,3);
\draw (-1,1.5) -- (0,1);
\draw (-1,0.5) -- (0,1);
\draw (0,3) -- (2,3); \wdot{0}{3}; \bdot{2}{3};
\draw (0,1) -- (2,1); \bdot{0}{1}; \wdot{2}{1};
\draw [->] (3,2) [out = 30, in = 150] to (5,2);
\begin{scope}[xshift = 8 cm]
\draw (2,4) -- (2,0);
\draw (-2,3.5) -- (-1,3);
\draw (-2,2.5) -- (-1,3);
\draw (-2,1.5) -- (-1,1);
\draw (-2,0.5) -- (-1,1);
\draw (-1,3) -- (2,3); \wdot{-1}{3}; \bdot{2}{3};
\draw (-1,1) -- (2,1); \bdot{-1}{1}; \wdot{2}{1};
\edge{0.5}{1}{2};
\draw [->] (3,2) [out = 30, in = 150] to (5,2);
\end{scope}
\begin{scope}[xshift = 16 cm]
\draw (2,4) -- (2,0);
\draw (-2,3.5) -- (-1,3);
\draw (-2,2.5) -- (-1,3);
\draw (-2,1.5) -- (-1,1);
\draw (-2,0.5) -- (-1,1);
\draw (-1,3) -- (2,3); \wdot{-1}{3}; \bdot{2}{3};
\draw (-1,1) -- (2,1); \bdot{-1}{1}; \wdot{2}{1};
\edge{0.5}{1}{2}; \bdot{-0.25}{3}; \wdot{-0.25}{1};
\end{scope}
\end{tikzpicture}
\caption{Adding a bridge whose endpoints are not lollipops.  Note that after adding the bridge, we add additional vertices of degree $2$ to create a bipartite graph.}
\end{subfigure}
\caption{Adding bridges to a plabic graph.}
\label{addbridge}
\end{figure}
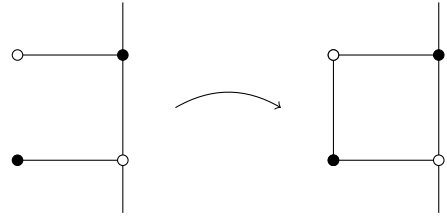
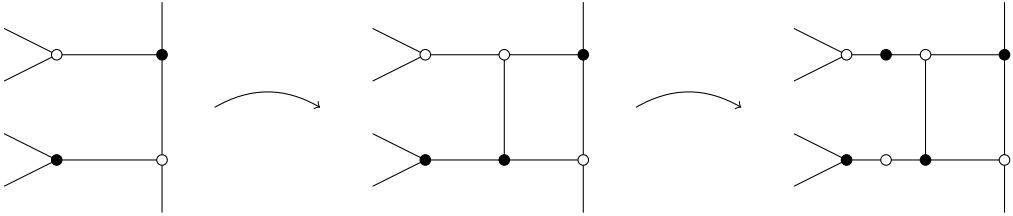

The zero-dimensional positroid varieties correspond to the points in $\Gr(k,n)$ which have a single non-zero Pl\"{u}cker coordinate $\mu$.  There is a unique reduced plabic graph for each $\mu \in {{[n]}\choose{k}}$, which consists of $n$ lollipops.  The $k$ lollipops corresponding to elements of $\mu$ are white; the rest are black.  We call a plabic graph consisting only of lollipops a \emph{lollipop graph}.  

A \emph{bridge graph} is a plabic graph which is constructed from a lollipop graph by successively adding bridges 
\begin{equation}(a_1,b_1),\ldots,(a_d,b_d),\end{equation} 
where at each step, $(a_i,b_i)$ satisfies the hypothesis of Proposition \ref{canadd} for the graph obtained by adding the first $i-1$ bridges.
Hence a bridge graph is always reduced.  
Let $u \leq_k w$.  Then by \eqref{permfactors} we have
\begin{equation}f_{u,w}= t_{u([k])}uw^{-1}\end{equation}
where $t_{u([k])}$ is the translation element corresponding to $u([k])$.
To construct a bridge graph for 
\begin{math}\Pi_{\langle u,w\rangle_k}\end{math}
we begin with the lollipop graph corresponding to $u([k])$, and successively add bridges to obtain a graph with bounded affine permutation $f_{u,w}$.  It is perhaps not obvious that every positroid variety has a bridge graph.  However, this follows from earlier work on the subject.  In particular, Postnikov's \Le-diagrams correspond to a particular choice of bridge graph for each bounded affine permutation.

\subsection{Parametrizations from plabic graphs}

Let $G$ be a reduced plabic network whose bounded affine permutation has type $(k,n)$.  Suppose $G$ has $e$ edges, and assign weights 
\begin{math}t_1,\ldots,t_e\end{math}
to the edges of $G$.  Postnikov defined a surjective map from the space of positive real edge weights of $G$ to the positroid cell 
\begin{math}\left(\mathring{\Pi}^A_G\right)_{\geq 0}\end{math}
in $\Gr_{\geq 0}(k,n)$, called the \emph{boundary measurement map} \citep[Section 11.5]{Pos06}.   Postnikov, Speyer and Williams re-cast this construction in terms of almost perfect matchings \citep[Section 4-5]{PSW09}, an approach Lam developed further in \citep{Lam13}.  Muller and Speyer showed that we can apply the same map to the space of nonzero complex edge weights, and obtain a map to the positroid variety 
\begin{math} \mathring{\Pi}^A_G\end{math}
in $\Gr(k,n)$ \citep{MS14}.  We use the definition of the boundary measurement map found in \citep{Lam13}. 

For $P$ an almost perfect matching on a plabic graph $G$ with $e$ edges, let 
\begin{equation}\partial(P) = \{\text{black boundary vertices used in $P$}\} \cup \{\text{white boundary vertices \emph{not} used in $P$}\}\end{equation}
Then $|\partial(P)|=k$, and we define the \emph{boundary measurement map} 
\begin{equation}\partial_G: \complexes^e \rightarrow \mathbb{P}^{{n \choose k}-1}\end{equation}
to be the map which sends 
\begin{math}(t_1,\ldots,t_e)\end{math}
to the point with homogeneous coordinates 
\begin{equation}\Delta_J = \sum_{\partial(P)=J} t^P\end{equation}
where the sum is over all matchings $P$ of $G$, and $t^P$ is the product of the weights of all edges used in $P$ \citep{Lam13}.

For positive real edge weights, the boundary measurement map $\partial_G$ is surjective onto the positroid with bounded affine permutation $f_G$.  If instead we let the edge weights range over 
\begin{math} \complexes^{\times},\end{math}
we obtain a well-defined map to the open positive variety 
\begin{math} \mathring{\Pi}^A_G\end{math}
in $\Gr(k,n)$. The image is an open dense subset of $\Pi^A_G$ \citep{MS14}.

The boundary measurement map is typically not injective, due to the action of the \emph{gauge group}.  Let $V$ be the set of all internal vertices of $G$.  The gauge group $\mathbb{G}^V$ is a copy of $(\mathbb{C}^{\times})^{|V|}$ with coordinates indexed by $V$.  Let $\mathbb{G}^E$ denote the space of complex edge weights of $G$.  Then $\mathbb{G}^V$ acts on $\mathbb{G}^E$ by \emph{gauge transformations}.  For $\mu \in \mathbb{G}^V$ and $v \in V$, let $\mu_v$ be the coordinate of $\mu$ corresponding to $v$.  Then the action of $\mu$ multiplies the weights of each edge incident to $v$ by $\mu_v$.  The weight of an edge $(v,w)$ is thus multiplied by the product $\mu_v\mu_w$.  It is easy to see that the action of $\mathbb{G}^V$ preserves the boundary measurement map.  Taking the quotient by this action, we obtain a map 
\[\mathbb{D}_G:\mathbb{G}^E/\mathbb{G}^V \rightarrow \mathring{\Pi}^A_G\]
which carries the image of each weighting $(t_1,\ldots,t_e)$ in $\mathbb{G}^E/\mathbb{G}^V$ to $\partial_G(t_1,\ldots,t_e)$.  This map is not only injective, but birational onto its image \citep{MS14}.

Analogous statements hold for positive real edge weights, where the action is by positive real gauge transformations; in this setting, taking the quotient by the gauge group gives an isomorphism onto the positroid cell corresponding to $G$ \citep{Pos06}.  We will abuse terminology slightly, and refer to both $\partial_G$ and $\mathbb{D}_G$ as the \emph{boundary measurement map}; it should be clear from context which map is meant.

Taking the quotient by the gauge group is equivalent to specializing an appropriate set of edge weights to $1$, and letting the remaining edge weights range over either $\mathbb{R}^+$ or $\mathbb{C}^{\times}$.  Indeed, suppose
$F \subset E$ is a set which meets the following conditions. 
\begin{enumerate}
\item $F$ is a spanning forest of $G$.
\item Each connected component of $F$ has exactly one vertex on the boundary.
\end{enumerate}
We may construct such a set inductively for any $G$.  
It is not hard to show that each point in $\mathbb{G}^E/\mathbb{G}^V$ can be represented uniquely by a weighting of $G$ with all edges in $F$ gauge-fixed to $1$.  Let
$\mathbb{G}^F$ denote the space of all such weightings.  Then the natural map  $\mathbb{G}^E/\mathbb{G}^V \rightarrow \mathbb{G}^F$ is an isomorphism.

If $G$ is a bridge graph, there is a natural specialization of edge weights, and we have a simple procedure for constructing the desired parametrization. 
Let 
\begin{math} \Pi^A_G=\Pi^A_{\langle u, w \rangle_k}\end{math}
and let 
\begin{math}d = \dim (\Pi^A_G).\end{math} 
 Assign a weight 
\begin{math} t_1,\ldots,t_d\end{math}
to each bridge, in the order the bridges were added, and set all other edge weights to $1$.  Begin with the $k \times n$ matrix in which the columns indexed by $u([k])$ form a copy of the identity, while the remaining columns contain only $0$'s.  Say the $r^{th}$ bridge is from $a_r$ to $b_r$ with $a_r < b_r$.  When we add the $r^{th}$ bridge to the graph, we multiply our matrix on the right by 
\begin{math} x^A_{(a_r,b_r)}(\pm t_r),\end{math}
the elementary matrix with nonzero entry $\pm t_r$ in row $a_r$ and column $b_r$.
The sign is negative if
\begin{math} |u([k]) \cap [a_r+1,b_r-1]|\end{math}
is odd, and positive if 
\begin{math} |u([k]) \cap [a_r+1,b_r-1]|\end{math}
is even.

Suppose $G$ and $G'$ are related to each other by one of the local moves defined in \ref{plabic}.  Then the maps $\partial_G$ and $\partial_{G'}$ are related by a birational change of variables.  For degree-two vertex removal, we may simply assume that both edges adjacent to the degree-two vertex are fixed to $1$, so the change of variables is trivial.  For the square move, assume all unlabeled edges in Figure \ref{squarecoords} are gauge-fixed to $1$.  Then we have the transformation shown in Figure \ref{squarecoords}, where
\[a' = \frac{a}{ac+bd},\:b' = \frac{b}{ac+bd},\;c' = \frac{c}{ac+bd},\;d' = \frac{d}{ac+bd}.\]

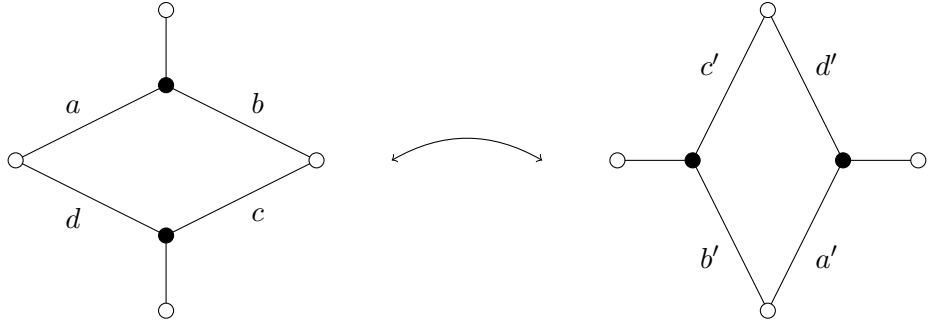
\begin{figure}[H]
\begin{tikzpicture}
\draw (-2,0) -- (0,1) -- (2,0) -- (0,-1) -- (-2,0);
\draw (0,1) -- (0,2);
\draw (0,-1) -- (0,-2);
\wdot{-2}{0};\wdot{2}{0};\bdot{0}{1};\bdot{0}{-1};\wdot{0}{2};\wdot{0}{-2};
\node[above left] at (-1,0.5) {$a$};
\node [below left] at (-1,-0.5) {$d$};
\node[above right] at (1,0.5) {$b$};
\node [below right] at (1,-0.5) {$c$};
\draw [<->] (3,0) [out = 30, in = 150] to (5,0);
\begin{scope}[xshift = 8 cm]
\draw (-1,0) -- (0,2) -- (1,0) -- (0,-2) -- (-1,0);
\draw (-2,0) -- (-1,0);
\draw (1,0) -- (2,0);
\wdot{-2}{0};\wdot{2}{0};\bdot{-1}{0};\bdot{1}{0};\wdot{0}{2};\wdot{0}{-2};
\node[above left] at (-0.5,1) {$c'$};
\node [below left] at (-0.5,-1) {$b'$};
\node[above right] at (0.5,1) {$d'$};
\node [below right] at (0.5,-1) {$a'$};
\end{scope}
\end{tikzpicture}
\caption{A square move}
\label{squarecoords}
\end{figure}

\subsection{Bridge graphs and MR parametrizations}

Let $\Pi^A$ be a positroid variety in $\Gr(k,n)$.  We have a family of MR parametrizations of $\Pi^A$, and another family of parametrizations of $\Pi^A$ arising from bridge graphs.  The following theorem, which is the main result of \citep{Kar14}, shows that these two families of parametrizations are essentially the same.  Note that in \cite{Kar14}, MR parametrizations were called ``Deodhar parametrizations.'' (See Remark \ref{changeterms}).

\begin{thm} Let $\Pi^A$ be a positroid variety.  For each MR parametrization of $\Pi^A$, there is a bridge graph which yields the same parametrization.  
Conversely, any parametrization of $\Pi^A$ which comes from a bridge graph agrees with some MR parametrization.
\end{thm}

In Section \ref{Cbridge}, we will use this result to define a class of networks which encode MR parametrizations for top-dimensional Deodhar components of projected Richardson varieties in $\Lm(2n)$.  In Section 5, we generalize these type C bridge graphs to construct a broader class of type C plabic graphs.  Here, we review the construction which yields a bridge graph for each MR parametrization.

Let \begin{math}\mf{w}=s_{i_1}s_{i_2}\cdots s_{i_m}\end{math} be a reduced word for $w \in S_n$, let \begin{math}u \leq_k w,\end{math} and let \begin{math}\mf{u}\preceq \mf{w}\end{math} be the PDS for  $u$ in $\mf{w}$.  
Let \begin{math} j_1,\ldots, j_d\end{math} be the indices of the simple transpositions of $\mf{w}$ which are not in $\mf{u}$.  For \begin{math}1 \leq r \leq d,\end{math} let 
\begin{equation}(a_r,b_r)=u_{(j_r-1)}(s_{i_{j_r}})u_{(j_r-1)}^{-1}.\end{equation}
Then we have
\begin{equation}(a_1,b_1)\cdots (a_d,b_d)=wu^{-1}.\end{equation}
Reversing the order of the transpositions gives
\begin{equation} \label{gens} (a_d,b_d)\cdots(a_2,b_2)(a_1,b_1)=uw^{-1}\end{equation}
and so we have
\begin{equation} t_{u([k])}(a_d,b_d)\cdots (a_1,b_1)=f_{u,w} \label{rightperm} \end{equation}

To construct the desired bridge graph, we successively add bridges 
\begin{displaymath}(a_d,b_d),\dots,(a_1,b_1)\end{displaymath}
to the lollipop graph corresponding to $u([k])$.  The hypotheses of Proposition \ref{canadd} are satisfied at each step, and so the result is a bridge graph $G$ with the desired bounded affine permutation.  For details, see \citep{Kar14}.

We note that the parametrization arising from $G$ is precisely the projected MR parametrization corresponding to 
\begin{math} \mf{u} \preceq \mf{w}.\end{math}  Indeed, recall that the matrices $\beta_r$ in \eqref{betas} are given by 
\begin{displaymath} \beta_r = y^A_{(a_r,b_r)}(\pm t_r)\end{displaymath}
for all \begin{math}1 \leq r \leq d.\end{math}  We project to the Grassmannian by taking each representative matrix $M \in \Sl(n)$ to the \emph{transpose} of the $n \times k$ matrix formed by its first $k$ columns.  Notice that taking transposes reverses the order of multiplication, and that $x_{\alpha}^A(t) = (y_{\alpha}^A(t))^T$ for each $\alpha$.  

Hence, we may construct both the bridge graph parametrization described above and the MR parametrization for $\mathcal{D}_{\mf{u},\mf{w}}$ by taking the $k \times n$ matrix which has a single non-zero Pl\"ucker coordinate $\Delta_{u([k])}$ and multiplying on the right by a sequence of factors
\[x^A_{(a_d,b_d)}(\pm t_d)\cdots x^A_{(a_1,b_1)}(\pm t_1)\]
For each $1 \leq r \leq d$, the sign of the parameter $t_r$ is negative if $|u([k]) \cap [a_r+1,b_r-1]|$ is odd, and positive otherwise, so the parametrizations are the same.  

\subsection{\protect\Le-diagrams}

A \Le-\emph{diagram} is a Young diagram filled with $0$'s and $+$'s according to certain rules.  (The symbol \Le\,
is pronounced ``le,'' which is ``ell'' backwards.)  Postnikov showed that \Le-diagrams are in bijection with PDS's of Grassmannian permutations, and constructed a plabic graph corresponding to each \Le-diagram \citep{Pos06}. Hence \Le-diagrams provide the first example of the relationship between planar graphs and PDS's.  Lam and Williams defined analogs of \Le-diagrams for all \emph{cominiscule Grassmannians} \citep{LW07}.  Since their construction depends only on Weyl group data, the \emph{type $B$ \Le-diagrams} of Lam and Williams index projected Richardson varieties in $\Lm(2n)$.
Although \Le-diagrams will not play a major role in our results, we review them here for completeness.

To construct a \Le-diagram, begin with a $k \times (n-k)$ rectangle, subdivided into unit boxes. Order the boxes as shown in Figure \ref{le}.  The lower order ideals with respect to this ordering are precisely the \emph{Young diagrams} that fit inside a $k \times (n-k)$ rectangle. (Note that our Young diagrams are drawn using French notation, with minimal elements at the bottom left corner; Postnikov uses the English notation, with minimal elements at the top left.)

We label each box with a simple transpositions $S_n$ as shown in the figure; boxes on the same diagonal are labeled with the same transposition.  
With these conventions, Young diagrams that fit inside a $k \times (n-k)$ rectangle are in bijection with Grassmannian permutations of type $(k,n)$.  The bijection is as follows.  The boxes in a Young diagram give a sequence of simple transpositions in $S_n$.  We read off the boxes in increasing order, and list the corresponding transpositions from right to left.  This process yields a reduced word for some Grassmannian permutation $w_Y \in S_n,$ and the map $Y \rightarrow w_Y$ is a bijection. 
See Figure \ref{le} for an example.

\begin{figure}[ht]
\centering
\begin{tikzpicture}
\draw[step = 0.5 cm] (0,0) grid (1.5,1);
\node at (0.25,0.25) {1};
\node at (0.75,0.25) {2};
\node at (1.25,0.25) {3};
\node at (0.25,0.75) {4};
\node at (0.75,0.75) {5};
\node at (1.25,0.75) {6};
\begin{scope}[xshift = 2 cm]
\draw[step = 0.5 cm] (0,0) grid (1.5,1);
\node at (0.25,0.25) {$s_2$};
\node at (0.75,0.25) {$s_3$};
\node at (1.25,0.25) {$s_4$};
\node at (0.25,0.75) {$s_1$};
\node at (0.75,0.75) {$s_2$};
\node at (1.25,0.75) {$s_3$};
\end{scope}
\begin{scope}[xshift = 5 cm]
\draw[step = 0.5 cm] (0,0) grid (1,1);
\draw (1,0) -- (1.5,0) -- (1.5,0.5) -- (1,0.5);
\node at (0.25,0.25) {$0$};
\node at (0.75,0.25) {$0$};
\node at (1.25,0.25) {$0$};
\node at (0.25,0.75) {$0$};
\node at (0.75,0.75) {$0$};
\end{scope}
\begin{scope}[xshift = 7 cm]
\draw[step = 0.5 cm] (0,0) grid (1,1);
\draw (1,0) -- (1.5,0) -- (1.5,0.5) -- (1,0.5);
\node at (0.25,0.25) {$+$};
\node at (0.75,0.25) {$0$};
\node at (1.25,0.25) {$+$};
\node at (0.25,0.75) {$0$};
\node at (0.75,0.75) {$0$};
\end{scope}
\end{tikzpicture}
\caption{Constructing \protect\Le-diagrams for the case $k=2,$ $n=5$.  The diagrams at right correspond to expressions $s_2s_1s_4s_3s_2$ and $s_2s_11s_31$, respectively}
\label{le}
\end{figure}
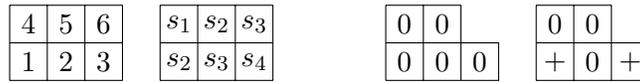 
 
Let $Y$ be a Young diagram, and let $\mf{w}_Y$ be the corresponding reduced word for $w_Y$.  We represent subexpressions of $\mf{w}_Y$ using \emph{$\oplus$-diagrams} of shape $Y$, that is, fillings of $Y$ with $0$'s and $+$'s. To construct the subexpression corresponding to an $\oplus$-diagram of shape $Y$, we replace each factor of $\mf{w}_Y$ corresponding to a box containing $+$ by the identity permutation $1$.  An $\oplus$-diagram is a \emph{\Le-diagram} if it corresponds to a PDS. 

Recall that positroid varieties are indexed by pairs $u \leq w$ with $w$ Grassmannian.  Further, if $u \leq w$, and $\mf{w}$ is a reduced word of $w$, there is a unique PDS for $u$ in $\mf{w}$.  It follows that \Le-diagrams are in bijection with positroid varieties.  Moreover, \Le-diagrams are characterized by a simple pattern-avoidance condition.  

\begin{thm}[\cite{Pos06}]
Let $Y$ be a Young diagram.  An $\oplus$-diagram of shape $Y$ is a \Le-diagram if no $0$ has both a $+$ in the same row to its left, and a $+$ in the same column below it.
\end{thm}

Note that with Postnikov's conventions, the three boxes in a forbidden pattern forms a backwards $L$ shape, which is the source of the name $\Le$-diagram.  

The type B construction is analogous.  The $k \times (n-k)$ rectangle is replaced with a staircase shape of size $n$, with boxes ordered and labeled as shown in Figure \ref{lec}.  Lower order ideals in the staircase shape give \emph{type $B$ Young diagrams}.  They correspond to reduced words for elements of $S_n^C$ which are minimal-length left coset representatives for $S_n^C/(S_n^C)_n$.  Let $Y$ be a type $B$ Young diagram.  As in the type $A$ case, the $\oplus$-diagrams of shape $Y$ correspond to subexpressions of some reduced word; such a diagram is a \emph{type $B$ \Le-diagram} if it represents a PDS.  Once again, we have a characterization of \Le-diagrams in terms of pattern avoidance.

\begin{figure}[ht]
\centering
\begin{tikzpicture}
\draw[step = 0.5 cm] (0,0) grid (1,1);
\draw (0,0) -- (0,-0.5) -- (0.5,-0.5) -- (0.5,0);
\draw (1,1) -- (1.5,1) -- (1.5,0.5) -- (1,0.5);
\node at (0.25,0.25) {2};
\node at (0.75,0.25) {3};
\node at (0.25,0.75) {4};
\node at (0.75,0.75) {5};
\node at (1.25,0.75) {6};
\node at (0.25,-0.25) {1};
\begin{scope}[xshift = 2 cm]
\draw[step = 0.5 cm] (0,0) grid (1,1);
\draw (0,0) -- (0,-0.5) -- (0.5,-0.5) -- (0.5,0);
\draw (1,1) -- (1.5,1) -- (1.5,0.5) -- (1,0.5);
\node at (0.25,0.25) {$s_2$};
\node at (0.75,0.25) {$s_3$};
\node at (0.25,0.75) {$s_1$};
\node at (0.75,0.75) {$s_2$};
\node at (1.25,0.75) {$s_3$};
\node at (0.25,-0.25) {$s_3$};
\end{scope}
\end{tikzpicture}
\caption{We construct type $B$ \protect\Le-diagrams inside a staircase shape.  The figure shows $n=3$.}
\label{lec}
\end{figure}
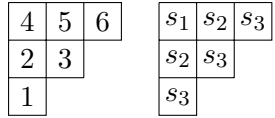

\begin{thm}[\cite{LW07}]
Let $Y$ be a type-$B$ Young diagram. An $\oplus$-diagram of type $B$ is a \emph{type $B$ \Le-diagram} if the following conditions hold:
\begin{enumerate}
\item No $0$ has both a $+$ in the same row to its left, and a $+$ in the same column below it.
\item No $0$ which lies in a diagonal box has a $+$ in the same row to its left.
\end{enumerate}
\end{thm}

\section{Bounded affine permutations and Bruhat intervals in type $C$}
\label{bound}

By \citep[Corollary 6.4]{Wil07}, the poset $\Q^C(2n)$ is graded, with rank function 
\[\rho(\langle u, w \rangle_n^C) = \frac{n(n+1)}{2} - (\ell^C(w)-\ell(^Cu)).\]  
Recall that $\Pi_{u,w}^C \subseteq \Lm(2n)$ has dimension $\ell^C(w)-\ell^C(u)$ \citep{KLS14}.  Since the dimension of $\Lm(2n)$ is $\frac{n(n+1)}{2}$ it follows that $\Pi_{u,w}^C$ has codimension $\frac{n(n+1)}{2}-(\ell^C(w)-\ell^C(u))$.  Hence the rank of an element of $\Q^C(2n)$ is the \emph{codimension} of the corresponding projected Richardson variety in $\Lm(2n)$.

Our goal is to show that $\Q^C(2n)$ is isomorphic to a subposet of $\Q(n,2n)$.  The image of $(S_n^C)_n$ in $S_{2n}$ is the group of permutations in $S_n \times S_n$ which satisfy Equation \ref{signed}.  We note that $w \in S_n^C$ is a left coset representative of $(S_n^C)_n$ of minimal (respectively, maximal) length if and only if $w$ is a left coset representative of $S_n \times S_n$ in $S_{2n}$ of minimal (respectively, maximal) length.  This follows easily from the discussion in \citep[Chapter 8]{BB05}.

\begin{prop}
Let $u,w \in S_n^C$, where we view $S_n^C$ as a subgroup of $S_{2n}$.  Then $u \leq_n^C w$ in $S_n^C$ if and only if $u \leq_n w$ in $S_{2n}$. 
\end{prop}

\begin{proof}For the forward direction, it is enough to show this when $u \lessdot_n^C w$ in $S_n^C$.  There are two cases to consider.  Recall that $a' = 2n+1-a$ for $a \in [2n]$.  Either $w = us^A_{(a,a')}$ for some $a \in [n]$, or $w = us^A_{(a,b)}s^A_{(b',a')}$ for some $a \in [n]$, $b \in [n+1,2n]$.  In the first case, the claim is immediate; in the second, it is easy to see that
\[u \lessdot_n us^A_{(a,b)} \lessdot_n us^A_{(a,b)}s^A_{(b',a')}\]
so $u \leq_n w$ as desired.

For the reverse direction, let $u,w \in S_n^C$, and suppose $u \leq_n w \in S_{2n}$. Then $w$ factors uniquely as $w^nw_n$ where $w_n \in S_n \times S_n$ and $w^n$ is a minimal-length coset representative of $S_n \times S_n$ in $S_{2n}$.  Moreover, we must have $w^n,w_n \in S_n^C$.
Let $u' = uw_n^{-1}$.  Then by Lemma \ref{stillless} we have $u' \leq _n w^n$, and $\langle u',w^n\rangle_n = \langle u,w \rangle_n$.
Now $w^n$ is a minimal-length coset representative for $S_n \times S_n$  in $S_{2n}$, and hence a minimal-length coset representative for $(S_n^C)^n$ in $S_n^C$.
Thus by Lemma \ref{minrep} we have $u' \leq_n^C w^n$.
Moreover, since the inclusion of $S_n^C$ into $S_{2n}$ is a Bruhat embedding, the factorizations 
$u = u'w_n$ and $w = w^nw_n$ are both length-additive in $S_n^C$.  Hence $u \leq_n^C w$ as desired.  
\end{proof}

Again, let $u,w,x,y \in S_n^C$.  It is clear that $\langle u,w \rangle_n = \langle x,y \rangle_n$ if and only if $\langle u,w \rangle_n^C = \langle x,y \rangle_n^C$.
Hence, $\Q^C(2n)$ is a subset of $\Q(n,2n)$.  Moreover, it is easy to see that if $\langle u, w \rangle_n^C \leq \langle x,y \rangle_n^C$, then $\langle u,w \rangle_n \leq \langle x,y \rangle_n$.
We claim that the converse holds, so that $\Q^C(2n)$ embeds as a sub-poset of $\Q(n,2n)$.

\begin{prop}\label{closures}
Suppose $\langle u,w \rangle_n \leq \langle x,y \rangle_n$ with $u,w,x,y \in S_n^C$.  Then $\langle u,w \rangle_n^C \leq \langle x,y \rangle_n^C$.   
\end{prop}

\begin{proof}
We may assume for simplicity that $w$ is Grassmannian.  Since $\langle u,w \rangle_n \leq \langle x,y \rangle_n$ we have some representative $[x',y']_n$ of $\langle x,y \rangle_n$ such that 
\[u \leq x' \leq_n y' \leq w.\]
It suffices to show that we can choose $x',y' \in S_n^C$.  Write $y = y^ny_n$ where $y^n$ is Grassmannian and $y_n \in S_n \times S_n.$
Note that $y_n = t s$ where $t$ fixes $[n]$ and $s$ fixes $[n+1,2n]$.  
Let $s'$ be a permutation in $S_n \times S_n$ which fixes $[n]$, such that $s s' \in S_n^C$.  Since $w$ is Grassmannian, for $v,w \in S_{2n}$, we have $v \leq w$ if and only if $v(i) \leq w(i)$ for $i \in [n]$, while $v(i) \geq w(i)$ for $i \in [n+1,2n]$.  By symmetry, since $w \in S_n^C$, we have $y^* = y^n s s' \leq w \in S_{2n}$.  Moreover, since $x \leq_n y$ with $x \in S_n^C$, we have $x = x^n r r' ss'$ where $r$ fixes $[n+1,2n]$, $r'$ fixes $[n]$, $rr' \in S_n^C$, and $rs$ is a length-additive factorization.  Let 
\[x^* = x^n rr' ss'\]
Then $x^* \leq_n y^*$, and we have
\[\langle x^*, y^*\rangle_n = \langle x,y\rangle_n.\]

It remains to show that $u \leq x^*$. Let $\mf{x}$ be a reduced word for $x$ in $S_{2n}$ whose leftmost portion corresponds to a reduced word $\mf{q}$ of $x_n rr'$ under the embedding $S_n^C \hookrightarrow S_{2n}$.  Consider the lexicographically leftmost subexpression $\mf{u}$ for $u$ in $\mf{x}$.  As in the construction of PDS's, we choose the factors of $\mf{u}$ greedily, working from the left.  Since $\mf{q}$ corresponds to a reduced word in $S_n^C$, it follows that the portion of $\mf{u}$ which is contained in $\mf{q}$ does also.  The remaining factors in $\mf{u}$ multiply to give some element of $z \in (S_n^C)_n$ such that $st$ contains a reduced word for $z$ in $S_{2n}$.  Let $\mf{s}$ be reduced word for $s$, and $\mf{s'}$ a reduced word for $s'$.  Then there is a reduced subexpression for $z$ in the reduced word $\mf{ss'}$.  Hence we can find a reduced word for $u$ in the reduced word $\mf{qss'}$ of $x^*.$  This completes the proof.
\end{proof}

\begin{cor}The poset $\mathcal{Q}^C(2n)$ embeds as a sub-poset of $\mathcal{Q}(n,2n)$.
\end{cor}

These results give a compact description of projected Richardson varieties in $\Lm(2n)$.

\begin{prop}
Let $\mathring{\Pi}_{u,w}^C$ be an open projected Richardson variety in $\Lm(2n)$, where $u, w \in S_n^C$ with $u \leq_n^C w$.  Then $u \leq_n w \in S_{2n}$ and set-theoretically we have 
\[\mathring{\Pi}_{u,w}^C = \mathring{\Pi}_{u,w}^A \cap \Lm(2n)\]
where $\mathring{\Pi}_{u,w}^A$ is the open positroid variety corresponding to $\langle u,w \rangle_n$ in $\Gr(n,2n)$.

The set-theoretic intersection of an open positroid variety $\mathring{\Pi}_{u,w}^A$ with $\Lm(2n)$ is empty unless $\langle u,w \rangle_n$ has a representative $\langle u',w'\rangle_n$ where $u',w' \in S_n^C$, in which case the intersection is $\mathring{\Pi}_{u',w'}^C$.

Finally, the closure partial order on projected Richardson varieties in $\Lm(2n)$ is induced by the closure partial order on positroid varieties in the obvious way.  In particular, for $u \leq_n^C w$, we have
\[\Pi_{u,w}^C = \Pi_{u,w}^A \cap \Lm(2n).\]

\label{nicepos} 
\end{prop}

\begin{proof}
Let $\langle u,w \rangle_n^C \in \Q^C(2n)$.  Recall that we have
\[\mathring{R}_{u,w}^C= \mathring{R}_{u,w}^A \cap \Sp(2n)/B_+^{\sigma}\]
since $u \leq w \in S_n^C$.  Since $u \leq_n^C w$, we have $u \leq_n w$, and the projection $\pi_n:\mathcal{F}\ell(2n) \rightarrow \Gr(n,2n)$ carries $\mathring{R}_{u,w}$ isomorphically to $\mathring{\Pi}_{u,w}^A$.  It follows that $\mathring{\Pi}_{u,w}^C \subseteq \mathring{\Pi}_{u,w}^A \cap \Lm(2n)$.

Since open projected Richardson varieties stratify $\Lm(2n)$, while open positroid varieties stratify $\Gr(n,2n)$, we have
\[\Lm(2n) = \bigsqcup_{\langle u,w \rangle_n^C \in \mathcal{Q}^C(2n)} \mathring{\Pi}_{u,w}^A \cap \Lm(2n)\]
which in turn implies
\[\mathring{\Pi}^C_{u,w} = \mathring{\Pi}^A_{u,w} \cap \Lm(2n)\]
for all $\langle u,w \rangle_n^C$. 

Since the open positroid varieties of the form $\mathring{\Pi}^A_{u,w}$ for $\langle u,w \rangle_n^C$ cover $\Lm(2n)$, it follows that $\mathring{\Pi}^A_{x,y} \cap \Lm(2n)$ is empty if $\langle x,y \rangle_n$ is not contained in the image of the embedding $\mathcal{Q}^C(2n) \hookrightarrow \mathcal{Q}(n,2n)$.

The final statement follows from Lemma \ref{closures}, since the partial orders on $\Q^C(2n)$ and $\Q(n,2n)$ give the reverse of the closure partial orders on projected Richardson varieties in $\Lm(2n)$ and $\Gr(n,2n)$, respectively.
\end{proof}

We now define a type $C$ analog of the poset $\Bd(k,n)$.  Recall that we have a well-defined isomorphism $\Q(k,n) \rightarrow \Bd(k,n)$ given by $\langle u, w \rangle_k \mapsto f_{\langle u,w \rangle_k}$ where 
\[f_{\langle u,w \rangle_k} = ut_{[k]}w^{-1}.\]

\begin{defn}
The set $\Bd^C(2n)$ of \emph{type C bounded affine permutations} is the image of $\Q^C(2n)$ under the map $\mathcal{Q}(n,2n) \rightarrow \Bd(n,2n)$.
\end{defn}

Recall that bounded affine permutations are elements of the extended affine Weyl group of $\Gl(n)$.  We show that a similar statement holds for $\Bd^C(2n)$.  A \emph{symplectic similitude} $A \in \Gl(2n)$ is a linear transformation such that for all $v,w \in \mathbb{C}^{2n}$ we have
\[\langle Av, Aw \rangle = \mu \langle v,w \rangle\]
for a fixed nonzero scalar $\mu$. Let $\text{GSp}(2n)$ denote the group of symplectic similitudes, which is a reductive group of type $C_n$.

 The extended affine Weyl group $\widetilde{S}_n^C$ of $\text{GSp}(2n)$ may be realized as a subgroup of $\widetilde{S}_{2n}$, and the inclusion is a Bruhat embedding.  Concretely, the extended affine Weyl group of $\text{GSp}(2n)$ consists of all affine permutations $wt$, with $w \in S_n^C$ and $t = (a_1,\ldots,a_{2n})$ a translation element satisfying 
\[a_i+a_{i'} = a_j + a_{j'}\]
for all $1 \leq i,j \leq n$.
For details, see \citep{KR99}.
In particular, the Bruhat order on $\widetilde{S}_n^C$ induces a partial order on $\Bd^C(2n)$ which agrees with the partial order inherited from $\Bd(n,2n)$.

Each element of $f \in \Bd^C(2n)$ satisfies
\[f(2n+a-1) = 4n+1 - f(a)\]
for all $a \in [2n]$.
We claim that every $f \in \Bd(n,2n)$ satisfying this condition must in fact be contained in $\Bd^C(2n)$.  
For this, it suffices to show that each such bounded affine permutation has the form $f = ut_{[n]}w^{-1}$
where $w$ is Grassmannian, and $u,w \in S_{n}^C$.
Certainly, we know that $f = ut_{[n]}w^{-1}$ for some $u,w \in S_{2n}$, with $w$ Grassmannian.  Since
\[f(2n+a-1) = 4n+1 - f(a)\]
for each $a,$ exactly one of each pair $\{a,a'\}$ must be in $w^{-1}[n]$.  From this, it follows that $w \in S_{n}^C$.  But this, in turn, forces $u \in S_{n}^C$, since
$uw^{-1}(a') = uw^{-1}(a)'$
for all $a \in [n]$. 
We now have the following.

\begin{prop}
The poset $\Bd^C(2n)$ consists of all bounded affine permutations $f \in \Bd(n,2n)$ which satisfy
\[f(2n+a-1) = 4n+1 -f(a)\]
for all $a \in [2n]$.
\end{prop}

Let $\ell^{\widetilde{C}}$ denote the Bruhat order on $S_n^{\widetilde{C}}$.  For $f \in \widetilde{S}_{2n}$, we define an equivalence relation on inversions of $f$ by setting two inversions $(a,b)$ and $(c,d)$ equivalent if either 
\[(c,d) = (a+2rn,b+2rn)\]
for some $r \in \mathbb{Z}$ or 
\[(c,d) = (2n+1-a,2n+1-b).\]  We call the resulting equivalence classes \emph{type $\widetilde{C}$ inversions.}  The following is an immediate consequence of the discussion in \citep[Chapter 8]{BB05}.  

\begin{prop}
Let $f$ be a bounded affine permutation in $\Bd^C(2n)$.  Then $\ell^{\widetilde{C}}(f)$ is the number of type $\widetilde{C}$ inversions of $f$.
Alternatively $\ell^{\widetilde{C}}(f)$ is the number of type $\widetilde{A}$ inversions of $f$ which have a representative of one of the following forms, for $n+1 \leq i \leq j \leq 2n$, and $r$ a positive integer: $(i,j)$, $(i',j)$, $(i,j+2rn)$, $(i,j'+2rn)$, or $(j',i+2rn)$.
\end{prop}

We note the type $C$ analog of Theorem 3.16 from \citep{KLS14}.  The proof is entirely analogous to the type $A$ version.

\begin{prop}
If $f  = ut_{[n]}w^{-1} \in \Bd^C(2n)$, then $f$ has length 
\[\frac{n(n+1)}{2} - (\ell^C(w) - \ell^C(u))\]
so the bijection from $\Q^C(2n)$ to $\Bd^C(2n)$ is graded.
The codimension of $\Pi_f^C$ in $\Lm(2n)$ is equal to $\ell^{\widetilde{C}}(f).$

\end{prop}

\begin{rmk}
He and Lam gave a construction which yields analogs of $\Bd(k,n)$ for many partial flag varieties \cite{HL11}.  We focus here on a special case, which suffices for our purposes.  Let $G$ be a quasi-simple reductive group with Weyl group $(W,S)$ and extended affine Weyl group $\widehat{W}$.  We may assume that $G$ is \emph{adjoint}, so that $\widehat{W}$ is as large as possible. In particular, for each simple root $\alpha_i$ of $G$, the cocharacter lattice of $G$ contains a \emph{fundamental coweight} $\lm_i$ which is dual to $\alpha_i$.

Let $\alpha_i$ be a simple root of $W$.   Let $J =  S \backslash \{\alpha_i\}$, let $W^J$ denote the set of minimal-length representatives of $W/W_J$, and let $\lm_i$ be the fundamental coweight which is dual to $\alpha_i$.  Let $P_J$ be the parabolic subgroup of $G$ corresponding to $J$.  Then the desired indexing set for projected Richardson varieties in $G/P_J$ is given by
\[\mathcal{Q} = \{ut^{-\lm_i}w^{-1} \mid u \in W,\,w\in W^J\text{ and }u\leq w\} \subseteq \widehat{W}\]
where $t^{-\lm_i}$ is the \emph{translation element} corresponding to $-\lm_i$. 

Since the pairs $(u,w)$ listed above are a complete set of representatives for the set of $W_J$-Bruhat intervals $\langle u,w \rangle_J$ in $W$, the set $\mathcal{Q}$ has a poset structure inherited from the poset of $W_J$-Bruhat intervals.  Moreover, this coincides with the partial order induced on $\mathcal{Q}$ by the Bruhat order on $\widehat{W}$.  The poset $\mathcal{Q}$ is graded by the length function on $\widehat{W}$, and the length of an element $ut^{-\lm_i}w^{-1}$ of $\mathcal{Q}$ is equal to the codimension of the projected Richardson variety corresponding to $\langle u,w\rangle_J$ \cite{HL11}.

Note that we realize $\Bd(k,n)$ as a subset of the extended affine Weyl group of $\Gl(n)$, and $\Bd^C(2n)$ as a subset of the extended affine Weyl group of $\GSp(2n)$.  Since $\Gl(n)$ and $\GSp(2n)$ are not quasi-simple, we cannot apply the results of \cite{HL11} in this setting. Rather, we must look at the extended affine Weyl groups of the adjoint quasi-simple Lie groups of types $A_{n-1}$ and $C_{n}$, respectively.

The adjoint Lie group of type $A_{n-1}$ is $\PSL(n)$, the quotient of $\Sl(n)$ by the subgroup of scalar matrices.  Similarly, the adjoint Lie group $\PSp(2n)$ of type $C_n$ is the quotient of $\Sp(2n)$ by the subgroup of symplectic scalar matrices.  Taking the quotient of $\PSL(n)$ by the image of the subgroup of upper-triangular matrices gives $\mathcal{F}\ell(n)$, and similarly for $\PSp(2n)$.  Let $\widehat{W}_n^A$ denote the extended affine Weyl group of $\PSL(n)$, and $\widehat{W}_n^C$ denote the extended affine Weyl group of $\PSp(2n)$.

It is not hard to show that He and Lam's result applied to $G = \PSL(n)$ and the simple root $\alpha_k$ gives the desired isomorphism of graded posets between $\Q(k,n)$ and $\Bd(k,n)$.  The translation element $t_{[k]}$ in $\widetilde{S}_n^A$ plays an analogous role to the translation element $t^{-\lm_k}$ in $\widehat{W}_n^A,$ where $\lm_k$ is the fundamental coweight which is dual to $\alpha_k$.  The situation is similar in type $C$.  Here $t_{[n]}$ in $\widetilde{S}^C_n$ plays an analogous role to the translation element $t^{-\lm_n}$ in $\widehat{W}_{n}^C$, where $\lm_n$ is the fundamental coweight dual to $\alpha_n$.
\end{rmk}

\begin{rmk}
In \citep{LW07}, the authors introduced the poset of \emph{type $B$ decorated permutations} of order $2n$, denoted $\mathcal{D}_n^B$, and showed that it indexes projected Richardson varieties in both the odd orthogonal Grassmannian $\text{OG}(n,2n+1)$ (a flag variety of type $B_n$) and the Lagrangian Grassmannian $\Lm(2n)$. The correspondence between decorated permutations and bounded affine permutations maps $\mathcal{D}_n^B$ isomorphically to $\Bd^C(2n)$. Lam and Williams also give an isomorphism from $\Q^C(2n)$ to $\mathcal{D}_n^B$ which is equivalent to our isomorphism from $\Q^C(2n)$ to $\Bd^C(2n)$.  What is new in the present paper is the realization of $\Q^C(2n)$ as an induced subposet $\Q(n,2n)$, and the description of $\Bd^C(2n)$ in terms of $\widetilde{S}_n^C$.
\end{rmk}

\section{Bridge graphs and MR parametrizations for $\Lm(2n)$}

\label{Cbridge}

We now define bridge graphs for projected Richardson varieties in $\Lm(2n)$, and show that they encode MR parametrizations. 
Recall that $a' = 2n+1-a$ for all $a \in [2n]$.  Let $f \in \Bd^C(2n)$.
Suppose for some $a \in [n]$, we have 
\begin{enumerate}
\item $f(a') > f(a).$
\item Every $c \in [a+1,a'-1]$ is a fixed point of $f.$
\end{enumerate}
Then $f$ has a \emph{symmetric bridge} at $(a,a')$.  
Alternatively, for $a,b \in [n]$ with $a < b$, we have
\begin{enumerate}
\item $f(a) > f(b)$ and  $f(b') > f(a').$
\item Every $c \in [a+1,b-1] \cup [b'+1,a']$ is a fixed point of $f$.
\end{enumerate}
Then we say $f$ has a \emph{symmetric pair of bridges} at $(a,b)$ and $(b',a')$.

Now suppose that for $f,g \in \Bd^C(2n)$, we have $g = fs_{(a,a')}$ where $f$ has a symmetric bridge at $(a,a')$.  Then $g <  f$ in the Bruhat order on $\Bd(n,2n)$, and hence in the Bruhat order order on $\Bd^C(2n).$  It follows from \citep[Proposition 8.4.1]{BB05} that in fact $g \lessdot f$.  Similarly, if $g = fs_{(a,b)}s_{(b',a')}$ where $f$ has a symmetric pair of bridges at $(a,b)$ and $(b',a')$, then $g \lessdot f$ in the Bruhat order on $\Bd^C(2n)$.  

Let $f = ut_{[n]}w^{-1}$, where $u,w \in S_n^C$.  Then a \emph{symmetric bridge graph} for $f \in \Bd^C(n,2n)$ is a graph obtained by starting with the symmetric lollipop graph corresponding to $t_{u[n]}$ and repeatedly adding either symmetric bridges $(a,a')$ or symmetric pairs of bridges $(a,b)(b',a')$ until we obtain a graph with bounded affine permutation $f$.  See Figure \ref{symbrd} for an example.  It is perhaps not obvious that each $f \in \Bd^C(n)$ corresponds to a symmetric bridge decomposition.  However, this will follow from our results.

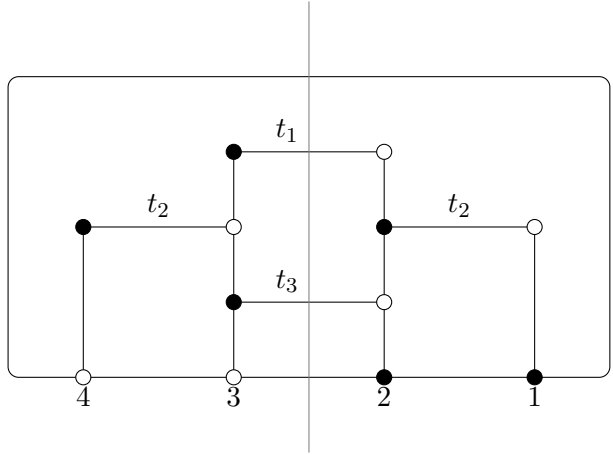
\begin{figure}[H] 
\centering
\begin{tikzpicture}
\draw [rounded corners] (0,0) rectangle (8,4);
\draw (3,0) -- (3,3) -- (5,3) -- (5,0);
\draw (1,0) -- (1,2) -- (3,2);
\draw (5,2) -- (7,2) -- (7,0);
\draw (3,1) -- (5,1);
\bdot{1}{2};\bdot{3}{3};\bdot{5}{2};\bdot{3}{1};
\wdot{3}{2};\wdot{5}{3};\wdot{5}{1};\wdot{7}{2};
\wdot{1}{0};\wdot{3}{0};\bdot{5}{0};\bdot{7}{0};
\node [below] at (1,0) {$4$};
\node [below] at (3,0) {$3$};
\node [below] at (5,0) {$2$};
\node [below] at (7,0) {$1$};
\draw [gray, thin] (4,5) -- (4,-1);
\node [above left] at (4,3) {$t_1$};
\node [above] at (2,2) {$t_2$};
\node [above] at (6,2) {$t_2$};
\node [above left] at (4,1) {$t_3$};
\end{tikzpicture}
\caption{A symmetric bridge graph with symmetric weights.}
\label{symbrd}
\end{figure}

Let $B$ be a symmetric bridge graph for $f$.  Weight each symmetric bridge $(a_i,b_i)$ with an indeterminate $t_i$.  For each symmetric pair of bridges, weight both bridges $(a_j,b_j)$ and $(b_j',a_j')$ with the same indeterminate $t_j$.  Applying the boundary measurement map, we get a parametrization of a locally closed subset of $\mathring{\Pi}_A^f$.  We claim that the image lies in $\Lm(2n)$.  In fact, we prove something much stronger.  The following is a Lagrangian analog of the main theorem of \citep{Kar14}. 

\begin{prop}Let $f=ut_{[n]}w^{-1} \in \Bd^C(2n)$, where $u \leq_n^C w$.  Then every MR parametrization for $\mathring{\Pi}^C_f$ corresponds to a parametrization arising from a symmetric bridge graph.  Conversely, every parametrization of $\mathring{\Pi}^C_f$ from a symmetric bridge graph arises from some MR parametrization.
\end{prop}

\begin{proof}

The proof is nearly identical to the type $A$ version.  We sketch the argument here, and refer the reader to \citep{Kar14} for details. Let $\widetilde{\mf{w}}$ be a reduced word for $w$ in $S_n^C$, and let $\widetilde{\mf{u}} \preceq \widetilde{\mf{w}}$ be the PDS for $u$ in $\widetilde{\mf{w}}$.  Let $\mf{u} \preceq \mf{w}$ be the corresponding PDS in $S_{2n}$, which is unique up to commutation moves.  Note that $u \leq_n w$ as elements of $S_{2n}$.  Hence $\mf{u} \preceq \mf{w}$ corresponds to a unique bridge graph; labeling the bridges with parameters gives the MR parametrization of $\mathcal{D}^A_{\mf{u},\mf{w}}$ corresponding to $\mf{u} \preceq \mf{w}$.  As in the proof of Lemma \ref{deoAC}, setting the weights on the two bridges in each symmetric pair equal to each other gives a parametrization of $\mathcal{D}^C_{\widetilde{\mf{u}},\widetilde{\mf{w}}}$, which we view as a parametrization of $\mathring{\Pi}_f^C$.

As an aside, we note that the planarity of the resulting graph puts restrictions on the sequence of factors $x^C_{\alpha}$ which may appear in the parametrization corresponding to a PDS $\mf{u} \preceq \mf{w}$ when the parametrization is written as in Equation \ref{betas}.  In particular, the only roots $\alpha$ which appear are of the form
\[\epsilon_j - \epsilon_i= - \alpha_i - \alpha_{i+1} - \cdots - \alpha_{j-1}\]
for $1 \leq i < j \leq n$, and those of the form
\[ -2\epsilon_{i} = -2\alpha_i - 2\alpha_{i+1} - \cdots - \alpha_n\]
for $1 \leq i \leq n$.
Indeed, a factor $x_{\alpha}$ for $\alpha = \epsilon_i + \epsilon_j$ would correspond to a pair of bridges $(i,j')$ and $(j,i')$, contradicting the planarity of the bridge graph.
 
For the reverse direction, it is enough to show that every symmetric bridge graph corresponds to a so-called \emph{bridge diagram}, defined in \citep{Kar14}, which is symmetric up to isotopy with respect to reflection through the horizontal axis.  The proof is entirely analogous to the type $A$ case.  As in type $A$, we build the desired bridge diagrams iteratively, by either adding bridges or lollipops; in the type $C$ case, to maintain the symmetry of our diagrams, we always add lollipops in symmetric black-white pairs.  
\end{proof}

We say a symmetric bridge graph has \emph{symmetric weights} if whenever two bridges form a symmetric pair, their weights are equal.  So the above result says that symmetric bridge graphs with symmetric weights give parametrizations of projected Richardson varieties in $\Lm(2n)$.  Let $B$ be a symmetric bridge graph with symmetric weights, corresponding to $\mathring{\Pi}_B^C$.  By the properties of MR parametrizations, restricting the weights of the bridges in $B$ to $\mathbb{R}^+$ gives a parametrization of the totally nonnegative part of $\mathring{\Pi}_B^C$.

\section{The Lagrangian boundary measurement map}

\label{Cbdry}

\subsection{Symmetric plabic graphs.}

We now define symmetric plabic graphs, first introduced in \citep{KS15}.  Just as ordinary plabic graphs yield parametrizations of positroid varieties, symmetric plabic graphs give parametrizations of projected Richardson varieties in $\Lm(2n)$.

\begin{rmk}Throughout this paper, we require plabic graphs to be bipartite.  Postnikov's original definition allows plabic graphs which are not bipartite; however, these graphs can be made bipartite by either contracting unicolor edges, or adding degree-$2$ vertices.  Our construction of the boundary measurement map only applies to bipartite graphs; the general case requires a different construction, given in \citep{Pos06}.  We believe that our results extend naturally to the non-bipartite case, but have not checked the details.
\end{rmk}

\begin{defn}A symmetric plabic graph $G$ is a plabic graph with $2n$ boundary vertices, which has a distinguished diameter $d$ such that the following hold:
\begin{enumerate}
\item The diameter $d$ has one endpoint between boundary vertices $2n$ and $1$, and the other between $n$ and $n+1$.
\item No vertex of $G$ lies on $d$, although some edges may cross $d$.
\item Reflecting the graph $G$ through the diameter $d$ gives a graph $G'$ which is identical to $G$, but with the colors of vertices reversed.
\end{enumerate}
\end{defn}

See Figure \ref{graph} for an example.
The following is an immediate consequence of Theorem 3.1 from \citep{KS15}.

\begin{lem}
Let $G$ be a symmetric plabic graph.  Then $f_G \in \Bd^C(2n)$.  Conversely, for every $g \in \Bd^C(2n)$, there is a symmetric plabic graph $G$ such that $f_G = g$.
\end{lem}

Let $G$ be a symmetric plabic graph, with vertex set $V$.  We define a map $r:V \rightarrow V$ which maps each vertex $v \in V$ to its image under reflection through $d$.
  
\begin{defn}A weighting $\mu$ of a symmetric plabic graph $G$ is \emph{symmetric} if $\mu$ assigns the same weight to $(u,v)$ and $(r(u),r(v))$ for each edge $(u,v)$ of $G$.\end{defn}

We will show that for a symmetric plabic graph $G$, the boundary measurement map takes a symmetric weighting of $G$ to a point in $\Lm(2n)$.  We first characterize points of $\Lm(2n)$ in terms of Pl\"{u}cker coordinates.  Recall that for $i \in [2n]$, we have $i' = 2n+1-i$.

\begin{lem}\label{Lagrangian}
Let $V \in \Gr(n,2n)$, and let $I = \{i_1,\ldots,i_n\}$ be the lex-first non-zero Pl\"{u}cker coordinate of $V$.  Then $V \in \Lm(2n)$ if and only if the following hold:
\begin{enumerate}
\item For each $1 \leq i \leq n$, we have $i \in I$ if and only if $i' \not\in I$.
\item For each $j < k \in [n]$ with $i_j < i_k'$, we have
\[\Delta_{(I\backslash \{i_j\}) \cup \{i_k'\}} = \Delta_{(I \backslash \{i_k\})\cup \{i_j'\}}.\]
\end{enumerate}
\end{lem}

\begin{proof}
Let $V$ be an $n$-dimensional subspace of $\mathbb{C}^{2n}$.
Represent $V$ by a $k \times n$ matrix in reduced row-echelon form, so the columns indexed by $I$ form a copy of the identity matrix.  For $j < k$, let $v_j$ and $v_k$ represent rows $j$ and $k$ of $M$, and say
\[v_j=(a_{1},\ldots,a_{2n})\]
\[v_k=(b_{1},\ldots,b_{2n})\]
Then we have
\begin{equation} \label{product} \langle v_j , v_k \rangle = \sum_{r = 1}^{n} (a_{(2r-1)}b_{(2r-1)'} - a_{(2r)}b_{(2r)'}) \end{equation}
Now, $a_{\ell} = 0$ for $\ell < i_j$ and $b_{\ell'} = 0$ for $\ell > i_k'$, and $a_{i_j} = b_{i_k} = 1$.
Suppose $i_k' = i_j$, so that $V$ violates condition (1) above.
Then $\langle v_j , v_k \rangle = \pm 1$, and $V$ is not Lagrangian.  

Assume now that $V$ satisfies condition (1).  We will show that $V$ is Lagrangian if and only if $V$ satisfies condition $(2)$.
By Equation \eqref{product}, we have $\langle v_j,  v_k \rangle = 0$ unless $i_j \leq i_k'$.

Suppose $i_j \leq i_{k'}$.  
Since $i_k' \neq i_j$ by hypothesis, we have $i_j  < i_{k'}$.  For each of $\ell \not\in \{i_j,i_k,i_j',i_k'\}$, either $a_{\ell} = 0$ or $b_{\ell'} = 0$, since either $\ell$ or $\ell'$ is a pivot column of $M$.  Hence, we have
\[\langle v_j, v_k \rangle = (-1)^{i_j+1}b_{i_j'} + (-1)^{i_k'+1}a_{i_k'}.\]

Thus $V$ is Lagrangian if and only if, for all $j < k$ with $i_j < i_k'$, we have
\[b_{i_j'} = \begin{cases}
a_{i_k'} & \text{$i_j$ and $i_k$ have the same parity}\\
-a_{i_k'} &  \text{$i_j$ and $i_k$ have opposite parity}
\end{cases}\]
or equivalently, we have
\[b_{i_j'} = (-1)^{(i_k - i_j)}a_{i_j}'.\]

In the language of Pl\"ucker coordinates, this is equivalent to a collection of relations of the form
\[\Delta_{(I \backslash \{i_j\}) \cup \{i_k'\}} = \pm \Delta_{(I \backslash \{i_k\}) \cup \{i_j'\}}\]
for all $i,j$ with $i_j < i_k'$.
We calculate the relative sign in each case.
Since $i_j < i_k'$ we have
\begin{align*} 
a_{i_k'} &=(-1)^{|I \cap [i_j+1,i_k'-1]|}\Delta_{(I \backslash \{i_j\}) \cup \{i_k'\}}\\
b_{i_j'} & = (-1)^{|I \cap [i_k+1,i_j'-1]|}\Delta_{(I \backslash \{i_k\}) \cup \{i_j'\}}
\end{align*}
Consider first the case where $i_j < i_k \leq n$. Then 
\[(-1)^{|I \cap [i_j+1,i_k'-1]|}\Delta_{(I\backslash \{i_j\}) \cup \{i_k'\}} = (-1)^{(i_k-i_j)}(-1)^{|I \cap [i_k+1,i_j'-1]|}\Delta_{(I \backslash \{i_k\}) \cup \{i_j'\}}\]
The pivot columns that lie strictly between $i_k$ and $i_k'$ contribute at factor of $-1$ to each side of this equation.  Canceling these factors, we are left with
\[(-1)^{|I \cap [i_j+1,i_k]|}\Delta_{(I\backslash \{i_j\}) \cup \{i_k'\}} = (-1)^{(i_k-i_j)}(-1)^{|I \cap [i_k',i_j'-1]|}\Delta_{(I \backslash \{i_k\}) \cup \{i_j'\}}\]
Note that $\ell \in I \cap [i_j+1,i_k]$ is contained in $I$ and only if $\ell' \in [i_k',i_j'-1]$ is \emph{not} contained in $I$. 
Hence 
\[|I \cap [i_j+1,i_k]|+|I \cap [i'_k,i'_j-1]| = |[i_j+1,i_k]| = i_k- i_j\]
and we have 
\[\Delta_{(I \backslash \{i_k\}) \cup \{i_j'\}} = \Delta_{(I\backslash \{i_j\}) \cup \{i_k'\}}\]

Now, take the case where $n \leq i_k \leq i_j'$.
Again, we have
\[(-1)^{|I \cap [i_j+1,i_k'-1]|}\Delta_{(I\backslash \{i_j\}) \cup \{i_k'\}} = (-1)^{(i_k-i_j)}(-1)^{|I \cap [i_k+1,i_j'-1]|+1}\Delta_{(I \backslash \{i_k\}) \cup \{i_j'\}}\]
By a similar argument to the above, we have
\[|I \cap [i_j+1,i_k'-1]|+|I \cap [i_k+1,i_j'-1] = |[i_j+1,i_k'-1]| = i_k' - i_j - 1\]
Since $i_k'$ and $i_k$ have opposite parity, it follows that $(-1)^{(i_k' - i_j -1)} = (-1)^{(i_k-i_j)}$, and so
\[\Delta_{(I\backslash \{i_j\}) \cup \{i_k'\}} = \Delta_{(I \backslash \{i_k\}) \cup \{i_j'\}}.\]
This completes the proof.
\end{proof}

\begin{figure}[H]
\centering
\begin{tikzpicture}
\draw (0,0) circle (3);
\draw ({2*cos(60)},{2*sin(60)}) -- ({-2*cos(60)},{2*sin(60)})  -- ({-2*cos(60)},{-2*sin(60)})  -- ({2*cos(60)},{-2*sin(60)}) --  ({2*cos(60)},{2*sin(60)});
\draw [gray, thin] (0,3.5) -- (0,-3.5);
\draw ({2*cos(60)},{2*sin(60)}) -- ({3*cos(60)},{3*sin(60)});
\draw ({-2*cos(60)},{2*sin(60)}) -- ({-3*cos(60)},{3*sin(60)});
\draw ({-2*cos(60)},{-2*sin(60)}) -- ({-3*cos(60)},{-3*sin(60)});
\draw ({2*cos(60)},{-2*sin(60)}) -- ({3*cos(60)},{-3*sin(60)});
\draw ({3*cos(30)},{3*sin(30)}) -- (2,0) -- ({3*cos(30)},{-3*sin(30)});
\draw ({-3*cos(30)},{3*sin(30)}) -- (-2,0) -- ({-3*cos(30)},{-3*sin(30)});
\bdot{{2*cos(60)}}{{2*sin(60)}};\wdot{{-2*cos(60)}}{{2*sin(60)}};\bdot{{-2*cos(60)}}{{-2*sin(60)}};\wdot{{2*cos(60)}}{{-2*sin(60)}};
\wdot{{3*cos(60)}}{{3*sin(60)}};\bdot{{3*cos(30)}}{{3*sin(30)}};\wdot{2}{0};\bdot{{3*cos(30)}}{{-3*sin(30)}};\bdot{{3*cos(60)}}{{-3*sin(60)}};
\bdot{{-3*cos(60)}}{{3*sin(60)}};\wdot{{-3*cos(30)}}{{3*sin(30)}};\bdot{-2}{0};\wdot{{-3*cos(30)}}{{-3*sin(30)}};\wdot{{-3*cos(60)}}{{-3*sin(60)}};
\node [above right] at ({3*cos(60)},{3*sin(60)}) {1};
\node [above right] at ({3*cos(30)},{3*sin(30)}) {2};
\node [below right] at ({3*cos(-30)},{3*sin(-30)}) {3};
\node [below right] at ({3*cos(-60)},{3*sin(-60)}) {4};
\node [below left] at ({3*cos(240)},{3*sin(240)}) {5};
\node [below left] at ({3*cos(210)},{3*sin(210)}) {6};
\node [above left] at ({3*cos(150)},{3*sin(150)}) {7};
\node [above left] at ({3*cos(120)},{3*sin(120)}) {8};
\node [right] at (0,2) {4};
\node [right] at (0,-2) {3};
\node [right] at (1,0) {7};
\node [left] at (-1,0) {7};
\end{tikzpicture}
\caption{A symmetric weighting of a symmetric plabic graph.  All unlabeled edges have weight $1$.}
\label{graph}
\end{figure}
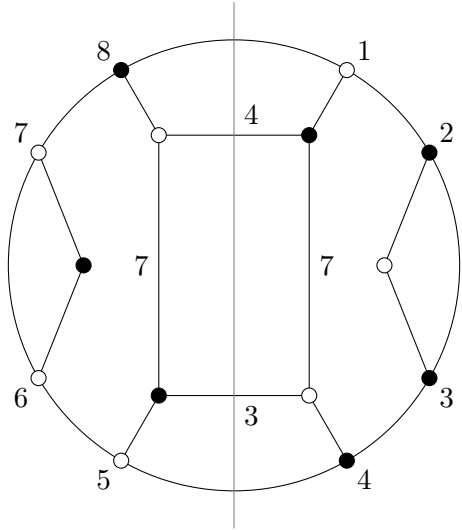

\begin{prop}Let $G$ be a symmetric plabic graph with a symmetric weighting $\mu$, and suppose $G$ is reduced as an ordinary plabic graph.  Then the point $P = \partial_G(\mu)$ is contained in $\Lm(2n)$.
\end{prop}

\begin{proof}
Let $I$ be the lex-first element of the matroid of $P$.  Then by results of \citep[Section 16]{Pos06}, we have
\[I = \{ i \in [n] \mid \sigma_G^{-1}(i) > i \text{ or $\sigma_G(i)$ is a white fixed point}\}.\]
Since $G$ is a symmetric plabic graph, the chord diagram of $\sigma_G$ is symmetric about the distinguished diameter $d$.  It follows that $i \in I$ if and only if $i' \not\in I$.  Hence exactly one member of each pair $(i,i')$ is contained in $I$.

Recall that for $J \in {{[2n]}\choose{n}}$, we have $R(J) = [2n]\backslash \{j' \mid j \in J\}$.  It follows from the discussion in \citep[Section 3]{KS15} that $\Delta_{J}(P) = \Delta_{R(J)}(P)$ for all $J \in {{[2n]}\choose{n}}$.  While the statement in that paper is only for positive real edge 
weights, the same argument holds for nonzero complex weights. 

Let $1 \leq j < k \leq n$ such that $i_k' > i_j$.  Let $J = I \backslash \{i_j\} \cup \{i_k'\}$ and 
let $J' = I \backslash \{i_k\} \cup \{i_j'\}$. Then $J' = R(J)$, so $\Delta_J(P) = \Delta_{J'}(P).$  By the lemma above, it follows that the symmetric weighting of $G$ indeed corresponds to a point in $\Lm(2n)$, and the proof is complete.  
\end{proof}

\subsection{Local moves for symmetric plabic graphs}

Our goal is to show that symmetric plabic graphs with symmetric weights give parametrizations of projected Richardson varieties in $\Lm(2n)$, just as ordinary plabic graphs give parametrizations of positroid varieties in $\Gr(k,n)$.  In this section, we define local moves and reductions for symmetric plabic graphs.  

For each move or reduction from \citep{Pos06}, we have a corresponding symmetric move or reduction, defined as follows. Let $G$ be a symmetric plabic graph.  Suppose we can perform an ordinary move or reduction on $G$, such that the affected portion of the graph lies entirely on one side of the distinguished diameter $R$.  Then simultaneously performing the corresponding move or reduction on the opposite side of $R$ yields a new symmetric plabic graph, equivalent to the first.  This gives two moves and two reductions, corresponding to those of Postnikov.  

We have two additional moves and one additional reduction, shown in Figure \ref{symsquare}.  For the first additional move, if we have a square face which is bisected by the diameter $d$, performing a square move at that face and contracting or uncontracting edges as in Figure \ref{symsquare1} yields a symmetric plabic graph.  For the second, suppose we have an edge which crosses the midline, both of whose vertices have valence two.  Then removing both of these vertices again yields a symmetric plabic graph; conversely, we may add a pair of two-valent vertices of opposite colors to an edge which crosses the midline. For the additional reduction, if we have a pair of parallel edges which span $d$, performing a parallel edge reduction again yields the desired graph.  Finally, we may perform  symmetric gauge transformations, by applying the same gauge transformation to a pair of vertices $v_1$ and $v_2$ which are symmetric with respect to the midline.  

\begin{defn}A symmetric plabic graph is \emph{reduced} if it cannot be transformed by symmetric moves into a graph on which one can perform a symmetric reduction.  A symmetric plabic graph is \emph{strongly reduced} if it is reduced as an ordinary plabic graph.
\end{defn}

Each of our moves or reductions are compositions of Postnikov's moves for ordinary plabic graphs.  Suppose $G$ is a strongly reduced.  Then performing a symmetric move, and transforming the edge weights according to Postnikov's rules, carries symmetric weightings of $G$ to symmetric weightings.  

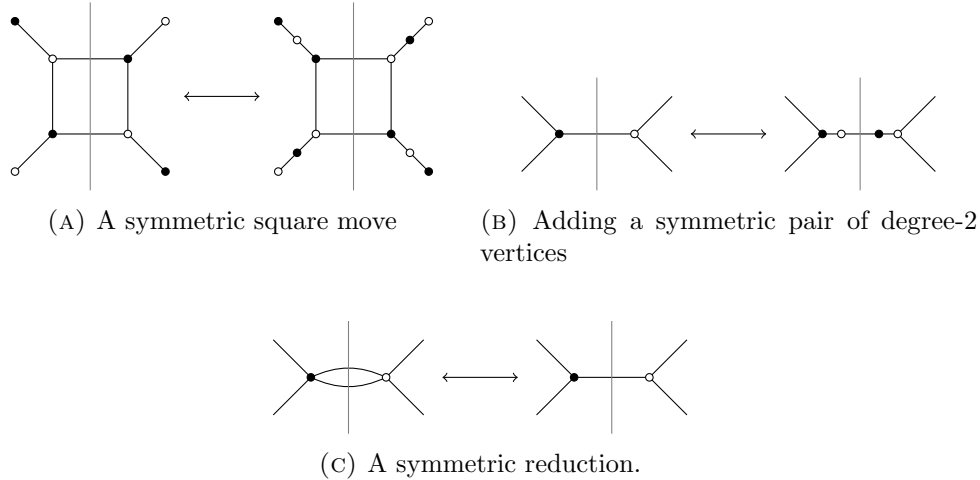
\begin{figure}
\centering
\begin{subfigure}[t]{0.4\textwidth}
\centering
\begin{tikzpicture}[scale = 0.5]
\draw (1,1) -- (1,-1) -- (-1,-1) -- (-1,1) -- (1,1);
\draw (1,1) -- (2,2);
\draw (1,-1) -- (2,-2);
\draw (-1,-1) -- (-2,-2);
\draw (-1,1) -- (-2,2);
\draw [gray, thin] (0,2.5) -- (0,-2.5);
\bdot{1}{1};\wdot{1}{-1};\bdot{-1}{-1};\wdot{-1}{1};
\wdot{2}{2};\bdot{2}{-2};\wdot{-2}{-2};\bdot{-2}{2};
\draw [<->] (2.5,0) -- (4.5,0);
\begin{scope}[xshift = 7 cm]
\draw (1,1) -- (1,-1) -- (-1,-1) -- (-1,1) -- (1,1);
\draw (1,1) -- (2,2);
\draw (1,-1) -- (2,-2);
\draw (-1,-1) -- (-2,-2);
\draw (-1,1) -- (-2,2);
\draw [gray, thin] (0,2.5) -- (0,-2.5);
\wdot{1}{1};\bdot{1}{-1};\wdot{-1}{-1};\bdot{-1}{1};
\bdot{1.5}{1.5};\wdot{1.5}{-1.5};\bdot{-1.5}{-1.5};\wdot{-1.5}{1.5};
\wdot{2}{2};\bdot{2}{-2};\wdot{-2}{-2};\bdot{-2}{2};
\end{scope}
\end{tikzpicture}
\caption{A symmetric square move}
\label{symsquare1}
\end{subfigure}
\begin{subfigure}[t]{0.4\textwidth}
\centering
\begin{tikzpicture}[scale = 0.5]
\draw (-1,0) -- (1,0);
\draw (-2,1) -- (-1,0);
\draw (-2,-1) -- (-1,0);
\draw (1,0) -- (2,1);
\draw (1,0) -- (2,-1);
\draw [gray, thin] (0,1.5) -- (0,-1.5);
\bdot{-1}{0};\wdot{1}{0};
\draw [<->] (2.5,0) -- (4.5,0);
\begin{scope}[xshift = 7 cm]
\draw (-1,0) -- (1,0);
\draw (-2,1) -- (-1,0);
\draw (-2,-1) -- (-1,0);
\draw (1,0) -- (2,1);
\draw (1,0) -- (2,-1);
\draw [gray, thin] (0,1.5) -- (0,-1.5);
\bdot{-1}{0};\wdot{-0.5}{0};\bdot{0.5}{0};\wdot{1}{0};
\end{scope}
\end{tikzpicture}
\caption{Adding a symmetric pair of degree-2 vertices}
\end{subfigure}
\vspace{0.25 in}

\begin{subfigure}{\textwidth}
\centering
\begin{tikzpicture}[scale = 0.5]
\draw (-1,0) [out = 25, in = 155] to (1,0);
\draw (-1,0) [out = -25, in = -155] to (1,0);
\draw (-2,1) -- (-1,0);
\draw (-2,-1) -- (-1,0);
\draw (1,0) -- (2,1);
\draw (1,0) -- (2,-1);
\draw [gray, thin] (0,1.5) -- (0,-1.5);
\bdot{-1}{0};\wdot{1}{0};\
\draw [<->] (2.5,0) -- (4.5,0);
\begin{scope}[xshift = 7 cm]
\draw (-1,0) -- (1,0);
\draw (-2,1) -- (-1,0);
\draw (-2,-1) -- (-1,0);
\draw (1,0) -- (2,1);
\draw (1,0) -- (2,-1);
\draw [gray, thin] (0,1.5) -- (0,-1.5);
\bdot{-1}{0};\wdot{1}{0};
\end{scope}
\end{tikzpicture}
\caption{A symmetric reduction.}
\end{subfigure}
\caption{A symmetric reduction.}
\label{symsquare}
\end{figure}

The goal of this section is to show that a symmetric plabic graph is reduced if and only if it is strongly reduced, and that reduced symmetric plabic graphs with the same bounded affine permutation are equivalent via symmetric local moves.  We make extensive use of the following lemma, which is Lemma 13.5 from \citep{Pos06}.  Note that earlier, we required a bridge from $i$ to $i+1$ to have a white vertex on the leg at $i$, and a black one at $i+1$.  In this section, we also allow bridges which have a black vertex on the leg at $i$, and a white vertex on the leg at $i+1$.

\begin{lem}\label{key}
Let $G$ be a reduced plabic graph with trip permutation $\pi_G$, where $\pi_G$ has no fixed points.  Let $i < j$ be indices such that $\pi_G(i)=j$ or $\pi_G(j)=i$.  
Suppose there is no pair $a,b \in [i+1,j-1]$ such that $\pi_G(a) = b$.  Then $G$ is move-equivalent to a graph with a bridge from $i$ to $i+1$.  
If $\pi_G(i) = j$ and $\pi_G(j)=i$, then $i$ and $j$ are connected by a path whose non-boundary vertices all have degree $2$.
\end{lem}

The lemma below follows from results of Postnikov.  A proof, in the language of bounded affine permutations, may be found in \citep{Lam13}.

\begin{lem}\label{crossing}Let $G$ be a reduced plabic graph with decorated permutation $\pi_G$ and bounded affine permutation $f_G$, and assume $\pi_G$ has no fixed points.  Suppose the chords $i \rightarrow \pi_G(i)$ and $i+1 \rightarrow \pi_G(i+1)$ represent a crossing in $G$.  Then we may transform $G$ into a graph with a bridge that is white at $i$, and black at $i+1$.
\end{lem}

We recall Postnikov's criterion for reducedness, which is Theorem 13.2 of \citep{Pos06}.  

\begin{thm}\label{crit}Let $G$ be a plabic graph which has no leaves, except perhaps some leaves attached to boundary vertices.  Then $G$ is reduced if and only if the following conditions hold.
\begin{enumerate}
\item $G$ has no round-trips.
\item $G$ has no trips which use the same edge more than once, except perhaps for trips corresponding to boundary leaves.
\item $G$ has no pair of trips which cross twice at edges $e_1$ and $e_2$, where $e_1$ and $e_2$ appear in the same order in both trips.
\item If $\pi_G(i) = i$ then $G$ has a lollipop at $i$.  
\end{enumerate}
\end{thm}
Note in particular that while a trip may have a self-intersection at a vertex, no trip in a reduced plabic graph may cross itself.
The following is an easy topological consequence of the above result.  See Figure \ref{alignments} for the concept of alignment.

\begin{lem}\label{aligned}Let $G$ be a reduced plabic graph, and suppose the chords $i \rightarrow \pi_G(i)$ and $j \rightarrow \pi_G(j)$ are aligned in the sense of \citep[Section 16]{Pos06}.  Then the trips from $i$ to $\pi_G(i)$ and $j$ to $\pi_G(j)$ do not cross in $G$.
\end{lem}

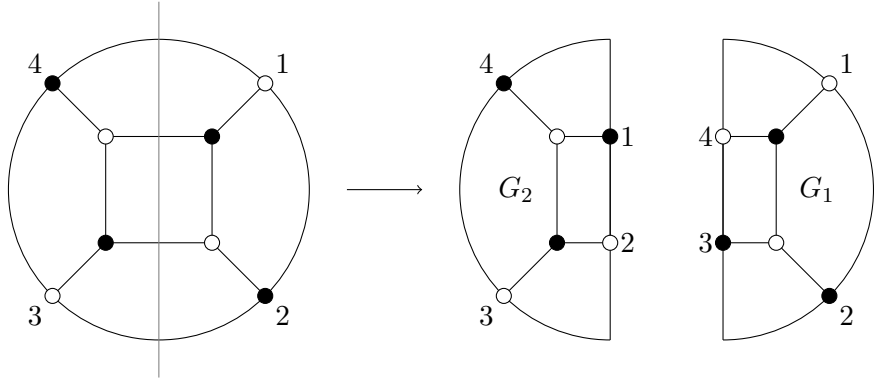
\begin{figure}[ht]
\centering
\begin{tikzpicture}
\draw (0,0) circle (2 cm);
\draw (-0.707,-0.707) rectangle (0.707, 0.707);
\draw (-0.707,-0.707) -- (-1.414,-1.414);
\draw (0.707,-0.707) -- (1.414,-1.414);
\draw (-0.707,0.707) -- (-1.414,1.414);
\draw (0.707,0.707) -- (1.414,1.414);
\bdot{-0.707}{-0.707}\bdot{0.707}{0.707}\bdot{1.414}{-1.414}\bdot{-1.414}{1.414}
\wdot{0.707}{-0.707}\wdot{-0.707}{0.707}\wdot{1.414}{1.414}\wdot{-1.414}{-1.414}
\node [above left] at (-1.414,1.414) {4};
\node [below left] at (-1.414,-1.414) {3};
\node [below right] at (1.414,-1.414) {2};
\node [above right] at (1.414,1.414) {1};
\draw [gray, thin] (0,-2.5) -- (0,2.5);
\draw [->] (2.5,0) -- (3.5,0);
\begin{scope}[xshift = 6 cm]
\draw (0,-2) arc (270:90:2);
\draw (0,-2) -- (0,2);
\draw (-0.707,-0.707) rectangle (0.0, 0.707);
\draw (-0.707,-0.707) -- (-1.414,-1.414);
\draw (-0.707,0.707) -- (-1.414,1.414);
\bdot{-0.707}{-0.707}\bdot{-1.414}{1.414}
\wdot{-0.707}{0.707}\wdot{-1.414}{-1.414}
\node [above left] at (-1.414,1.414) {4};
\node [below left] at (-1.414,-1.414) {3};
\bdot{0}{0.707};\wdot{0}{-0.707};
\node[right] at (0,0.707) {$1$};
\node[right] at (0,-0.707) {$2$};
\node at (-1.25,0) {$G_2$};
\end{scope}
\begin{scope}[xshift = 7.5 cm]
\draw (0,-2) arc (-90:90:2);
\draw (0,-2) -- (0,2);
\draw (0.707,-0.707) rectangle (0.0, 0.707);
\draw (0.707,-0.707) -- (1.414,-1.414);
\draw (0.707,0.707) -- (1.414,1.414);
\wdot{0.707}{-0.707}\wdot{1.414}{1.414}
\bdot{0.707}{0.707}\bdot{1.414}{-1.414}
\node [above right] at (1.414,1.414) {$1$};
\node [below right] at (1.414,-1.414) {$2$};
\wdot{0}{0.707};\bdot{0}{-0.707};
\node[left] at (0,0.707) {$4$};
\node[left] at (0,-0.707) {$3$};
\node at (1.25,0) {$G_1$};
\end{scope}
\end{tikzpicture}
\caption{The distinguished diameter $d$ divides a symmetric plabic graph into two regions, which correspond to a pair of plabic graphs $G_1$ and $G_2$.}
\label{twosides}
\end{figure}

Let $G$ be a symmetric plabic graph.  The diameter $d$ of the disc divides $G$ into two regions.  Consider the region to the right of $d$.  Drawing a new boundary segment along $d$, and adding boundary vertices wherever an edge of $G$ intersects $d$, we obtain a new plabic graph $G_1$.  Similarly, we may define a plabic graph $G_2$ corresponding to the region of $G$ to the left of $d$.  See Figure \ref{twosides}.  

Notice that any move or reduction we may perform in $G_1$ or $G_2$ corresponds to a valid move or reduction in $G$.  By inserting pairs of degree-two vertices along edges that cross $d$, we may assume that a move or reduction in $G_1$ does not affect $G_2$, and vice versa.  Hence performing the corresponding move on each side of $d$ gives a symmetric move or reduction in $G$.  Thus if $G$ is a reduced symmetric plabic graph, $G_1$ and $G_2$ are reduced as ordinary plabic graphs.  Since strongly reduced implies reduced, this is also true for strongly reduced graphs.

\begin{prop}\label{bridge} Let $G$ be a strongly reduced symmetric plabic graph with more than one face.  Then $G$ may be transformed by symmetric moves into a symmetric plabic graph which either has a symmetric pair of bridges not crossing the midline $d$, or a single bridge which does cross $d$.  In the latter case, we may assume the two endpoints of the bridge are connected by a path containing only two-valent vertices.
\end{prop}

\begin{proof}Since $G$ is strongly reduced, every fixed point of $G$ corresponds to a boundary leaf, so we may assume without loss of generality that $G$ has no fixed points.  Let $G_1$ and $G_2$ be as above.  Let $B_1$ be the segment of the boundary of $G_1$ which lies on the boundary of $G$, and define $B_2$ similarly.  Suppose some pair of boundary vertices $1 \leq i < j \leq n$ on $B_1$ satisfies the hypotheses of Lemma \ref{key}.  Then we can transform $G_1$ into a graph with a bridge $(i,k)$ for some $i < k < j$.  Performing the corresponding sequence of symmetric moves in $G$ yields a symmetric plabic graph with two commuting bridges $(i,k)$ and $(k',i')$.  

Next, suppose no such pair $(i,j)$ exists.  Then $\pi_{G_1}$ does not map any vertex on $B_1$ to another vertex in $B_1$.  Note that since $G$ is strongly reduced, no trip in $G$ may cross the midline more than once.  Indeed, suppose $T$ is such a trip, and suppose $T$ crosses the midline at edges $e_1$ and $e_2$.  Let $T'$ be the trip which is the mirror image of $T'$ (such a trip exists, since $G$ is symmetric).  Then $T'$ also crosses the midline at edges $e_1$ and $e_2$, in that order; this contradicts the reducedness criterion.  Hence the trip permutation cannot map any boundary vertex of $G_1$ which lies on $d$ to another vertex on $d$.  By assumption, since no pair of vertices on $B_1$ satisfies the conditions of Lemma \ref{key}, the permutation $\pi_{G_1}$ cannot map any vertex on $B_1$ to another vertex on $B_1$.  Hence $\pi_{G_1}$ maps each vertex on $B_1$ to a vertex on $d$, and vice versa.  

Number the vertices of $G_1$ clockwise so that each vertex on $B_1$ comes before each vertex on $d$.  Let $r_1$ be the first vertex on $d$, and let $t_1,\ldots,t_n$ be the vertices on $B_1$.  If $t_a =\pi_{G_1}(r_1) \neq t_n$ then the pair $(t_a,r_1)$ satisfies the condition of Lemma \ref{key}, so we can transform $G_1$ into a graph with a bridge adjacent to $t_a$.  Performing the corresponding sequence of symmetric moves transforms $G$ to a graph with two symmetric bridges that do not cross the midline, and we are done in this case.  Similarly, if $t_b = \pi_{G_1}^{-1}(r_1) \neq t_n$, then $(t_b,r_1)$ satisfies the condition, and we are done as above.

There remains the case where $\pi_{G_1}(t_n) = r_1$ and $\pi_{G_1}(r_1) = t_n$.  In this case, we may transform $G_1$ to a graph where $t_n$ and $r_1$ are connected by a two-valent path.  Performing the corresponding moves on $G_2$, see that $G$ is move equivalent to a symmetric graph where the vertices corresponding to $t_n$ and its reflection over the midline $d$ are connected by a path with all vertices two-valent, and we are done.
\end{proof}

\begin{prop}\label{symred}A symmetric plabic graph $G$ is reduced if and only if it is strongly reduced.
\end{prop}

\begin{proof}A symmetric plabic graph which is strongly reduced is certainly reduced, since all symmetric reductions are composition of Postnikov's reductions.  We must prove the converse.  That is, suppose $G$ is a symmetric plabic graph with no non-boundary leaves which does not satisfy the conditions of Theorem \ref{crit}.  We claim that we can transform $G$ into a graph $G'$ upon which we may perform a symmetric reduction. 
We induce on the number of faces of $G$.  If $G$ has a single face, then $G$ is a lollipop graph, and the result is trivial.  Suppose $G$ has $m$ faces, and suppose the result holds for graphs with fewer than $m$ faces.  

Let $G,$ $G_1$ and $G_2$ be as above.  If $G_1$ is non-reduced as a plabic graph, we may transform $G_1$ into a graph on which we can perform some reduction.  Hence we may transform $G$ by symmetric moves into a graph where we may perform some symmetric reduction, and $G$ is non-reduced as a symmetric graph.  We may thus assume $G_1$ and $G_2$ are reduced.

Suppose there is a trip $T_1$ in $G$ which crosses the midline $d$ more than once, with consecutive crossings $e_1$ and $e_2$.  Let $T_2$ be the trip which is the mirror image of $T_1$ in $G$.  (Such a trip exists, since $G$ is symmetric.)  Then $T_1$ and $T_2$ cross at edges $e_1$ and $e_2$, which appear in the same order in both trips.  Note that since $G_1$ and $G_2$ are reduced, any round trip, trip which uses the same edge twice, or trip which starts and ends at the same point in $G$ must cross $d$ at least twice, and hence induce some instance of this configuration.  

Let $T_1$ and $T_2$ be as above.  Let $S_1$ by the segment of $T_1$ between $e_1$ and $e_2$ inclusive, and define $S_2$ similarly.  Uncontracting edges, we may eliminate any self-intersections of $S_1$ or $S_2$.  Hence we may assume the area bounded by $S_1$ and $S_2$ is homeomorphic to a disk \citep[Section 13]{Pos06}.  Let $H$ be the subgraph of $G$ bounded by $S_1$ and $S_2$.  We may further uncontract edges to ensure that each vertex on $S_1$ or $S_2$ is adjacent to at most one vertex inside $H$.  Hence the subgraph $H$ of $G$ bounded by $S_1$ and $S_2$ is a reduced symmetric plabic graph, which is strongly reduced by induction.  Note that since $G$ has no leaves, the trip permutation of $H$ has no fixed points.  We claim that we can reduce to the case where $H$ has only one face.

Suppose without loss of generality that $S_1$ is a clockwise trip.  The trip $T_1$ alternates between black and white vertices.  All vertices on $S_1$ incident to edges in $H$ are white, while those adjacent to edges outside of $H$ are black. Moreover, by contracting and un-contracting edges, we may ensure that each vertex on $S_1$ or $S_2$ has degree $3$.

Suppose $H$ has more than once face.  Then we may transform $H$ into a graph with either a pair of commuting bridges on either side of the midline, or a valent-two path between two neighboring vertices on either side of the midline, by Lemma \ref{bridge}.  If, after these transformations, the graphs $G_1$ and $G_2$ are no longer reduced, we are done, so suppose $G_1$ and $G_2$ remained reduced.

Suppose the first case holds, so that $H$ has a pair of commuting bridges.  After contracting some edges, each bridge gives a square face of $G$, two of whose sides are on the boundary of $H$.  (Note that we must obtain such a square face when we contract edges; the other option, a pair of parallel edges, would contradict the fact that $G_1$ is reduced after the transformation.)  Performing a square move at this face, and the symmetric move on the other side, we reduce the number of faces of $H$ by two.

Next, suppose the second case holds.  Then we have a two-valent path between the boundary vertex of $H$ on $S_1$ and its reflection through the midline.  Symmetrically removing two-valent vertices and contracting edges, we obtain a symmetric square face in $G$, three of whose edges are part of the boundary of $H$.  Performing a symmetric square move at this face reduces the number of faces of $H$ by $1$.

Hence $H$ has a single face.  But this, in turn, means that $e_1$ and $e_2$ are a pair of parallel edges.  Hence we may perform a parallel edge reduction, and this case is complete.  

Next suppose that no two trips in $G$ cross the midline more than once.  Since $G$ is not reduced, $G$ must have two trips $T_1$ and $T_2$ which cross twice, with the two crossings occurring in the same order in both trips.  Moreover, $T_1$ and $T_2$ must cross once at an edge to the right of $d$, and once at and edge the left of $d$. By symmetry, we may assume that $T_1$ and $T_2$ originate to the left of $d$, cross once to the left of $d$, and then cross again to the right of $d$. Hence, two trips in $G_1$ which originate on $d$ must cross.

Let $r_1,\ldots,r_m$ be the boundary vertices of $G_1$ which lie on $d$, from bottom to top.  Let $B_1$ be the segment of the boundary of $G_1$ which coincides with the boundary of $G$.  Each trip in $G_1$ which originates on $d$ must end on $B_1$.  Hence each pair of trips originating on $d$ represents either an alignment or a crossing.  Moreover, the trips originating at $r_1,\ldots,r_{i+1}$ cannot all represent alignments: by Lemma \ref{aligned} and the previous paragraph, some of these trips must cross.

Consider the trips in $G_1$ which originate at $r_1,\ldots,r_n$.  Let $i$ be the first index such that $\pi_{G_1}$ such that the trips originating at $r_i$ and $r_{i+1}$ form a crossing.  Then the trips originating at $r_1,\ldots,r_i$ are pairwise aligned. Assume that $i$ is maximal among all symmetric graphs which are equivalent to $G$ by symmetric moves, and which satisfy the conditions that no trip crosses $d$ more than once and the subgraph on each side of $d$ is reduced.  We may transform $G_1$ into a graph with a bridge $(r_i,r_{i+1})$ which is white at $r_{i}$ and black at $r_{i+1}$, and perform the corresponding transformation on $G_2$.  Performing a symmetric square move in $G$ then yields a new graph $G^*$.  Note that no trip in $G^*$ crosses the midline more than once.  

Now $G_1^*$ is identical to $G_1$, except that the bridge at $(r_i,r_{i+1})$ is now black at $r_i$ and white at $r_{i+1}$. If $G_1^*$ is non-reduced, then $G$ is non-reduced as a symmetric plabic graph, and we are done.  Otherwise, suppose the trips starting at $r_i$ and $r_{i-1}$ are now aligned.  Then the trips originating at $r_1,\ldots,r_{i+1}$ are pairwise aligned, contradicting the maximality of $i$.  Hence the trips which originate at $r_{i-1}$ and $r_i$ must form a crossing. We may thus transform $G_1^*$ into a graph with a bridge at $(i-1,i)$, and perform a symmetric square move in $G^*$.  Continuing in this fashion, we may successively uncross trips starting at $r_{\ell}$ and $r_{\ell-1}$, for $\ell = i,i-1,\ldots,2$.  This process must eventually yield a graph $\overline{G}$ where $\overline{G}_1$ is non-reduced. Otherwise, we could may transform $G$ by symmetric moves into a graph $\overline{G}$ with $\overline{G}_1$ reduced, no trip crossing the midline more than once, and trips originating at $r_1,\ldots,r_{i+1}$ pairwise aligned, a contradiction.  This completes the inductive step, and with it the proof.
\end{proof}

\begin{prop}Let $G$ and $G'$ be two reduced symmetric plabic graphs.  Then $G$ and $G'$ have the same bounded affine permutation if and only if we can transform $G$ into $G'$ by a series of symmetric moves.
\end{prop}

\begin{proof}

Each symmetric move is a composition of one or more of Postnikov's moves for plabic graphs.  Hence by Lemma 13.1 of \citep{Pos06}, the symmetric moves do not change the decorated permutation of $G$, and the forward direction is clear.   

For the reverse direction, note that since $G$ is reduced, $G$ is strongly reduced.  In particular, no trip in $G$ crosses the midline more than once.  Let $G_1$ and $G_2$ be defined as above, and let $r_1,\ldots,r_m$ be the vertices of $G_1$ which lie along $d$, numbered from bottom to top.  By the argument in the previous proof, the trips in $G_1$ originating on $d$ represent pairwise alignments or crossings.  Moreover, we claim that we may reduce to the case where these trips are pairwise aligned.  

Suppose the trips starting at $r_i$ and $r_{i+1}$ are the first pair which represent a crossing, so that the trips starting at $r_1,\ldots,r_i$ are pairwise aligned.  Assume that we have chosen $i$ maximal in the move-equivalence class of $G$.  Then we can transform $G_1$ into a graph with a white-black bridge at $(r_i,r_{i+1})$ and perform a square move.  As above, this uncrosses the trips originating at $r_i$ and $r_{i+1}$, while leaving the other trips which originate on $d$ unchanged.  Since $G$ is reduced, the result must be a graph $G^*$ with $G_1^*$ and $G_2^*$ reduced.  If $r_i$ and $r_{i-1}$ now represent a crossing, we iterate this process, and uncross $r_{i}$ and $r_{i-1}$.  Continuing in this fashion, we may transform $G^*$ into a new graph $\overline{G}$ where the trips originating at $r_1,\ldots,r_{i+1}$ are pairwise aligned. 
This is a contradiction, since $i$ was assumed to be maximal.  

Hence we may transform $G$ by symmetric moves into a configuration where the trips originating at $r_1,\ldots,r_m$ are aligned, and the symmetric statement is true for $G_2$.  Note that with this condition, the trip permutations $\pi_{G_1}$ and $\pi_{G_2}$ are uniquely determined by $\pi_G$.

Repeating the argument, we may transform $G'$ into the same form.  But now $G_1$ and $G_1'$ are reduced plabic graphs with the same trip permutation, and similarly for $G_2$ and $G_2'$.  Hence we can transform $G_1$ into $G_1'$ by a series of moves.  Performing the corresponding symmetric moves on $G$ yields $G'$, and the proof is complete.
\end{proof}

We have shown that two symmetric plabic graphs have the same bounded affine permutation, and hence are associated to the same cell in the Lagrangian Grassmannian, if and only if they are equivalent by symmetric local moves.  Hence the combinatorics of symmetric plabic graphs neatly parallels the combinatorics of ordinary plabic graphs, as desired.  

\subsection{Network parametrizations for projected Richardson varieties in $\Lm(2n)$.}

Let $G$ be a reduced symmetric plabic graph with edge set $E$, and let $\Pi_G^A$ denote the corresponding positroid variety in $\Gr(n,2n)$.  Let $\G^E$ denote the space of edge weightings of $G$, and let $\G^E_C$ be the space of symmetric weightings of $G$.  Let $\G^V$ denote the group of gauge transformations of $G$, and let $\G_C^V$ denote the group of symmetric gauge transformations; that is, gauge transformations which act by the same value on $v$ and $r(v)$ for each internal vertex $v$ of $G$.

\begin{lem}\label{torus}The image of $\G_C^E$ is closed in $\G^E/\G^V$.
\end{lem}

\begin{proof}
First, we choose a subset $F$ of the edges of $G$ which meets all of the following conditions:
\begin{enumerate}
\item $F$ is the disjoint union of a collection of trees in $G$, each of which contains exactly one boundary leg of $G$.
\item $F$ covers each vertex of $G$ exactly once.
\item $F$ is symmetric about the midline $d$ of $G$; that is, an edge $e$ of $G$ is contained in $F$ if and only if the same is true for $r(e)$.
\end{enumerate}
It is not hard to show that such a subset exists, since $G$ is symmetric and reduced.

Working inward from the boundary, we may then successively gauge-fix each remaining edge in $F$ to $1$.  We call the resulting weighting of $G$ an $F$-weighting.  Each point in $\G^E/\G^V$ may be represented by a unique $F$-weighting, and the weights of edges in $E - F$ give coordinates on a dense subset of $\Pi_G^A$.

Consider a weighting $w$ of $G$ which lies in $\G_C^E$.  Since the forest $F$ is symmetric about the midline, the weighting $w$ may be transformed into an $F$-weighting by a series of symmetric gauge transformations, which preserve $\G_C^E$.  Every gauge-equivalence class contains a unique $F$-weighting, so an equivalence class $\overline{w} \in \G^E/\G^V$ is in the image of $\G_C^E$ if and only if the corresponding $F$-weighting is symmetric.  Moreover, the map 
\[\omega_F:\G^E \rightarrow \G^{E\backslash F}\]
which carries each weighting $w$ to the corresponding $F$-weighting is continuous.

Now, let $w$ be a weighting of $G$ which is \emph{not} contained in the image of the space $\G_C^E$ of symmetric weightings, so that the $F$-weighting $w'$ corresponding to $\overline{w}$ is not symmetric.  Then there is a neighborhood of the point $\omega_F(w')$ in $\G^{E\backslash F}$  which corresponds to $F$-weightings which likewise are not symmetric; taking the preimage in $\G^E$, we have an open subset of $G^E$ containing the preimage of $\overline{w}$, which does not intersect the image of $\G_C^E$.  This completes the proof.
\end{proof}

We have shown that the image of $\G_C^E$ is closed in $\G^E/\G^V$.  Moreover, two elements of $\G_C^E$ map to the same point in $\G^E/\G^V$ if and only if they are related by a symmetric gauge transformation.  Let $\G_C^V$ be the group of symmetric gauge transformations.  Then we obtain a Lagrangian boundary measurement map 
\[\mathbb{D}_G^C:\G_C^E/\G_C^V \rightarrow \Pi_G^C\]
simply by composing the boundary measurement map with the obvious embedding of $\G_C^E/\G_C^V\hookrightarrow \G^E/\G^V.$

\begin{thm}
The Lagrangian boundary measurement map $\mathbb{D}^C$ takes $\G_C^E/\G_C^V$ birationally to a dense subset of $\Pi_G^C$.  
\end{thm}

\begin{proof}
By the main theorem of \citep{MS14}, the boundary measurement map 
\[\mathbb{D}_G:\G^E/\G^V \rightarrow \Pi_G^A\]
is a birational map, which is regular on its domain of definition.  Moreover, $\mathbb{D}_G^{-1}$ is defined precisely on the image of $\mathbb{D}_G$, which is open in $\Pi_G^A$. 
It follows that $\mathbb{D}_G(\G^E/\G^V) \cap \Pi_G^C$ is an open, nonempty subset of the irreducible algebraic set $\Pi_G^C$ and is therefore dense in $\Pi_G^C$.  Now, the torus $\G_C^E/\G_C^V$ embeds as a closed subset of 
$\G^E/\G^V$.  
Hence, the image $\D^C_G(\G_C^E/\G_C^V)$ is a locally closed subset of $\mathbb{D}_G(\G^E/\G^V) \cap \Pi_G^C$.  

We show that $\mathbb{D}_G^C(\G_C^E/\G_C^V) = \mathbb{D}_G(\G^E/\G^V) \cap \Pi_G^C$ and hence $\mathbb{D}_G^C(\G_C^E/\G_C^V)$ is open and dense in $\Pi_G^C$.  For this, in turn, it suffices to show
\begin{equation} \label{fulldim} \dim \left(\G_C^E/\G_C^V \right)= \dim \Pi_G^C \end{equation}
since a full-dimensional closed subset of the irreducible algebraic set $\mathbb{D}_G(\G^E/\G^V) \cap \Pi_G^C$ must be the entire set.  

To check \eqref{fulldim}, it is enough to consider the case of a bridge graph.  The image of each symmetric bridge graph is indeed full-dimensional in $\Pi_G^C$, since bridge graphs encode MR parameterizations.  This completes the proof.
\end{proof}

Hence symmetric plabic graphs parametrize projected Richardson varieties in $\Lm(2n)$, just as ordinary plabic graphs parametrize positroid varieties in $\Gr(n,2n)$.  From this result, we obtain a collection of relations which cut out $\Lm(2n)$ in $\Gr(n,2n)$.  

\begin{thm}
Let $P \in \Gr(n,2n)$.  Then $P$ lies in $\Lm(2n)$ if and only if, for each $I \in {{[2n]}\choose{n}}$, we have
\[\Delta_I(P) =\Delta_{R(I)}(P).\]
\end{thm}

\begin{proof}
We first show that the desired relations hold on all of $\Lm(2n)$.  It is enough to show that the desired relations hold on each projected Richardson variety $\Pi^C$ of $\Lm(2n)$, since projected Richardson varieties form a stratification of $\Lm(2n)$.  Clearly, these relations are satisfied by any point corresponding to a symmetric plabic graph with a symmetric weighting.  Let $G$ be a symmetric plabic graph corresponding to $\Pi^C$.  Then the Lagrangian boundary measurement map $\mathbb{D}^C$ takes the space of Lagrangian weightings of $G$ to a dense subset of $\Pi^C$.  Since $\Pi^C$ is an irreducible algebraic set, the relations hold on all of $\Pi^C$, and we are done.

Conversely, let $P \in \Gr(n,2n)$ be a point which satisfies
\[\Delta_I(P) = \Delta_{R(I)}(P)\]
for all $I \in {{[2n]}\choose{n}}$.  Let $J$ be the lex-minimal basis of $P$.  
We claim that $J$ contains exactly one element of each pair $\{a,a'\}$ with $a \in [2n]$.  
Suppose the claim holds.  Then $P$ satisfies all of the relations in the statement of Lemma \ref{Lagrangian}, and so $P \in \Lm(2n)$ and we are done.  It remains to check the claim.  

Suppose the claim fails.  Let $a$ be the smallest element of $[2n]$ such that either $\{a,a'\} \subseteq J$ or $\{a,a'\} \subseteq [2n] \backslash J$.  
Now $\Delta_{R(J)}(P) = \Delta_{J}(P)$, so $J \leq R(J)$ in the lex order on ${{[2n]}\choose{n}}$.  Note that if $a > 1$, then we have 
\[J \cap ([1,a-1] \cup [a'+1,2n]) = R(J) \cap ([1,a-1] \cup [a'+1,2n]).\]
If $a,a' \not\in J$, then $a,a' \in R(J)$, so $R(J) < J$ in lex order, a contradiction.  Thus, we have $a,a' \in J$.  But this, in turn, forces $a,a' \not\in R(J)$.  

Let $M$ be a matrix representative for $P$.  Then the minor of $M$ indexed by $R(J)$ includes precisely the pivot columns of $M$ indexed by $[1,a-1] \cup [a'+1,2n]$, together with some subset of the columns $[a+1,a'-1]$.  Hence, the span of these columns is contained in the span of $J\backslash \{a'\}$, and the corresponding minor vanishes, a contradiction.  It follows that $J$ contains exactly one of each pair $\{a,a'\}$, and the proof is complete.
\end{proof}

\section{Total nonnegativity for $\Lm(2n)$}

\label{positivity}

In the Grassmannian case, positivity of Pl\"{u}cker coordinates agrees with Lusztig's notion of total nonnegativity for partial flag manifolds.  Moreover, plabic graphs with positive real edge weights parametrize totally nonnegative cells in $\Gr(k,n)$.  We now prove analogous statements for $\Lm(2n)$.

\begin{prop}
Let $\mathring{\Pi}^C$ be a projected Richardson variety in $\Lm(2n)$.  Then set theoretically, 
\[\mathring{\Pi}^C_{\geq 0}  = \mathring{\Pi}^C \cap \Gr_{\geq 0}(k,n).\]
\end{prop}

\begin{proof}
Let $\langle u,w \rangle_n^C \in \mathcal{Q}^C(2n)$ be the equivalence class corresponding to $\mathring{\Pi}^C$, and consider the corresponding class $\langle u,w \rangle_n \in \mathcal{Q}(n,2n)$.  
Let $\widetilde{\mf{w}}$ be a reduced word for $w$ in $S_n^C$, and let $\widetilde{\mf{u}} \preceq \widetilde{\mf{w}}$ denote the unique PDS for $u$ in $w$.  Let $\mf{u}$ and $\mf{w}$ denote the images of $\widetilde{\mf{u}}$ and $\widetilde{\mf{w}}$, respectively, under the embedding $S_n^C \hookrightarrow S_{2n}$; this is uniquely determined up to commutation moves.  The totally nonnegative part of $\mathring{\Pi}^C$ is the image in $\Lm(2n)$ of the subset of $\mathcal{R}_{\widetilde{\mf{u}},\widetilde{\mf{w}}}^C$ where all parameters take nonnegative real values.

Embed $\Sp(2n)/B_+^{\sigma}$ in $\mathcal{F}\ell(n)$ as before.  It follows from the proof of Lemma \ref{deoAC} that each Deodhar component of $\mathring{R}_{u,w}^C$ is the subset of the corresponding Deodhar component of $\mathring{R}_{u,w}^A$ where the parameters satisfy a number of conditions of the form $t_i = t_{i+1}$.  In particular, the totally nonnegative part of $\mathring{R}_{u,w}^C$ is the locally closed subset of $\mathcal{R}_{\mf{u},\mf{w}}$ cut out by these equalities, and is hence the intersection of $\mathring{R}_{u,w}^C$ with the totally nonnegative part of $\mathring{R}_{u,w}^A$.  Projecting to $\Lm(2n) \subseteq \Gr(n,2n)$ gives the desired result.
\end{proof}

Since open projected Richardson varieties form a stratification of $\Lm(2n)$, it follows that set-theoretically we have
\[\Lm_{\geq 0} (2n) = \Lm(2n) \cap \Gr_{\geq 0}(n,2n)\]
with our given choice of embedding and of symplectic form.  Hence, the totally nonnegative part of $\Lm(2n)$ is precisely the \emph{symmetric part} of $\Gr_{\geq 0}(k,n)$ studied in \citep{KS15}.

Let $G$ be a symmetric plabic graph, with corresponding projected Richardson variety $\Pi_G^C$.  Consider the space $(\G_C^E)_{\geq 0}$ of Lagrangian weightings of $G$ such that all edges of $G$ have positive real weights.  Let $(\G_C^V)_{\geq 0}$ denote the group of symmetric, positive real gauge transformations of $G$.  Then we have 
\[(\G_C^E)_{\geq 0}/(\G_C^V)_{\geq 0} \hookrightarrow \G_C^E/\G_C^V.\]

\begin{thm}For $G$ a symmetric plabic graph, restricting the map $\mathbb{D}^C$ to $(\G_C^E)_{\geq 0}/(\G_C^V)_{\geq 0}$ gives an isomorphism of real semi-algebraic sets 
\[(\G_C^E)_{\geq 0}/(\G_C^V)_{\geq 0} \cong (\Pi_G^C)_{\geq 0}.\]
\end{thm}

\begin{proof}
As long as we restrict internal edge weights to positive real numbers, all symmetric local moves induce isomorphisms--not simply birational maps--between the spaces of symmetric edge weightings of symmetric plabic graphs.  Hence, it is enough to prove the claim for a single choice of $G$.  For this, we simply choose a symmetric bridge graph.  Symmetric bridge graphs with Lagrangian weightings encode MR parametrizations, so the claim follows easily in this case, and the proof is complete.
\end{proof}

\begin{cor}The totally nonnegative cells of $\Lm(2n)$ are precisely the nonempty matroid cells of $\Lm_{\geq 0}(2n)$.
\end{cor}

\section{Indexing projected Richardson varieties in $\Lm(2n)$}

\label{index}

We now state a theorem which gives the type $C$ versions of the major combinatorial indexing sets for positroid varieties.

\begin{thm}
Each of the following are in bijection with projected Richardson varieties in $\Lm(2n)$.  In cases where the indexing set has a natural poset structure, the partial order corresponds to the \emph{reverse} of the closure order on projected Richardson varieties in $\Lm(2n)$.
\begin{enumerate}
\item \label{idle} Type $B$ \Le-diagrams which fit inside a staircase shape of size $n.$
\item \label{inq} The poset $\mathcal{Q}^C(2n)$.  
\item \label{inbd} The poset $\Bd^C(2n)$.
\item \label{ingraph} Equivalence classes of reduced symmetric plabic graphs, where the equivalence relation is given by symmetric moves.
\item \label{inmatroid} The poset of positroids $\mathcal{J}$ which satisfy
\[I \in \mathcal{J} \Leftrightarrow R(I) \in \mathcal{J}\]
ordered by \emph{reverse} containment.
\item \label{innecklace} The poset of Grassmann necklaces $\mathcal{I} = (I_1,\ldots,I_{2n})$ which satisfy $I_i = R(I_{i'+1})$, ordered by setting $\mathcal{I} \leq \mathcal{K}$ if $I_i \leq_i \mathcal{K}_i$ for all $i$.
\end{enumerate}
\end{thm}

\begin{proof}
Part (\ref{idle}) was proved in \citep{LW07}.  We showed (\ref{inq}) and (\ref{inbd}) in Section \ref{bound}, while (\ref{ingraph}) follows from the discussion in Section \ref{Cbdry}. 
For (\ref{inmatroid}) note that the matroids of the totally nonnegative cells in $\Lm_{\geq 0}(2n)$ are precisely the positroids corresponding to symmetric plabic graphs.  The characterization of these positroids, and their corresponding Grassmann necklaces, follows from Theorem 3.1 of \citep{KS15}

For the statement about partial orders, note that each of the given posets embeds in the corresponding type $A$ poset which indexes positroid varieties.  In each case, the type $C$ poset indexes the open positroid varieties which intersect $\Lm(2n)$.  There are canonical isomorphisms among the type-$A$ posets, which preserve the correspondences with positroid varieties.  Hence it is enough to show the claim for any \emph{one} of the type $C$ posets.  Lemma \ref{nicepos} says precisely this for the poset $\mathcal{Q}^C(2n)$.
\end{proof}

\bibliographystyle{abbrvnat}
\bibliography{comboLG}

\begin{thebibliography}{26}
\providecommand{\natexlab}[1]{#1}
\providecommand{\url}[1]{\texttt{#1}}
\expandafter\ifx\csname urlstyle\endcsname\relax
  \providecommand{\doi}[1]{doi: #1}\else
  \providecommand{\doi}{doi: \begingroup \urlstyle{rm}\Url}\fi

\bibitem[Arkani-Hamed et~al.(2012)Arkani-Hamed, Bourjaily, Cachazo, Goncharov,
  Postnikov, and Trnka]{ABCGPT12}
N.~Arkani-Hamed, J.~L. Bourjaily, F.~Cachazo, A.~Goncharov, A.~Postnikov, and
  J.~Trnka.
\newblock Scattering amplitudes and the positive {G}rassmannian.
\newblock \emph{Preprint}, 2012.
\newblock arXiv:1212.5605v2 [hep-th].

\bibitem[Bergeron and Sottile(1998)]{BS98}
N.~Bergeron and F.~Sottile.
\newblock Schubert polynomials, the {B}ruhat order, and the geometry of flag
  manifolds.
\newblock \emph{Duke Math. J.}, 95:\penalty0 373--423, 1998.
\newblock arXiv:alg-geom/9703001.

\bibitem[Billey and Lakshmibai(2000)]{BL00}
S.~Billey and V.~Lakshmibai.
\newblock \emph{Singular loci of {S}chubert varieties}, volume 182 of
  \emph{Progress in Mathematics}.
\newblock Birkh\"auser Boston, Inc., Boston, MA, 2000.

\bibitem[Bjorner and Brenti(2005)]{BB05}
A.~Bjorner and F.~Brenti.
\newblock \emph{Combinatorics of {C}oxeter {G}roups}.
\newblock Springer Science+Business Media, Inc., New York, USA, 2005.

\bibitem[Brown et~al.(2006)Brown, Goodearl, and Yakimov]{BGY06}
K.~A. Brown, K.~R. Goodearl, and M.~Yakimov.
\newblock Poisson structures on affine spaces and flag varieties. {I}. {M}atrix
  affine {P}oisson space.
\newblock \emph{Advances in Mathematics}, 206\penalty0 (2):\penalty0 567--629,
  2006.
\newblock arXive:math/0509075v2 [math.QA].

\bibitem[Deodhar(1985)]{Deo85}
V.~V. Deodhar.
\newblock On some geometric aspects of {B}ruhat orderings. {I}. {A} finer
  decomposition of {B}ruhat cells.
\newblock \emph{Inventiones mathematicae}, 79:\penalty0 499--511, 1985.

\bibitem[He and Lam(2011)]{HL11}
X.~He and T.~Lam.
\newblock Projected {R}ichardson varieties and affine {S}chubert varities.
\newblock \emph{Preprint}, 2011.
\newblock arXiv:1106.2586 [math.AG].

\bibitem[Karpman(2016)]{Kar14}
R.~Karpman.
\newblock Bridge graphs and {D}eodhar parametrizations for positroid varieties.
\newblock \emph{Journal of Combinatorial Theory, Series A}, 142, 2016.
\newblock arXiv:math/1411.2997 [math.CO].

\bibitem[Karpman and Su(2015)]{KS15}
R.~Karpman and Y.~Su.
\newblock Combinatorics of symmetric plabic graphs.
\newblock \emph{Preprint}, 2015.
\newblock arXiv:1510.02122 [math.CO].

\bibitem[Knutson et~al.(2013)Knutson, Lam, and Speyer]{KLS13}
A.~Knutson, T.~Lam, and D.~Speyer.
\newblock Positroid varieties: juggling and geometry.
\newblock \emph{Compositio Mathematica}, 149\penalty0 (10):\penalty0
  1710--1752, 2013.
\newblock arXiv:0903.3694v1 [math.AG].

\bibitem[Knutson et~al.(2014)Knutson, Lam, and Speyer]{KLS14}
A.~Knutson, T.~Lam, and D.~Speyer.
\newblock Projections of {R}ichardson varieties.
\newblock \emph{J. Reine Angew. Math.}, 687:\penalty0 133--157, 2014.
\newblock arXiv:1008.3939v2 [math.AG].

\bibitem[Kottwitz and Rapoport(2000)]{KR99}
R.~Kottwitz and M.~Rapoport.
\newblock Miniscule alcoves for {$GL_{n}$} and {$GSp_{2n}$}.
\newblock \emph{Manuscripta Mathematica}, 102:\penalty0 403--428, 2000.

\bibitem[Lam(2013)]{Lam13}
T.~Lam.
\newblock Notes on the totally nonnegative {G}rassmannian.
\newblock Web, 2013.
\newblock Accessed: 09-26-2014.

\bibitem[Lam and Williams(2008)]{LW07}
T.~Lam and L.~Williams.
\newblock Total positivity for cominuscule {G}rassmannians.
\newblock \emph{New York Journal of Mathematics}, 14:\penalty0 53--99, 2008.
\newblock arXiv:math/0710.2932v1 [math.CO].

\bibitem[Lusztig(1994)]{Lus94}
G.~Lusztig.
\newblock Total positivity in reductive groups.
\newblock In \emph{Lie theory and geometry}, volume 123 of \emph{Progr. Math.},
  pages 531--568. Birkh\"auser Boston, Boston, MA, 1994.

\bibitem[Lusztig(1998)]{Lus98}
G.~Lusztig.
\newblock Total positivity in partial flag varieties.
\newblock \emph{Representation Theory}, 2:\penalty0 70--78, 1998.

\bibitem[Makisumi(2011)]{Ma11}
S.~Makisumi.
\newblock Structure theory of reductive groups through examples.
\newblock Web, 2011.
\newblock Accessed: 07-15-2015.

\bibitem[Marsh and Rietsch(2004)]{MR04}
R.~J. Marsh and K.~Rietsch.
\newblock Parametrizations of flag varieties.
\newblock \emph{Representation Theory}, 8:\penalty0 212--242 (electronic),
  2004.
\newblock arXiv:math/0307017.

\bibitem[Muller and Speyer(2016)]{MS14}
G.~Muller and D.~Speyer.
\newblock The twist for positroid varieties.
\newblock \emph{Preprint}, 2016.
\newblock arXiv:1606.08383 [math.CO].

\bibitem[Oh(2011)]{Oh11}
S.~Oh.
\newblock Positroids and {S}chubert matroids.
\newblock \emph{Journal of Combinatorial Theory, Series A}, 118\penalty0
  (8):\penalty0 2426--2435, November 2011.

\bibitem[Postnikov(2006)]{Pos06}
A.~Postnikov.
\newblock Total positivity, {G}rassmannians and networks.
\newblock \emph{Preprint}, 2006.
\newblock arXiv:math/0609764 [math.CO].

\bibitem[Postnikov et~al.(2009)Postnikov, Speyer, and Williams]{PSW09}
A.~Postnikov, D.~Speyer, and L.~Williams.
\newblock Matching polytopes, toric geometry, and the totally non-negative
  {G}rassmannian.
\newblock \emph{J. Algebraic Combin.}, 30\penalty0 (2):\penalty0 173--191,
  2009.
\newblock arXiv:math/0706.2501v3.

\bibitem[Rietsch(1999)]{Rie99}
K.~Rietsch.
\newblock An algebraic cell decomposition on the nonnegative part of a flag
  variety.
\newblock \emph{Journal of algebra}, 213:\penalty0 144--154, 1999.
\newblock arXiv:alg-geom/9709035.

\bibitem[Rietsch(2006)]{Rie06}
K.~Rietsch.
\newblock Closure relations for totally nonnegative cells in {G}/{P}.
\newblock \emph{Mathematical research letters}, 13:\penalty0 775--786, 2006.
\newblock arXiv:math/0509137v2 [math.AG].

\bibitem[Talaska and Williams(2013)]{TW13}
K.~Talaska and L.~Williams.
\newblock Network parametrizations for the {G}rassmannian.
\newblock \emph{Algebra and Number Theory}, 7\penalty0 (9):\penalty0
  2275--2311, December 2013.
\newblock arXiv:1210.5433v2 [math.CO].

\bibitem[Williams(2007)]{Wil07}
L.~Williams.
\newblock Shelling totally nonnegative flag varieties.
\newblock \emph{J. Reine Angew. Math.}, 609:\penalty0 1--21, 2007.
\newblock arXive:math/0509129v1 [Math.RT].

\end{thebibliography}

\end{document}